\documentclass[12pt]{amsart}
\usepackage{amsmath, amssymb,latexsym }
\usepackage{verbatim}

\usepackage[active]{srcltx}
\usepackage{graphicx}

\newcommand\Ha{\operatorname{Ha}^{(1)}}

\def\text{\textstyle}

\def\text{\textstyle}

 \def\H{\hcal}

\def\max{{\operatorname{max}}}

\newcommand{\szego}{Szeg\"o }

\newcommand{\R}{{\mathbb R}}
\newcommand{\C}{{\mathbb C}}
\newcommand{\Q}{{\mathbb Q}}
\newcommand{\Z}{{\mathbb Z}}

\newcommand{\dbar}{\bar\partial}
\newcommand{\ddbar}{\partial\dbar}

\renewcommand{\H}{{\mathbf H}}
\newcommand{\half}{{\frac{1}{2}}}

\renewcommand{\phi}{\varphi}

\newcommand{\ccal}{\mathcal{C}}
\newcommand{\dcal}{\mathcal{D}}
\newcommand{\ecal}{\mathcal{E}}

\newcommand{\hcal}{\mathcal{H}}
\newcommand{\ical}{\mathcal{I}}

\newcommand{\lcal}{\mathcal{L}}
\newcommand{\mcal}{\mathcal{M}}
\newcommand{\ncal}{\mathcal{N}}
\newcommand{\pcal}{\mathcal{P}}

\newcommand{\ocal}{\mathcal{O}}

\newcommand{\tcal}{\mathcal{T}}

\newcommand{\ucal}{\mathcal{U}}

\newcommand{\vcal}{\mathcal{V}}

\newcommand{\la}{\lambda}

\def    \half   {{\frac{1}{2}}}

\def    \t  {{\mathfrak t}}

\def    \Z  {{\mathbb Z}}
\def    \R  {{\mathbb R}}
\def    \C  {{\mathbb C}}

 \def   \half   {{\frac{1}{2}}}

\def\text{\textstyle}

\newcommand{\sgn}{{\rm sgn~}}

\def\sq2{\sqrt{2}}

\def\t2{{\mathbb T}^2}
\def\s2{{\mathbb S}^2}

\def\R{\mathbb{R}}
\def\Z{\mathbb{Z}}
\def\C{\mathbb{C}}

\def\sq2{\sqrt{2}}

\def\t2{{\mathbb T}^2}
\def\s2{{\mathbb S}^2}

\def\R{\mathbb{R}}
\def\Z{\mathbb{Z}}
\def\C{\mathbb{C}}

\def\hto0{\xrightarrow{\hbar\to 0}}

\def\rto0{\xrightarrow{r\to 0}}

\renewcommand{\le}{\leqslant}
\renewcommand{\ge}{\geqslant}

\newcommand{\Ss}{{\mathbb S}}
\newcommand{\Hh}{{\mathbb H}}

\renewcommand{\Re}{{\operatorname{Re\,}}}
\renewcommand{\Im}{{\operatorname{Im\,}}}
\renewcommand{\phi}{\varphi}
\renewcommand{\H}{{\mathbf H}}
\newtheorem{theo}{{\sc Theorem}}[section]
\newtheorem{theorem}{{\sc Theorem}}[section]
\newtheorem{cor}[theo]{{\sc Corollary}}
\newtheorem{corollary}[theo]{{\sc Corollary}}
\newtheorem{conj}[theo]{{\sc Conjecture}}
\newtheorem{lem}[theo]{{\sc Lemma}}

\newtheorem{prop}[theo]{{\sc Proposition}}

\newtheorem{prob}[theo]{{\sc Problem}}

\newenvironment{rem}{\medskip\noindent{\it Remark:\/} }{\medskip}

\newenvironment{defin}{\medskip\noindent{\it Definition:\/} }{\medskip}

\newtheorem{maintheo}{{\sc Theorem}}

\newtheorem{mainprop}[maintheo]{{\sc Proposition}}

\newtheorem{mainex}{\sc Exercise}
\newtheorem{mainprob}{\sc Problem}
\author{Steve Zelditch}
\address{Department of Mathematics, Northwestern  University, Evanston, IL 60208, USA}

\email{zelditch@math.northwestern.edu}

\thanks{Research partially supported by NSF grant DMS-1206527  .}

 \title[Park City Lectures on Eigenfunctions ]{Park City lectures on Eigenfuntions}

\begin{document}

\maketitle

\tableofcontents

\section{Introduction}

These lectures are devoted to recent results on the nodal geometry of eigenfunctions \begin{equation} 
\label{EIGPROB} 
\Delta_g \;\phi_{\lambda}  = \lambda^2 \; \phi_{\lambda} \end{equation} of the Laplacian $\Delta_g$ of
a Riemannian manifold $(M^m, g)$ of dimension $m$ and to  associated problems on $L^p$ norms of eigenfunctions
(\S \ref{Lpintro} and \S \ref{Lp}). 
The manifolds are generally assumed to be compact, although the problems can also be posed on  non-compact
complete Riemannian manifolds. The emphasis of these lectures  is on  {\it real analytic} Riemannian manifolds, but
we also mention some new results for general $C^{\infty}$ metrics.
Although we mainly discuss the Laplacian,  analogous problems and results
exist  for Schr\"odinger operators
$- \frac{\hbar^2}{2} \Delta_g + V$ for certain potentials $V$.  Moreover, many of the results on eigenfunctions
also hold for {\it quasi-modes} or approximate eigenfunctions defined by oscillatory integrals. 

The study of eigenfunctions of $\Delta_g$ and $- \frac{\hbar^2}{2} \Delta_g + V$ on Riemannian
manifolds  is a branch of harmonic
analysis. In these lectures, we  emphasize high frequency (or semi-classical) asymptotics of eigenfunctions and their relations
to the global dynamics of the geodesic flow $G^t: S^* M \to S^* M$ on the unit cosphere bundle of $M$. 
 Here and henceforth we identity vectors and covectors
using the metric.  As in \cite{Ze3} we give the name ``Global Harmonic Analysis"
to the use of global wave equation methods to obtain relations between eigenfunction behavior and geodesics.
The
relations between geodesics and eigenfunctions belongs to the general correspondence principle between
classical and quantum mechanics. The correspondence principle has evolved since the origins of 
quantum mechanics \cite{Sch2} into a systematic theory
of Semi-Classical Analysis and Fourier integral operators, of which \cite{HoI,HoII,HoIII,HoIV} and \cite{Zw}
give systematic presentations; see also \S \ref{FORMAT} for further references. Quantum mechanics provides
not only the intuition and techniques for the study of eigenfunctions, but in large part also provides the motivation.
Readers who are unfamiliar with quantum mechanics are encouraged to read standard
texts such as Landau-Lifschitz \cite{LL} or Weinberg \cite{Wei}. Atoms and molecules are multi-dimensional
and difficult to visualize, and there are many efforts to do so in the physics and chemistry literature. Some
examples
 may be found in  \cite{KP,He,Th,SHM}. Nodal sets of the hydrogen atom have recently
been observed using   quantum
microscopes  \cite{St}.  Here we concentrate on eigenfunctions of the Laplacian; some results on nodal sets
of eigenfunctions of Schr\"odinger operators can be found in \cite{Jin,HZZ}.

\bigskip

Among the fundamental tools  in  the study of eigenfunctions are  parametrix constructions of wave and Poisson kernels,
and the method of stationary phase in the real and complex domain. The plan of these lectures is to concentrate at once 
on applications to nodal sets and $L^p$ norms and refer to Appendices or other texts for the  techniques. Parametrix
constructions and stationary phase are techniques whose role in spectral asymptotics are as basic as the maximum
principle is in elliptic PDE.  We do
include Appendices on the geodesic flow (\S \ref{GEOFLOWAPP}), on  parametrix constructions for the  wave group (\S \ref{WAVEAPP}), on general facts
and definitions concerning oscillatory integrals (\S \ref{LAGAPP})  and on spherical harmonics (\S \ref{SHAPP}).

\subsection{The eigenvalue problem on a compact Riemannian manifold}

The (negative)
  Laplacian $\Delta_g$  of $(M^m, g)$ is the unbounded essentially self-adjoint operator on $C_0^{\infty}(M) \subset L^2(M, dV_g)$ defined by the Dirichlet form
$$D(f) = \int_M |\nabla f|^2 d V_g,$$
where $\nabla f$ is the metric gradient vector field and $|\nabla f|$ is its length in the metric $g$. Also, $dV_g$ is the volume
form of the metric.  In terms of the metric Hessian $D d$, 
$$\Delta f = \mbox{trace} \; D d f. $$ In local coordinates, 
\begin{equation} \label{DELTADEF}  \Delta_g =   \frac{1}{\sqrt{g}}\sum_{i,j=1}^n
\frac{\partial}{\partial x_i} \left( g^{ij} \sqrt{g}
\frac{\partial}{\partial x_j} \right), \end{equation}
in a standard notation that we assume the reader is familiar with (see e.g. \cite{BGM,Ch} if not).

\begin{rem} We are not always consistent on the sign given to $\Delta_g$. When we work with $\sqrt{- \Delta_g}$
we often define $\Delta_g$ to be the opposite of \eqref{DELTADEF} and write $\sqrt{\Delta}$ for notational simplicity. We  often omit
the subscript $g$ when the metric is fixed. We hope the notational conventions do not cause confusion. \end{rem}

 A more geometric definition uses
at each point $p$ an orthomormal basis $\{e_j\}_{j = 1}^m$ of $T_p M$ and geodesics $\gamma_j$ with $\gamma_j(0) = p$,
$\gamma_j'(0) = e_j$. Then
$$\Delta f (p) = \sum_j \frac{d^2}{dt^2} f(\gamma_j(t)). $$
We refer to \cite{BGM} (G.III.12). 
\bigskip

\begin{mainex} Let $m = 2$ and let $\gamma $ be a geodesic arc on $M$.  Calculate $(\Delta f) (s, 0)$ in Fermi normal coordinates
along $\gamma$. \end{mainex}

Background:  Define Fermi normal coordinates $(s, y)$ along
$\gamma$ by identifying a small ball bundle of the  normal bundle $N \gamma$ along $\gamma(s)$  with its image (a tubular neighborhood
of $\gamma$) under the normal
exponential map, $\exp_{\gamma(s)} y \nu_{\gamma(s)}. $ Here, $\nu_{\gamma(s)}$ is the unit normal at $\gamma(s)$ (fix one of the
two choices) and  $\exp_{\gamma(s)} y \nu_{\gamma(s)}$ is the unit speed geodesic in the direction $\nu_{\gamma(s)}$ of length $y$.
\bigskip

The focus of these lectures is on the   eigenvalue problem  \eqref{EIGPROB}. As mentioned above, we sometimes multiply $\Delta$ and the eigenvalue
by $-1$ for notational simplicity. 
We  assume throughout that $\phi_{\lambda}$ is $L^2$-normalized,
$$||\phi_{\lambda}||_{L^2}^2 = \int_M |\phi_{\lambda}|^2 d V = 1. $$
When $M$ is compact, the spectrum of eigenvalues of the Laplacian is discrete there  exists
an orthonormal basis of eigenfunctions. We fix such a basis $\{\phi_j\}$ so that
  \begin{equation} \label{EIGPROBb} \Delta_g\; \phi_j = \lambda_j^2\; \phi_j, \;\;\;\;\;\; \langle \phi_j, \phi_k \rangle_{L^2(M)} := \int_M \phi_j \phi_k dV_g = \delta_{jk} \end{equation}
 If $\partial M
\not= \emptyset$ we impose Dirichlet or Neumann boundary conditions. Here $dV_g$ is the volume form.  When $M$ is compact, the spectrum of $\Delta_g$
is a discrete set
\begin{equation} \label{LAMBDAS} \lambda_0 = 0  < \lambda_1^2 \leq \lambda_2^2 \leq \cdots  \end{equation}
repeated according to multiplicity.  Note that $\{\lambda_j\}$ denote the {\it frequencies}, i.e. square roots of $\Delta$-eigenvalues. We mainly consider the behavior of eigenfuntions in the `high frequency' (or high energy) limit
$\lambda_j \to \infty$.

The   Weyl law asymptotically counts the number of
eigenvalues less than $\lambda,$
\begin{equation}\label{WL} N(\lambda ) = \#\{j:\lambda _j\leq \lambda \}= \frac{|B_n|}{(2\pi)^n} Vol(M, g) \lambda ^n
+O(\lambda ^{n-1}). \end{equation} Here, $|B_n|$ is the Euclidean
volume of the unit ball and $Vol(M, g)$ is the volume of $M$ with
respect to the metric $g$.  The
size of the remainder reflects the measure of closed geodesics \cite{DG,HoIV}. It is a basic example of global
the effect of the global dynamics on the spectrum. See \S \ref{Lpintro} and \S \ref{Lp} for related results
on eigenfunctions.

\begin{enumerate}

\item  In the {\it aperiodic} case where the set of closed geodesics has measure zero, the Duistermaat-Guillemin-Ivrii two term Weyl law
states
$$N(\lambda ) = \#\{j:\lambda _j\leq \lambda \}=c_m \;
Vol(M, g) \; \lambda^m +o(\lambda ^{m-1})$$
 where $m=\dim M$ and where $c_m$ is a universal constant.

\item  In the {\it periodic} case where the geodesic flow is periodic (Zoll manifolds such as the round sphere),
 the spectrum of $\sqrt{\Delta}$ is a union of eigenvalue clusters $C_N$ of the form
$$C_N=\{(\frac{2\pi}{T})(N+\frac{\beta}{4}) +
 \mu_{Ni}, \; i=1\dots d_N\}$$
with $\mu_{Ni} = 0(N^{-1})$.   The number $d_N$ of eigenvalues in
$C_N$ is a polynomial of degree $m-1$.
\end{enumerate}

\begin{rem} The proof that the spectrum is discrete is based on the study of spectral kernels such as the
heat kernel or Green's function or wave kernel. The  standard proof is to show that $\Delta_g^{-1}$ (whose
kernel is the Green's function,
defined on the orthogonal complement of the constant functions) is a compact self-adjoint operator. By
the spectral theory for such operators, the eigenvalues of  $\Delta_g^{-1}$ are discrete, of finite multiplicity,
and only accumulate at $0$.  Although we concentrate on parametrix constructions for the wave kernel,
one can construct the Hadamard parametrix for the Green's function in a similar way. Proofs of the above
statements can be found in \cite{GSj,DSj,Zw,HoIII}. 

The proof of the integrated and pointwise Weyl law are based on  wave equation techniques and
Fourier Tauberian theorems. The wave equation techniques mainly involve the construction of parametrices
for the fundamental solution of the wave equation and the method of stationary phase. 
In \S \ref{LAGAPP} we review We refer to \cite{DG,HoIV} for detailed background. \end{rem} 

\subsection{Nodal and critical point sets}

The main focus of these lectures is on  nodal hypersurfaces
\begin{equation} \label{RNODAL} \ncal_{\phi_{\lambda}} = \{x\in M: \phi_{\lambda}(x)  = 0\}. \end{equation}
The nodal domains are the components  of the complement of the nodal set,
$$M \backslash \ncal_{\phi_{\lambda}}= \bigcup_{j = 1}^{\mu(\phi_{\lambda})}  \Omega_j. $$ For generic
metrics,  $0$ is a regular
value of $\phi_{\lambda}$ of all eigenfunctions, and the nodal sets are smooth non-self-intersecting hypersurfaces \cite{U}.
Among the main problems on nodal sets are to determine the hypersurface volume $\hcal^{m-1}(\ncal_{\phi_{\lambda}})$
and ideally how the nodal sets are  distributed. Another well-known  question is to determine the number $\mu(\phi_{\lambda})$
of nodal domains in terms of the eigenvalue in generic cases.  One may also consider
the other level sets
\begin{equation} \label{RLEVEL} \ncal^a_{\phi_{j}} = \{x \in M: \phi_j (x)  = a\} \end{equation} 
and  sublevel sets
\begin{equation} \label{SUBLEVEL}   \{x \in M: |\phi_j (x)|   \leq a \}. \end{equation}
The zero level is distinguished since the symmetry  $\phi_j \to - \phi_j$ in the equation
preserves the nodal set.
\bigskip

\begin{rem} Nodals sets belong to individual eigenfunctions. To the author's knowledge there
do not exist any results on averages of nodal sets over the spectrum in the sense of \eqref{WL}-\eqref{PLWL}. 
That is, we do not know of any asymptotic results concerning the functions
$$Z_f(\lambda): = \sum_{j: \lambda_j \leq \lambda} \int_{\ncal_{\phi_{j}}} f d S, $$
where $\int_{\ncal_{\phi_{j}}} f d S$ denotes the integral of a continuous function $f$ over the nodal
set of $\phi_j$. When the eigenvalues are multiple, the sum $Z_f$ depends on the choice of 
 orthonormal basis.

Randomizing by taking Gaussian random combinations of eigenfunctions simplifies nodal problems profoundly,
and are studied in many articles (see e.g. \cite{NS}).

\end{rem}
\bigskip

One would also like to know the ``number'' and distribution of critical points,
\begin{equation} \label{CRIT} \ccal_{\phi_{j}} = \{ x \in  M : \nabla \phi_j (x)  = 0\}. \end{equation}
In fact, the  critical point set can be a hypersurface in $M$, so for counting problems it makes more sense to count
the number of critical values,
\begin{equation} \label{CRITV} \vcal_{\phi_{j}} = \{ \phi_j(x):  \nabla \phi_j (x)  = 0\} .\end{equation}
At this time of writing, there exist few  rigorous upper bounds on the number of critical
values, so we do not spend much space on them here. If we `randomize' the problem and consider
the average number of critical points (or equivalently values) or random spherical harmonics on the 
standard $\Ss^m$, one finds
that the random spherical harmonics of degree $N$ (eigenvalue $\simeq N^2$)  has $C_m  N^m$ critical
points in dimension $m$. This is not surprising since spherical harmonics are harmonic polynomials.   In the non-generic case  that the critical
manifolds are of co-dimension one, the hypersurface volume is calcualated  in the real analytic case in \cite{Ba}. The 
upper bound is also given in \cite{Ze0}.

The  frequency $\lambda$ of an eigenfunction (i.e. the square root of the eigenvalue)   is a measure of
its  ``complexity", similar to specifying the degree of a polynomial, 
and the high frequency limit is the large complexity limit. A sequence of eigenfunctions of increasing frequency 
oscillates more and more rapidly and the problem is to find its ``limit shape".  Sequences of eigenfunctions
often behave like ``Gaussian random waves" but special ones exhibit highly localized oscillation and
concentration properties. 

\subsection{\label{MOTIVATION} Motivation} Before stating the problems and results, let us motivate the study of eigenfunctions
and their high frequency behavior. The eigenvalue problem \eqref{EIGPROB} arises in many areas of physics, for 
example the theory of vibrating membranes. But renewed motivation to study eigenfunctions comes from
quantum mechanics. As is discussed in any textbook on quantum physics or chemistry (see e.g.
\cite{LL, Wei}), the Schr\"odinger
equation resolves the problem of how an electron  can orbit the nucleus without losing its energy in radiation.
The classical Hamiltonian equations of motion of a particle  in phase space are orbits of Hamilton's equations 
$$\left\{ \begin{array}{l} \frac{dx_j }{dt} = \frac{\partial H}{\partial \xi_j}, \;\; \\ \\
\frac{d \xi_j }{dt} = -  \frac{\partial H}{\partial x_j}, \end{array} \right. $$
where the Hamiltonian 
$$H(x, \xi) = \half |\xi|^2 + V(x) : T^* M\to \R $$ is the total Newtonian  kinetic + potential energy.
The idea of Schr\"odinger is to  model the electron by a wave function $\phi_j$  which solves  the eigenvalue problem 
\begin{equation} \label{EigP} \hat{H} \phi_j : =  (- \frac{\hbar^2}{2} \Delta + V) \phi_{j} = E_j(\hbar) \phi_j,   \end{equation}
for the  Schr\"odinger operator $\hat{H}$, 
where  $V$ is the potential, a multiplication 
operator on $L^2(\R^3). $ Here $\hbar$ is Planck's constant, a very small constant. The semi-classical limit
$\hbar \to 0$ is mathematically equivalent to the high frequency limit when $V = 0$.
 The time
evolution of an `energy state'  is given by 
\begin{equation} \label{Uth} U_{\hbar}(t)  \phi_j := e^{-i \frac{t}{\hbar}  (- \frac{\hbar^2}{2} \Delta + V) } \phi_j =  e^{-i \frac{t E_j(\hbar)}{\hbar} }  \phi_j. \end{equation}  The unitary oprator $ U_{\hbar}(t)  $ is often called
the {\it propagator}. In the Riemannian case with $V = 0$, the factors of $\hbar$ can be absorbed in the $t$
variable and it suffices to study
\begin{equation} U(t) = e^{i t \sqrt{\Delta}}. \end{equation}

An $L^2$-normalized energy state  $\phi_j$ defines   a probability amplitude, i.e. its
 modulus square is a probability measure with  \begin{equation} \label{SQUARE} |\phi_j(x)|^2 dx= \;\; \mbox{ the probability density of finding the particle at $x$ } .\end{equation}  According to the physicists, the   {\it observable quantities} associated
to the energy state are the probability density \eqref{SQUARE} and `more generally'
the   matrix elements 
\begin{equation} \label{ME1} \langle A \phi_j, \phi_j \rangle = \int \phi_j(x) A \phi_j(x) dV  \end{equation} of observables (A is a self adjoint operator, and in these lectures it is assumed to be a pseudo-differential operator). Under the time evolution \eqref{Uth}, the  factors of 
$  e^{-i \frac{t E_j(\hbar)}{\hbar} }$ cancel and so the particle evolves  as if ``stationary", i.e.
observations of the particle are independent of the time $t$.

Modeling energy states by eigenfunctions \eqref{EigP} resolves the paradox \footnote{This is an over-simplified
account of the stability problem; see \cite{LS} for an in-depth account}  of particles which are simultaneously
in motion and are stationary, but at the cost of replacing the classical model of particles following the trajectories
of Hamilton's equations by `linear algebra', i.e. evolution by \eqref{Uth}. The quantum picture is difficult to visualize
or understand intuitively. Moreover, it is difficult to 
relate the classical picture of orbits with the quantum picture of eigenfunctions. 

The study of nodal sets was historically motivated in part by the desire to visualize energy states by 
finding the points where the quantum particle is least likely to be.  In fact, just recently (at this time or writing)
the nodal sets of the hydrogem atom energy states have become visible to microscopes \cite{St}.

\section{Results }

We now introduce the results whose proofs we sketch in the  later sections of this article.

\subsection{\label{LBintro} Nodal hypersurface volumes for $C^{\infty}$ metrics}
  In the late $70's$,
S. T. Yau conjectured that for general $C^{\infty}$  $(M, g)$ of any dimension $m$  there exist $c_1, C_2$ depending only on $g$ so that
\begin{equation} \label{YAU}  \lambda \lesssim  {\mathcal
H}^{m-1}(\ncal_{\phi_{\lambda}}) \lesssim \lambda. 
 \end{equation}
Here  and below $\lesssim$ means that there exists a constant $C$ independent of $\lambda$ for which
the inequality holds.
The upper   bound of \eqref{YAU} is the  analogue  for eigenfunctions of the fact that 
the  hypersurface volume of a real
algebraic variety is bounded above  by its degree. The lower bound is specific to eigenfunctions.  It is a strong
version of the statement that $0$ is not an ``exceptional value" of $\phi_{\lambda}$.  Indeed, a basic result
is  the following classical result, apparently due to   R. Courant
  (see  \cite{Br})). It  is used to obtain lower bounds
on volumes of nodal sets:

\begin{mainprop}\label{SMALLBALL}  For any $(M, g)$  there exists a constant $A > 0$ so that every ball of $(M, g)$ od  radius  greater than $\frac{A}{\lambda}$ contains a nodal point of any eigenfunction  $\phi_{\lambda}$. \end{mainprop}

We sketch the proof in \S \ref{PFSMALLBALL} for completeness, but leave some of the proof as an exercise to
the reader.

The lower bound of \eqref{YAU}  was proved for all $C^{\infty}$ metrics for
surfaces, i.e. for $m = 2$ by Br\"uning \cite{Br}.  For general
$C^{\infty}$ metrics in dimensions $\geq 3$, the known upper and lower bounds are far from the conjecture
\eqref{YAU}. At present the best lower bound available for general $C^{\infty}$ metrics of all dimensions is
the following estimate of Colding-Minicozzi \cite{CM}; a somewhat weaker bound was proved by Sogge-Zelditch \cite{SoZ}
and the later simplification of the  proof \cite{SoZa}  turned out to give  the same bound as \cite{CM}. We sketch
the proof from \cite{SoZa}.

\begin{maintheo}\label{CMSZ} 

$$\la^{1-\frac{n-1}2}\lesssim \hcal^{m-1} \ncal_\la),$$
\end{maintheo}

The original result of \cite{SoZ} is based on lower bounds on the $L^1$ norm of eigenfunctions. 
Further work of Hezari-Sogge \cite{HS} shows that the Yau lower bound is correct when one has 
 $||\phi_{\lambda}||_{L^1} \geq C_0 $ for some $C_0 > 0$. It is not known for which $(M, g)$
such an estimate is valid.  At the present time, such lower bounds
are obtained from upper bounds on the $L^4$ norm of $\phi_{\lambda}$.   The study of $L^p$ norms
of eigenfunctions is of independent interest and we  discuss some recent results which are not
directly related to nodal sets  In \S \ref{Lp} and in \S \ref{Lpintro}.  The study of $L^p$ norms splits
into two very different cases: there exists a critical index $p_n$ depending on the dimension of $M$,
and for $p \geq p_n$ the $L^p$ norms of eigenfunctions are closely related to the structure of geodesic
loops (see \S \ref{Lpintro}). For $2 \leq p \leq p_n$ the $L^p$ norms are governed by different geodesic
properties of $(M, g)$ which we discuss in \S \ref{Lp}.

We also recall in \S \ref{DUB} an interesting upper bound due to R. T. Dong (and Donnelly-Fefferman) in dimension 2, since
the techniques of proof of \cite{Dong} seem capable of further development.

\begin{maintheo} \label{RTDONGth}  For $C^{\infty}$ $(M, g)$ of dimension 2, 
$$\hcal^1(\ncal_\la) \lesssim \lambda^{3/2}. $$
\end{maintheo}

\subsection{Nodal hypersurface volumes for real analytic $(M, g)$}

In 1988,   Donnelly-Fefferman \cite{DF}  proved the conjectured bounds for real analytic Riemannian
manifolds (possibly with boundary). We re-state the result as the following

\begin{maintheo}   \label{DF}   

Let $(M, g)$ be a compact real analytic  Riemannian
manifold, with or without boundary. Then 
$$c_1 \lambda \lesssim {\mathcal
H}^{m-1}(Z_{\phi_{\lambda}}) \lesssim \lambda.
$$
\end{maintheo}

See also \cite{Ba} for a similar proof that the $\hcal^{m-1}$ measure of the critical set is
$\simeq \lambda$ in the real analytic case. 

\subsection{\label{INTERINTRO} Number of intersections of nodal sets with geodesics and number of  nodal domains}

The Courant nodal domain theorem (see e.g. \cite{Ch,Ch3}) asserts that the nth eigenfunction $\phi_{\lambda_n}$
has $\leq n$ nodal domains. This estimate is not sharp (see \cite{Po} and \S \ref{PL} for recent results) and
it is possible to find sequences of eigenfunctions with $\lambda_n \to \infty$ and with a bounded number of 
nodal domains \cite{L}. In fact, it has been pointed out \cite{Hof} that we do not even know if a given $(M, g)$
possesses any sequence of eignefunctions $\phi_n$  with $\lambda_n \to \infty$ for which the number of 
nodal domains tends to infinity. A nodal domain always contains a local minimum or maximum, so a necessary
condition that the number of nodal domains increases to infinity along  sequence of eigenfunctions is that
the number of their critical points (or values) also increases to infinity; see \cite{JN} for a sequence in which
the number of critical points is uniformly bounded.

Some new results give lower bounds on the number of nodal domains on surfaces with an orientation
reversing isometry with non-empty fixed point set. The first result is due to Ghosh-Reznikov-Sarnak \cite{GRS}.

\begin{theorem} \label{grz}

Let $\phi$ be an even Maass-Hecke $L^2$ eigenfunction on $\mathbb{X}=SL(2,\mathbb{Z})\backslash \mathbb{H}$. Denote
the  nodal domains which intersect a compact geodesic segment $\beta \subset \delta=\{iy~|~y>0\}$ by $N^\beta(\phi)$.
Assume $\beta$ is sufficiently long and assume the Lindelof Hypothesis for the Maass-Hecke $L$-functions. Then
\[
N^\beta(\phi) \gg_\epsilon \lambda_\phi^{\frac{1}{24}-\epsilon}.
\]
\end{theorem}

We refer to \cite{GRS} for background definitions. The strategy of the proof is to first prove that there are many
intersections of the nodal set with the vertical geodesic of the modular domain and that the eigenfunction changes
sign at many intersections. It follows that the nodal lines intersect the geodesic orthogonally. Using a topological
argument, the authors show that the nodal lines must often close up to bound nodal domains. 

As is seen from this outline, the main analytic ingredient is to prove that there are many intersections of the nodal
line with the geodesic. The study of such intersections in a more general context is closely related to a new
series of quantum ergodicity results known as QER (quantum ergodic restriction) theorems \cite{TZ,TZ2,CTZ}. 
We discuss these ingredients below. 

In recent work, the author and J. Jung have proved a general kind result in the same direction. 
The setting is that of a Riemann surface  $(M, J, \sigma)$ with an orienting-reversing involution $\sigma$ whose
fixed point set 
$\mbox{Fix}(\sigma)$ is separating, i.e. $M \backslash \mbox{Fix}(\sigma)$ consists of two connected components. 
The result is  that for any $\sigma$-invariant negatively curved metric, and  for almost the entire sequence of
even or odd eigenfunctions, the number of nodal domains tends to infinity.  In fact, the argument only uses ergodicity
of the geodesic flow.   

.
We first explain the hypothesis. When a Riemann surface possesses an  orientation-reversing involution
$\sigma: M \to M$, Harnack's theorem says that the fixed point set $\mbox{Fix}(\sigma)$
is a disjoint union
\begin{equation} \label{H} H = \gamma_1 \cup \cdots \cup \gamma_k \end{equation} of $0 \leq k \leq g + 1$ of simple closed
curves. 
It is possible that $\mbox{Fix}(\sigma) = \emptyset$ but we  assume $k \not= 0$. 
We also assume that  $H$ \eqref{H} is a separating set, i.e.  $M \backslash H = M_+ \cup M_-$ where $M_+^0 \cap M_-^0 = \emptyset$ (the interiors are disjoint),
where  $\sigma(M_+) = M_-$ and where $\partial M_+ = \partial M_- = H$.

If $\sigma^* g = g$ where $g$ is a negatively curved metric, then 
 $\mbox{Fix}(\sigma)$ is a finite union of simple closed geodesics. We denote by  $L_{even}^2(M)$  the set of $f \in L^2(M)$ such that $\sigma f=f$ and by $L_{odd}^2(Y)$ the $f$ such that $\sigma f = - f$. We denote by $\{\phi_j\}$ an orthonormal eigenbasis of Laplace eigenfunctions of $L_{even}^2(M)$, resp. $\{\psi_j\}$ for  $L^2_{odd}(M)$.

We further denote by
\[
\Sigma_{\phi_{\lambda}} = \{x \in \ncal_{\phi_{\lambda}}: d \phi_{\lambda}(x) = 0\}
\]
the  singular set of $\phi_{\lambda}.$ These are special critical points $d \phi_j(x) = 0$ which lie on the nodal set $Z_{\phi_j}$.

\begin{theo}\label{theoJJ}
Let $(M,g)$ be a compact negatively curved $C^\infty$ surface with an orientation-reversing  isometric involution $\sigma : M \to M$ with
$\mbox{Fix}(\sigma)$ separating.   Then for any orthonormal eigenbasis $\{\phi_j\}$ of $L_{even}^2(Y)$, resp. $\{\psi_j\}$ of $L_{odd}^2(M)$, one can find a density $1$ subset $A$ of $\mathbb{N}$ such that
\[
\lim_{\substack{j \to \infty \\ j \in A}}N(\phi_j) = \infty,
\]
resp.
\[
\lim_{\substack{j \to \infty \\ j \in A}}N(\psi_j) = \infty,
\]

\end{theo}
For odd eigenfunctions, the  conclusion holds as long  as $Fix(\sigma)\neq \emptyset$. 
A density one subset $A \subset {\bf N}$ is one for which
$$\frac{1}{N} \#\{j \in A, j \leq N \} \to 1, \;\; N \to \infty. $$

As the image  indicates, the surfaces in question are complexifications of real algebraic curves, with $Fix(\sigma)$
the underlying real curve. 
\begin{center}
\includegraphics[scale=0.8]{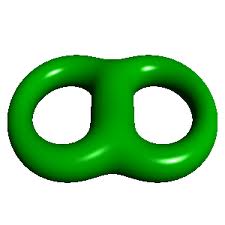}
\end{center}

The strategy of the proof is similar to that of Theorem \ref{grz}: we prove that for even or odd eigenfunctions,
the nodal sets intersect $\mbox{Fix}(\sigma)$ in many points with sign changes, and then use a topological 
argument to conclude that there are many nodal domains. We note that the study of sign changes has been
used in \cite{NPS} to study nodal sets on surfaces in a different way. 


In \S \ref{INTREALBDYintro} we also discuss upper bounds on nodal intersections in the real analytic case. 

\subsection{Dynamics of the geodesic or billiard flow}

Theorem \ref{theoJJ} used the hypothesis of ergodicity of the geodesic flow.  It is not obvious that nodal
sets of eigenfunctions should bear any relation to geodesics, but one of our central themes is that in some
ways they do. 

In general,
there are two broad classes of results on nodal sets and other properties of eigenfunctions:
\bigskip

\begin{itemize}

\item Local results which are valid for any local solution of \eqref{EIGPROB}, and which often 
use local arguments. For instance the proof of Proposition \ref{SMALLBALL} is local. 
\bigskip

\item Global results which use that eigenfunctions are global solutions of \eqref{EIGPROB}, or
that they satisfy boundary conditions when $\partial M \not= \emptyset$. Thus, they are 
also satisfy the unitary evolution equation \eqref{Uth}.  For instance the relation between closed
geodesics and the remainder term of Weyl's law is global  \eqref{WL}-\eqref{PLWL}. 

\end{itemize}
\bigskip

Global results often exploit the relation between classical and quantum mechanics, i.e. the
relation between the eigenvalue problem \eqref{EIGPROB}-\eqref{Uth} and the geodesic flow. 
Thus  the results often depend on the dynamical properties of the geodesic flow. 
The relations between eigenfunctions and the Hamiltonian flow are best
established in two extreme cases: (i) where the Hamiltonian flow is completely integrable 
on an energy surface, or (ii) where it is ergodic.  The extremes
are illustrated below in the case of (i) billiards on rotationally invariant annulus, (ii)
chaotic billiards on a cardioid. 

\begin{center}
\includegraphics[scale=0.5]{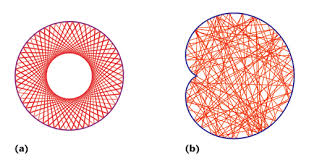} \vspace{1cm}
\end{center}

A random trajectory in the case of ergodic billiards is uniformly distributed, while
all trajectories are quasi-periodic in the integrable case.

We do not have the space to review the dynamics of geodesic flows or other Hamiltonian flows. 
We refer to \cite{HK} for background in dynamics and to \cite{Ze,Ze3, Zw} for relations between
dynamics of geodesic flows and eigenfunctions.

We  use the following basic construction: given a measure preserving map (or flow) $\Phi : (X, \mu) \to (T, \mu)$
one can consider the translation operator \begin{equation} \label{UPHI} U_{\Phi} f(x) = \Phi^* f(x) = f(\Phi(x)),
\end{equation} sometimes called
the Koopman operator or Perron-Frobenius operator  (cf. \cite{RS,HK}). It is a unitary
operator on $L^2(X, \mu)$ and hence its spectrum lies on the unit circle. $\Phi$ is ergodic if and only
if $U_{\Phi}$ has the eigenvalue 1 with multiplicity 1, corresponding to the constant functions. 

The geodesic (or billiard) flow is the Hamiltonian flow on $T^* M$ generated by the metric norm Hamiltonian
or its square,
\begin{equation} H(x, \xi) = |\xi|_g^2 = \sum_{i, j} g^{ij} \xi_i \xi_j. \end{equation}
In PDE one most often uses the $\sqrt{H}$ which is homogeneous of degree 1. The geodesic flow is
ergodic when the Hamiltonian flow $\Phi^t$ is ergodic on the level set $S^*M = \{H = 1\}$.


\subsection{\label{QERINTRO} Quantum ergodic restriction theorems and nodal intersections}

One of the main themes of these lectures is that ergodicity of the geodesic flow causes eigenfunctions
to oscillate rapidly everywhere and in all directions, and hence to have a `maximal' zero set. We will see 
this occur both in the real and complex domain. In the real domain (i.e. on $M$), ergodicity ensures
that restrictions of eigenfunctions in the two-dimensional case to geodesics have many zeros along the geodesic.
In \S \ref{QERsect} we will show that such oscillations and zeros are due to the fact that under generic
assumptions,  restrictions of
eigenfunctions to geodesics are `quantum ergodic' along the geodesic. Roughly this means that they have uniform
oscillations at all frequencies below the frequency of the eigenfunction.

We prove that this QER (quantum ergodic restriction) property has the following implications. First, 
any arc $\beta \subset H$,
\begin{equation} \label{LOG} |\int_{\beta} \phi_{\lambda_j} ds| \leq C \lambda_j^{-1/2} (\log \lambda)^{\half} \end{equation}
and
\begin{equation} \label{qer} \int_{\beta} |\phi_{\lambda_j} | ds \geq ||\phi_{\lambda_j} ||_{\infty}^{-1} ||\phi_{\lambda_j} ||_{L^2(\beta)}^2
\geq C \lambda_j^{-1/2} \log \lambda_j.  \end{equation} 
The first inequality is generic while the second uses the QER property.  The inequalities 
are inconsistent if  $\phi_{\lambda_j}  \geq 0$ on $\beta $, and that shows that eigenfunctions
have many sign changes along the geodesic.   The estimate also the well-known sup norm bound
\begin{equation} \label{NEGSUP} ||\phi_{\lambda} ||_{\infty}\leq \lambda^{1/4} /\log \lambda \end{equation}
for eigenfunctions on negatively curved surfaces \cite{Be}.

The same argument shows that the number
of singular points of odd eigenfunctions tends to infinity and one can adapt it to prove that the number
of critical points of even eigefunctions (on the geodesic)  tend to infinity. 

\subsection{Complexification of $M$ and Grauert tubes}

The next series of results concerns `complex nodal sets', i.e. complex zeros of analytic continuations
of eigenfunctions to the complexification of $M$. It  is difficult to draw conclusions about real
nodal sets from knowledge of their complexifications. But complex nodal sets are simpler to study than
real nodal sets and the results are stronger, just as  complex  algebraic varieties behave in simpler ways
than real algebraic varieties.

A real analytic manifold $M$ always possesses a unique
complexification $M_{\C}$ generalizing the complexification of
$\R^m$ as $\C^m$. The complexification is  an open  complex
manifold in  which $M$ embeds $\iota: M \to M_{\C}$  as a totally
real submanifold (Bruhat-Whitney)

The Riemannian metric determines a special kind of distance
function on $M_{\C}$  known as a Grauert tube function. In fact,  it is observed in \cite{GS1} that the Grauert
tube function is obtained from the distance function by setting $\sqrt{\rho}(\zeta) =
i \sqrt{ r^2(\zeta, \bar{\zeta})}$ where $r^2(x, y)$ is the
squared distance function in a neighborhood of the diagonal in $M
\times M$.

One defines the Grauert tubes $M_{\tau} = \{\zeta \in M_{\C}:
\sqrt{\rho}(\zeta) \leq \tau\}$. There exists a maximal $\tau_0$
for which $\sqrt{\rho}$ is well defined, known as the Grauert tube
radius. For $\tau \leq \tau_0$, $M_{\tau}$ is a strictly
pseudo-convex domain in $M_{\C}$.
Since $(M, g)$ is real analytic, the exponential map $\exp_x t \xi$ admits an analytic
continuation in $t$ and the imaginary time exponential map
\begin{equation} \label{EXP} E: B_{\epsilon}^* M \to M_{\C}, \;\;\; E(x, \xi) = \exp_x i \xi \end{equation} is, for small enough $\epsilon$, a
diffeomorphism from the ball bundle $B^*_{\epsilon} M$ of radius
$\epsilon $ in $T^*M$ to the Grauert tube $M_{\epsilon}$ in
$M_{\C}$. We have $E^* \omega = \omega_{T^*M}$ where $\omega = i
\ddbar \rho$ and where $\omega_{T^*M}$ is the canonical symplectic
form; and also $E^* \sqrt{\rho} = |\xi|$ \cite{GS1,LS1}. It
follows that $E^*$ conjugates the geodesic flow on $B^*M$ to the
Hamiltonian flow $\exp t H_{\sqrt{\rho}}$ of $\sqrt{\rho}$ with
respect to $\omega$, i.e. \begin{equation} \label{CONJUG} E(g^t(x,
\xi)) = \exp t \Xi_{\sqrt{\rho}} (\exp_x i \xi). \end{equation}
 In general $E$ only
extends as a diffemorphism  to a certain maximal radius $\epsilon_{\max}$. We assume throughout that
$\epsilon < \epsilon_{\max}$. 

\subsection{Equidistribution of nodal sets in the complex domain}

One may also consider the complex nodal sets
\begin{equation} \ncal_{\phi_{j}^{\C}} = \{ \zeta \in M_{\epsilon}: \phi_j^{\C}(\zeta) = 0\}, \end{equation}
and the complex critical point sets
\begin{equation}  \ccal_{\phi_{j}^{\C}} = \{ \zeta \in M_{\epsilon}: \partial \phi_j^{\C}(\zeta) = 0\}. \end{equation}

The following is proved in \cite{Ze5}:

\begin{maintheo}  \label{ERGCXZ} Assume $(M, g)$ is real analytic and that  the geodesic flow of $(M, g)$ is ergodic.
 Then for all but a sparse subsequence of
$\lambda_j$,
$$\frac{1}{\lambda_j} \int_{ \ncal_{\phi_{\lambda_j}^{\C}}} f \omega_g^{n - 1}  \to \frac{i}{\pi} \int_{M_{\epsilon}} f \overline{\partial} {\partial}
\sqrt{\rho} \wedge \omega_g^{n-1}.
$$
\end{maintheo}

The proof is based on quantum ergodicity of analytic continuation of eigenfunctions to Grauert tubes
and the growth estimates ergodic eigenfunctions satisfy.

We will say that a sequence $\{\phi_{j_k}\}$ of $L^2$-normalized
eigenfunctions  is {\it quantum ergodic} if
\begin{equation} \label{QEDEF} \langle A \phi_{j_k}, \phi_{j_k} \rangle \to
\frac{1}{\mu(S^*M)} \int_{S^*M} \sigma_A d\mu,\;\;\; \forall A \in
\Psi^0(M). \end{equation} Here, $\Psi^s(M)$ denotes the space of
pseudodifferential operators of order $s$, and $d \mu$ denotes
Liouville measure on the unit cosphere bundle $S^*M$ of $(M, g)$.
More generally, we denote by $d \mu_{r}$ the (surface) Liouville
measure on $\partial B^*_{r} M$, defined by
\begin{equation} \label{LIOUVILLE} d \mu_r = \frac{\omega^m}{d |\xi|_g} \;\; \mbox{on}\;\; \partial B^*_r
M.
\end{equation}
We also denote by $\alpha$ the canonical action $1$-form of
$T^*M$.

\subsection{\label{INTREALBDYintro} Intersection of nodal sets and real analytic curves on surfaces}

 In recent work, intersections of nodal
sets and curves on surfaces $M^2$  have been used in a variety of articles to obtain upper and lower bounds on nodal points
and domains. The work often is based on restriction theorems for eigenfunctions. Some of the 
recent articles on restriction theorems and/or nodal intersections  are  \cite{TZ,TZ2,GRS,JJ,JJ2,Mar,Yo,Po}.
\bigskip

First we consider a basic upper bound on the number of intersection points:

\begin{maintheo} \label{INTREALBDYint} Let $\Omega \subset \R^2$ be a piecewise analytic domain
and  let  $n_{\partial \Omega}(\lambda_j)$ be the number of
components of the nodal set of the $j$th Neumann or Dirichlet
eigenfunction which intersect $\partial \Omega$.  Then there
exists  $C_{\Omega}$   such that $n_{\partial \Omega}(\lambda_j)
\leq C_{\Omega} \lambda_j.$
\end{maintheo}

In the Dirichlet case, we delete the boundary when considering components of the nodal set. 

The
method of proof of Theorem \ref{INTREALBDYint} generalizes from
$\partial \Omega$ to a rather large class of  real analytic curves
$C \subset \Omega$, even when $\partial \Omega$ is not real
analytic. Let us call a real analytic curve $C$ a {\it good} curve
if there exists a constant $a > 0$ so that  for all $\lambda_j$
sufficiently large,
 \begin{equation}
\label{GOOD} \frac{\|\phi_{\lambda_j}\|_{L^{2} (\partial
\Omega)}}{\|\phi_{\lambda_j}\|_{L^{2}(C)}}
     \leq e^{a \lambda_j} .\end{equation}
     Here, the $L^2$ norms refer to the restrictions of the
     eigenfunction to $C$ and to $\partial \Omega$.
The following result deals with the case where $C \subset
\partial \Omega$ is an {\em interior} real-analytic
   curve. The  real curve $C$ may then
be holomorphically continued to a complex curve  $C_{\C} \subset
\C^2$ obtained by analytically continuing a real analytic
parametrization of $C$.

   \begin{maintheo}\label{GOODTH}
 Suppose that $\Omega \subset \R^2$ is a $C^{\infty}$ plane
domain, and let $C \subset \Omega$ be a good  interior real
analytic curve in the sense of (\ref{GOOD}). Let
  $n(\lambda_j, C) = \# Z_{\phi_{\lambda_j} } \cap C $ be the number of intersection points of
  the nodal set of the $j$-th Neumann (or Dirichlet) eigenfunction  with $C$.  Then there exists
  $A_{C, \Omega} > 0$ depending only on $C, \Omega$ such that  $n(\lambda_j, C)
  \leq A_{C, \Omega}
 \lambda_j$.

 \end{maintheo}

The proof of Theorem \ref{GOODTH} is somewhat simpler than that of Theorem \ref{INTREALBDYint}, i.e.
good interior analytic curves are somewhat simpler than the boundary itself. On the other hand, it is
clear that the boundary is good and it is hard to prove that other curves are good.
A recent paper of J. Jung shows that many natural curves in the hyperbolic plane are `good' \cite{JJ}.
See also \cite{ElHajT} for general results on good curves.

The upper bounds of Theorem \ref{INTREALBDYint} - \ref{GOODTH}  are proved by analytically continuing the restricted eigenfunction to the analytic
continuation of the curve. We then give a similar upper bound on complex zeros. Since real zeros are also complex zeros, 
we then get an upper bound on complex zeros. An obvious question is whether the order of magnitude estimate
is sharp. Simple examples in the  unit disc show that there are no non-trivial lower bounds on numbers of intersection points.
But
when the dynamics is ergodic we expect to  prove an equi-distribution theorem for nodal intersection points
(in progress). Ergodicity
once again implies that eigenfunctions oscillate as much as possible and therefore saturate bounds on zeros.

 Let $\gamma \subset
M^2$ be a generic geodesic arc on a real analytic Riemannian surface.
 For small $\epsilon$, the parametrization of $\gamma$ may be analytically
continued to a strip,   $$\gamma_{\C}:  S_{\tau}: = \{ t + i \tau \in \C:
|\tau| \leq \epsilon\} \to M_{\tau}.  $$ Then the eigenfunction restricted to $\gamma$ is 
$$\gamma_{\C}^* \phi_{j}^{\C} (t + i \tau) = \phi_{j}(\gamma_{\C}(t + i \tau)\;\; \mbox{on}\;\; S_{\tau}. $$

 Let
\begin{equation} \ncal^{\gamma}_{\lambda_j} : = \{(t + i \tau:
\gamma_{H}^* \phi_{\lambda_j}^{\C} (t + i \tau)  = 0\}
\end{equation}
be the complex zero set of this holomorphic function of one
complex variable.  Its zeros are the intersection points. 

Then as a current of integration,
\begin{equation}\label{PLLa}  [\ncal^{\gamma}_{\lambda_j}] = i \ddbar_{t + i
\tau} \log \left| \gamma^* \phi_{\lambda_j}^{\C} (t + i \tau)
\right|^2. \end{equation}

The following result is proved in \cite{Ze6}:

\begin{maintheo} \label{NODINT}   Let $(M, g)$ be real analytic with ergodic
geodesic flow. Then for generic $\gamma$  there exists a subsequence of eigenvalues
$\lambda_{j_k}$ of density one such that
$$\frac{i}{\pi \lambda_{j_k}} \ddbar_{t + i
\tau} \log \left| \gamma^* \phi_{\lambda_{j_k}}^{\C} (t + i \tau)
\right|^2 \to  \delta_{\tau = 0} ds.  $$
\end{maintheo}

Thus, intersections of (complexified) nodal sets and geodesics
concentrate in the real domain--  and are distributed by
arc-length measure on the real geodesic. \bigskip

The key point is that 
$$\frac{1}{\lambda_{j_k}}  \log | \phi_{\lambda_{j_k}}^{\C} (\gamma(t + i \tau)|^2 \to |\tau|. $$
Thus, the maximal growth occurs along individual  (generic) geodesics.

\subsection{\label{Lpintro} $L^p$ norms of eigenfunctions}

In \S \ref{LBintro} we mentioned that lower bounds on $\hcal^{n-1}(\ncal_{\phi_{\lambda}}$ are
related to lower bounds on $||\phi_{\lambda}||_{L^1}$ and to upper bounds on 
 $||\phi_{\lambda}||_{L^p}$ for certain $p$.  Such $L^p$ bounds are interesting for all $p$ and
depend on the shapes of the eigenfunctions. 

In \eqref{WL} we stated the Weyl law on the number of eigenvalues. There  also exists a 
pointwise local Weyl law which  is relevant to the pointwise behavior of  eigenfunctions.
The pointwise spectral function along the diagonal  is defined by
\begin{equation} \label{SPECFUN}  E(\lambda, x, x) = N(\lambda, x) : =\sum_{\lambda _{j} \leq
\lambda } |\varphi_j(x)|^2. \end{equation}
The pointwise Weyl law asserts tht 
\begin{equation} \label{PLWL} N(\lambda, x)   = \frac{1}{(2 \pi)^n} |B^n| \lambda^n +
R(\lambda, x),
\end{equation}
where $R(\lambda, x) = O(\lambda^{n-1})$ uniformly in $x$. These results are proved by studying the cosine
transform
\begin{equation} \label{COSINE}  E(t, x, x) = \sum_{\lambda _{j} \leq
\lambda }  \cos t \lambda_j\;\;|\varphi_j(x)|^2, \end{equation}
which is the fundamental (even) solution of the wave equation restricted to the diagonal. Background
on the wave equation is given in \S \ref{WAVEAPP}.

We note that the Weyl asymptote $\frac{1}{(2 \pi)^n} |B^n| \lambda^n $ is continuous, while the spectral
function \eqref{SPECFUN} is piecewise constant with jumps at the eigenvalues $\lambda_j$. Hence the
remainder must jump at an eigenvalue $\lambda$, i.e.
\begin{equation}\label{03} R(\lambda, x) - R(\lambda- 0, x) = \sum_{j: \lambda _{j} =
\lambda } |\varphi_j(x)|^2= O (\lambda^{n-1}). 
\end{equation}
on any compact Riemannian manifold. It follows immediately that
\begin{equation} \label{MAXSUP} \sup_M |\phi_j|  \lesssim  \lambda_j^{\frac{n-1}{2}}. \end{equation}

There exist $(M, g)$ for which this estimate is sharp, such as the standard spheres. However, as \eqref{NEGSUP}
the sup norms are smaller on manifolds of negative curvature. In fact, \eqref{MAXSUP} 
is very rarely sharp and the actual size of the sup-norms and other $L^p$ norms of eigenfunctions is another interesting
problem in global harmonic analysis. In \cite{SoZ} it is proved that if the bound \eqref{MAXSUP} is achieved by
some sequence of eigenfunctions, then there must exist a ``partial blow-down point'' or self-focal point $p$ where
a positive measure of directions $\omega \in S^*_p M$ so that the geodesic with initial value $(p, \omega)$
 returns  to $p$ at some time $T(p, \omega)$. Recently the authors have improved the result in the real
analytic case, and we sketch the new result in \S \ref{Lp}. 

To state it, we need some further notation and terminology. We only consider real analytic metrics for the
sake of simplicity. 
We call a point $p$  a  {\it self-focal point} or a{\it blow-down point} if there exists a time $T(p)$ so
that $\exp_p T(p) \omega = p$ for all $\omega \in S^*_p M$. Such a point is self-conjugate in a very
strong sense. In terms of symplectic geometry, the  flowout  manifold
\begin{equation} \label{Lambdap} \Lambda_{p} = \bigcup_{0 \leq t \leq T(p)} G^t S^*_{p} M  \end{equation}
is an embedded Lagrangian submanifold of $S^*M$ whose projection
$$\pi : \Lambda_{p} \to M $$
has a ``blow-down singuarity" at $t = 0, t = T(p)$ (see\cite{STZ}).
Focal points come in two basic kinds, depending on  the
first return map \begin{equation} \label{Phix} \Phi_p : S^*_p M \to S^*_p M,  \;\;\;\;
\Phi_p(\xi) := \gamma_{p, \xi}'(T(p)), \end{equation}
where $\gamma_{p, \xi}$ is the geodesic defined by the initial data $(p, \xi) \in S^*_x M$.
  We say that $p$  is  a {\it pole}  if  
 $$ \Phi_p = Id: S^*_p M \to S^*_p M. $$
On the other hand, it is possible that $\Phi_p = Id$ only on  a codimension one set  in $S^*_p M$.  We call such a $\Phi_p$ {\it twisted}.

  Examples of poles are 
the poles of a surface of revolution (in which case all geodesic loops at $x_0$ are smoothly closed). Examples
of self-focal points with fully  twisted return map are the 
four umbilic points of two-dimensional tri-axial ellipsoids, from  which  all geodesics loop back 
 at time $2 \pi$ but are almost never smoothly closed.  The only smoothly closed directions 
are the geodesic (and its time reversal) defined by the middle length `equator'.  
\bigskip

 At a self-focal point we have a kind of analogue of \eqref{UPHI} but not on $S^*M$ but just on $S^*_p M$. 
We define the 
Perron-Frobenius operator at a self-focal point by 
\begin{equation} \label{UX} U_x: L^2(S^*_x M, d \mu_x) \to L^2(S^*_x M, d\mu_x),
\;\;\;U_x f(\xi)  : = \begin{array}{l}  f(\Phi_x(\xi))
\sqrt{J_x(\xi)}.
\end{array}  \end{equation}
Here, $J_x$ is the Jacobian of the map $\Phi_x$, i.e.
$\Phi_x^*|d\xi| = J_x(\xi) |d\xi|$. 






The new result of C.D. Sogge and the author is the following:  \begin{maintheo} \label{TL}  If $(M, g)$ is real analytic and has maximal eigenfunction growth, then it
possesses a self-focal point whose first return map $\Phi_p$  has an invariant $L^2$ function in $L^2(S^*_p M)$. Equivalently,
it has an  $L^1$ invariant measure
in the class of the Euclidean volume density  $\mu_p$ on $S^*_p M$.  \end{maintheo}

For instance, the twisted first return map at an umbilic point of an ellipsoid has no such finite invariant
measure. Rather it has two fixed points, one of which is a source and one a sink, and the only finite invariant
measures are delta-functions at the fixed points. It also has an infinite invariant measure on the complement
of the fixed points, similar to $\frac{dx}{x}$ on $\R_+$. 

The results of \cite{SoZ,STZ,SoZ2} are stated for the $L^{\infty}$ norm but the same results are true
for $L^p$ norms above a critical index $p_m$ depending on the dimension (\S \ref{Lp}). The analogous
problem for lower $L^p$ norms is of equal interest, but the geometry of the extremals changes from
analogues of zonal harmonics to analogoues of Gaussian beams or highest weight harmonics. For the
lower $L^p$ norms there are also several new developments which are discussed in \S \ref{Lp}.

\subsection{\label{QMintro} Quasi-modes}

A significant generalization of eigenfunctions are {\it quasi-modes}, which are special kinds of approximate
solutions of the eigenvalue problem. The two basic types are:
\begin{itemize}

\item (i)  Lagrangian distributions given by oscillatory integrals which are approximate solutions of
the eigenvalue problem, i.e. which satisfy $||(\Delta + \mu_k^2) u_k || = O(\mu_k^{- p})$ for some $p \geq 0$. 

\item (ii) Any  sequence $\{u_k\}_{k = 1}^{\infty}$ of $L^2$ normalized solutions of an approximate 
eigenvalue problem.  In \cite{STZ} we worked with a more general class of ``admissible"  quasi-modes.
A
sequence $\{\psi_\lambda\}$, $\lambda=\lambda_j$, $j=1,2,\dots$ is
a sequence of admissible quasimodes if $\|\psi_\lambda\|_2=1$ and
\begin{equation}\label{admissible}
\|(\Delta+\lambda^2)\psi_\lambda\|_2+\|S^\perp_{2\lambda}\psi_\lambda\|_\infty
= o(\lambda).
\end{equation}

Here, $S_\mu^\perp$ denotes the projection onto the $[\mu,\infty)$
part of the spectrum of $\sqrt{-\Delta}$, and in what follows
$S_\mu=I-S_\mu^\perp$, i.e., $S_\mu f=\sum_{\lambda_j< \mu}
e_j(f)$, where $e_j(f)$ is the projection of $f$ onto the
eigenspace with eigenvalue $\lambda_j$.

\end{itemize}

The more special type (i) are studied in \cite{Arn,BB,CV2,K,R1, R2} and are analogous to Hadamard
parametrices for the wave kernel. However, none may  exist on a given $(M, g)$. For instance, Gaussian
beams \cite{BB,R1,R2} are special quasi-modes which concentrate along a closed geodesic. On the
standard $S^n$ they exist along any closed geodesic and are eigenfunctions. But one needs a stable
elliptic closed geodesic to construct a Gaussian beam and such closed geodesics do not exist when 
$g$ has negative curvature. They do however exist in many cases, but the associated Gaussian beams
are rarely exact eigenfunctions, and are typically just approximate eigenfunctions.  More general quasi-modes are constructed from geodesic flow-invariant
closed immersed Lagrangian submanifolds $\Lambda \subset S^*_g M$. But such invariant Lagrangian
submanifolds are also rare. They do arise in many interesting cases and in particular the unit co-sphere
bundle $S^*_g M$ is foliated by flow-invariant Lagrangian tori when the geodesic flow is integrable.
 this
is essentially the definition of complete integrability.

  For $p = 1$, the  more general type (ii) always exist: it suffices to define $u_k$ as superpositions of eigenfunctions
with frequencies $\lambda_j$ drawn from sufficiently narrow windows $[\lambda, \lambda + O(\lambda^{-p})]$.
 A model case admissible quasi-modes would be a sequence of
$L^2$-normalized functions $\{\psi_{\lambda_j}\}$ whose
$\sqrt{-\Delta}$ spectrum lies in intervals of the form
$[\lambda_j-o(1), \lambda_j+o(1)]$ as $\lambda_j\to \infty$.
If $p = 1$ then for generic metrics $g$,  it follows from the Weyl law with remainder  that there exist $\simeq C_g \lambda^{m-1}$ eigenvalues
in the window and one can construct many quasi-modes. For higher values of $p$ existence of refined quasi-modes
depends on the number of eigenvalues in very short intervals, and in general one only knows they exist 
by constructions of type (i). 

An important  example of such a
quasi-mode is a sequence of ``shrinking spectral projections",
i.e. the $L^2$-normalized projection kernels
$$\Phi_j^z(x) = \frac{\chi_{[\lambda_j, \lambda_j + \epsilon_j]}(x,
z)}{\sqrt{\chi_{[\lambda_j, \lambda_j + \epsilon_j]}(z, z)}}$$
with second point frozen at a point $z \in M$ and with width
$\epsilon_j \to 0$. Here, $\chi_{[\lambda_j, \lambda_j +
\epsilon_j]}(x, z)$ is the orthogonal projection onto the sum of
the eigenspaces $V_{\lambda}$ with $\lambda \in [\lambda_j,
\lambda_j + \epsilon_j]$ The zonal eigenfunctions of a surface of
revolution are examples of such shrinking spectral projections for
a sufficiently small $\epsilon_j$, and when $z$ is a partial focus
such $\Phi_j^z(x)$ are generalizations of zonal eigenfunctions. On
a general Zoll manifold, shrinking spectral projections of widths
$\epsilon_j = O(\lambda_j^{-1})$ are the direct analogues of zonal
spherical harmonics, and they would satisfy the analog of
\eqref{admissible} where $o(\lambda)$ is replaced by the much
stronger $O(\lambda^{-1})$.

There are several motivations to consider quasi-modes as well as actual modes (eigenfunctions).  First,
many results about eigenfunctions automatically hold for quasi-modes as well. Indeed, it is difficult
to distinguish modes and quasi-modes when using a wave kernel construction, and many of the methods
apply equally to modes and quasi-modes. Second, quasi-modes are often geometrically beautiful. They
are not stationary states but retain their shape under propagation by the wave group $U(t) = \exp it \sqrt{\Delta}$
for a ``very long time'' (essentially the Eherenfest time) before breaking up (see \cite{Ze2} for background). 
They are natural extremals for $L^p$ norms when they exist. Third, they are often the objects needed to
 close the gap between necessary and sufficient
conditions on eigenfunction growth. It is usually interesting to know whether a theorem about eigenfunctions
extends to quasi-modes as well.



\subsection{\label{FORMAT} Format of these lectures and references to the literature}

 In keeping with the format of the Park City summer school, we concentrate on the topics of the five 
lectures rather than give a systematic exposition of the subject. The more detailed account will
appear in \cite{Ze0}.  Various details of the proof 
are given as Exercises for the reader. The ``details" are intended to be stimulating and fundamental,  rather
than the tedious and routine aspects of proofs often left to readers in textbooks. As a result, the exercises vary widely in difficulty and amount of background assumed. Problems labelled {\bf Problems} are not exercises; they are
problems whose solutions are not currently known.

The technical backbone of the semi-classical analysis of eigenfunctions consists of wave equation methods
combined with the machinery of Fourier integral operators and Pseudo-differential operators. We do not have
time to review this theory. The main results we need are the construction of parametrices for the `propagator'
$E(t) = \cos t \sqrt{\Delta}$ and the Poisson kernel $\exp - \tau \sqrt{\Delta}$. We also need Fourier analysis
to construct approximate spectral projections $\rho(\sqrt{\Delta} - \lambda)$ and to prove Tauberian theorems
relating smooth expansions and cutoffs. 

The books \cite{GSj, DSj, D2, GSt1,GSt2, Sogb, Sogb2, Zw} give textbook treatments of the semi-classical methods with
applications to spectral asymptotics. Somewhat more classical background on the wave equation with many
explicit formulae in model cases can be found in \cite{TI,TII}. General spectral theory and the relevant functional
analysis can also be found in \cite{RS}.  The series \cite{HoI,HoII,HoIII,HoIV} gives a systematic presentation
of  Fourier integral operator theory: stationary phase and Tauberian theorems can be found in \cite{HoI},
Weyl's law and spectral asymptotics can be found in \cite{HoIII,HoIV}.

In \cite{Ze0} the author gives a more systematic presentation of results on nodal sets, $L^p$ norms
and other aspects of eigenfunctions. Earlier surveys \cite{Ze,Ze2,Ze3} survey  related material.  Other monographs
on $\Delta$-eigenfunctions can be found in \cite{HL} and \cite{Sogb2}.  The methods of \cite{HL} mainly
involve the local harmonic analysis of eigenfunctions and rely more on classical elliptic estimates, on frequency
functions and of one-variable complex analysis. The exposition in \cite{Sogb2} is close to the one given  here
but does not extend to the recent results that we highlight in these lectures and in  \cite{Ze0}.
\bigskip

\noindent{\bf Acknowledgements}  Many of the results  discussed in these lectures is joint work with   C. D. Sogge and/or John. A. Toth.  Some of the work in progress is also with B. Hanin and P. Zhou. We also thank E. Potash for his comments
on earlier versions. 

\section{\label{FOUND} Foundational results on nodal sets}

As mentioned above, the nodal domains of an eigenfunction are the connected components of $M \backslash \ncal_{\phi_{\lambda}}$.
In the case of a domain with boundary and Dirichlet boundary conditions, the nodal set is defined by taking the 
closure of the zero set in $M \backslash \partial M$. 

The eigenfunction is either positive or negative in each nodal domain and changes sign as the nodal set is crossed
from one domain to an adjacent domain. Thus the set of nodal domains can be given the structure of a bi-partite graph
\cite{H}.
Since the eigenfunction has one sign in each nodal domain, it is  the ground state  eigenfunction with Dirichlet
boundary conditions in each nodal domain.

In the case of domains $\Omega \subset \R^n$ (with the Euclidean metric), the Faber-Krahn inequality states
that the lowest eigenvalue (ground state eigenvalue, bass note) $\lambda_1(\Omega)$ for the Dirichlet problem
has the lower bound,
\begin{equation} \label{FK1} \lambda_1(\Omega) \geq |\Omega|^{- \frac{2}{n}} \; C_n^{\frac{2}{n}} \; j_{\frac{n-2}{2}}, 
\end{equation}
where $|\Omega$ is the Euclidean volume of $\Omega$, $C_n = \frac{\pi^{\frac{n}{2}} }{\Gamma(\frac{n}{2} + 1)} $
is the volume of the unit ball in $\R^n$ and where $j_{m, 1}$ is the first positive zero of the Bessel function $J_m$.
That is, among all domains of a fixed volume the unit ball has the lowest bass note. 


\subsection{Vanishing order and scaling near zeros}
By the vanishing order
$\nu(u, a)$ of $u$ at $a$ is meant the largest positive integer
such that $D^{\alpha} u(a) = 0$ for all $|\alpha| \leq \nu$. A unique continuation
theorem shows that the  vanishing order
of an eigenfunction at each zero is
finite.  The following estimate is a quantitative
version of this fact.

\begin{theo} \label{VO} (see \cite{DF}; \cite{Lin} Proposition 1.2 and Corollary 1.4; and \cite{H} Theorem 2.1.8.)
 Suppose that $M$ is compact and of dimension $n$. Then there exist constants $C(n), C_2(n)$ depending only on the dimension such that
the  the vanishing order $\nu(u, a)$ of $u$ at $a \in M$ satisfies
$\nu(u, a) \leq C(n) \; N(0, 1) + C_2(n)$ for all $a \in
B_{1/4}(0)$. In the case of a global  eigenfunction, $\nu(\phi_{\lambda},
a) \leq C(M, g) \lambda.$
\end{theo}

Highest weight spherical harmonics $C_n (x_1 + i
x_2)^N$ on $S^2$ are examples which vanish at the maximal order of
vanishing at the poles $x_1 = x_2 = 0, x_3 = \pm 1$.

The following Bers scaling rule extracts the leading term in the Taylor expansion of
the eigenfunction around a zero:

\begin{prop} \cite{Bers,HW2} Assume that $\phi_{\lambda}$ vanishes to order $k$ at
$x_0$. Let $\phi_{\lambda}(x) = \phi_k^{x_0} (x) + \phi^{x_0}_{k +
1} + \cdots$ denote the $C^{\infty}$ Taylor expansion of
$\phi_{\lambda}$ into homogeneous terms in normal coordinates $x$
centered at $x_0$.  Then $\phi_k^{x_0}(x)$ is a Euclidean harmonic
homogeneous polynomial of degree $k$.
\end{prop}

To prove this, one substitutes  the homogeneous expansion into the
equation  $\Delta \phi_{\lambda} = \lambda^2 \phi_{\lambda}$ and
rescales $x \to \lambda x$, i.e. one  applies the dilation operator
\begin{equation}D_{\lambda}^{x_0} \phi_{\lambda}(u ) = \phi(x_0 + \frac{u}{\lambda}). \end{equation}
The rescaled eigenfunction is an eigenfunction of the locally
rescaled Laplacian \begin{equation} \label{SCALING}  \Delta^{x_0}_{\lambda} : = \lambda^{-2}  D_{\lambda}^{x_0}
\Delta_g (D_{\lambda}^{x_0} )^{-1} = \sum_{j = 1}^n
\frac{\partial^2}{\partial u_j^2} + \cdots \end{equation}  in Riemannian
normal coordinates $u$ at $x_0$ but now with eigenvalue $1$,
\begin{equation}\begin{array}{l}  D_{\lambda}^{x_0}  \Delta_g (D_{\lambda}^{x_0} )^{-1} \phi(x_0 + \frac{u}{\lambda}) =
\lambda^2 \phi(x_0 + \frac{u}{\lambda}) \\ \\
\implies \Delta^{x_0}_{\lambda} \phi(x_0 + \frac{u}{\lambda}) =
\phi(x_0 + \frac{u}{\lambda}).  \end{array} \end{equation}
Since
$\phi(x_0 + \frac{u}{\lambda})$ is, modulo lower order terms, an
eigenfunction of a standard flat Laplacian on $\R^n$, it behaves near a zero as
a sum of homogeneous Euclidean harmonic polynomials.

In dimension 2, a homogeneous harmonic polynomial of degree $N$ is the real or imaginary part of the unique
holomorphic homogeneous polynomial $z^N$ of this degree, i.e. $p_N(r, \theta) = r^N \sin N \theta$. As observed in \cite{Ch},
there exists a $C^1$ local diffeormorphism $\chi$  in a disc around a zero $x_0$ so that $\chi(x_0) = 0$
and so that   $\phi_N^{x_0}  \circ \chi  = p_N.$
It follows that the restriction of $\phi_{\lambda}$ to a curve H is $C^1$ equivalent around a zero to $p_N$ restricted
to $\chi(H)$.
The nodal set of $p_N$ around $0$ consists of $N$ rays, $\{r (\cos \theta, \sin \theta) : r > 0, p_N |_{S^1}(v) = 0\}$.
It follows that the local structure of the nodal set in a small disc around a singular point p  is $C^1$ equivalent to
$N$ equi-angular rays emanating from $p$. We refer to \cite{HW, HW2,Ch,Ch1,Ch2,Bes} for futher background
and results. 

\bigskip

\noindent{\bf Question}  Is there any  useful scaling behavior of $\phi_{\lambda}$ around its critical points?

\subsection{\label{PFSMALLBALL} Proof of Proposition \ref{SMALLBALL}}


The proofs are based on rescaling the eigenvalue problem in small balls.

\begin{proof} Fix  $x_0, r$ and consider $B(x_0, r)$. If
$\phi_{\lambda}$ has no zeros in $B(x_0,r)$, then $B(x_0, r) \subset D_{j; \lambda}$
must be contained in the interior of a nodal domain $ D_{j; \lambda}$ of $\phi_{\lambda}$.
Now $\lambda^2 = \lambda_1^2 (D_{j; \lambda})$ where $ \lambda_1^2 (D_{j; \lambda})$ is the smallest
Dirichlet eigenvalue for the nodal domain. By domain monotonicity of the lowest Dirichlet eigenvalue (i.e.
$\lambda_1(\Omega)$ decreases as $\Omega$ increases),
$\lambda^2 \leq \lambda_1^2 (D_{j; \lambda}) \leq \lambda_1^2(B(x_0, r)).$ To complete the proof
we show that  $\lambda_1^2(B(x_0, r)) \leq \frac{C}{r^2}$ where $C$ depends only on the  metric. This is proved
by comparing $\lambda_1^2(B(x_0, r)) $   for the metric
$g$ with the lowest Dirichlet Eigenvalue  $\lambda_1^2(B(x_0, c r ); g_0)$ for the Euclidean ball $B(x_0, c r; g_0)$
 centered at $x_0$
of radius $ c r$ with Euclidean metric $g_0 $ equal to $g$ with coefficients  frozen at $x_0$; $c$ is chosen
so that $ B(x_0, c r; g_0) \subset B(x_0, r, g)$. Again by domain monotonicity, $\lambda_1^2 (B(x_0, r, g))
\leq \lambda_1^2 ( B(x_0, c r; g)) $ for $c < 1$. By comparing
 Rayleigh quotients $\frac{\int_{\Omega} |df|^2 dV_g}{\int_{\Omega} f^2 dV_g}$ one easily sees that
 $\lambda_1^2 ( B(x_0, c r; g)) \leq C \lambda_1^2 ( B(x_0, c r; g_0))$ for some $C$ depending only on the metric.
 But by explicit calculation with Bessel functions, $\lambda_1^2 ( B(x_0, c r; g_0)) \leq \frac{C}{r^2}. $
 Thus, $\lambda^2 \leq \frac{C}{r^2}$.

 \end{proof}

For background we refer to \cite{Ch}.

\subsection{A second proof}
Another proof is given in \cite{HL}:  Let $u_{r}$ denote the ground
state Dirichlet eigenfunction for $B(x_0, r)$.   Then  $u_{r} > 0$ on the interior of $B(x_0, r)$. If $B(x_0, r) \subset D_{j; \lambda}$
then also $\phi_{\lambda} > 0 $ in $ B(x_0, r)$.
Hence the ratio  $\frac{u_{r}}{\phi_{\lambda}}$
is smooth and non-negative, vanishes only on $\partial B(x_0,r)$, and must have its maximum at a point
$y$ in the interior of $B(x_0, r)$. At this point (recalling that our $\Delta$ is minus the sum of squares),
$$\nabla \left(\frac{u_{r}}{\phi_{\lambda}} \right) (y) = 0, \;\; - \Delta \left(\frac{u_{r}}{\phi_{\lambda}} \right) (y) \leq 0, $$
so at $y$,
$$0 \geq - \Delta \left( \frac{u_{r}}{\phi_{\lambda}} \right) = - \frac{\phi_{\lambda} \Delta u_{r} - u_r \Delta \phi_{\lambda}}{\phi_{\lambda}^2}
= -  \frac{( \lambda_1^2(B(x_0, r)) - \lambda^2)\phi_{\lambda}  u_{r}}{\phi_{\lambda}^2}. $$
Since $\frac{ \phi_{\lambda}  u_{r}}{\phi_{\lambda}^2} >0,$ this is possible  only if
$\lambda_1(B(x_0, r)) \geq \lambda$.

To complete the proof we note that if $r  = \frac{A}{\lambda}$ then the metric is essentially Euclidean. We rescale
the ball by $x \to \lambda x$ (with coordinates centered at $x_0$) and then obtain an essentially Euclidean ball of radius $r$.
Then $\lambda_1(B(x_0,  \frac{r}{\lambda}) = \lambda \lambda_1 B_{g_0} (x_0, r)$. Therefore we only need to choose
$r$ so that $\lambda_1 B_{g_0} (x_0, r)  = 1$.
\bigskip

\noindent{\bf Problem} Are the above results true as well for quasi-modes of order zero (\S \ref{QMintro}, \S \ref{QM})?

\subsection{Rectifiability of the nodal set}

We recall that the nodal set of an eigenfunction $\phi_{\lambda}$  is its zero set. When zero is a regular value
of $\phi_{\lambda}$ the nodal set is a smooth hypersurface. This is a generic property of eigenfunctions \cite{U}.
It is pointed out in \cite{Bae} that eigenfunctions can always be locally represented  in the form
$$\phi_{\lambda} (x) = v(x) \left(x_1^k + \sum_{j = 0}^{k - 1} x_1^j u_j(x') \right),$$
in suitable coordinates $(x_1, x')$ near $p$,  where $\phi_{\lambda}$ vanishes to order $k$ at $p$, where
$u_j(x')$ vanishes to order $k - j$ at $x' = 0$,  and where
$v(x) \not= 0$ in a ball around $p$. It follows that the nodal set is always countably $n-1$ rectifiable when
$\dim M = n$.

\section{\label{LB} Lower bounds for $\hcal^{m-1}(\ncal_\la)$ for $C^{\infty}$ metrics}

In this section we review the  lower bounds on $\hcal^{n-1}(Z_{\phi_{\lambda}})$ from
\cite{CM,SoZ,SoZa,HS,HW}.  Here  $$\hcal^{n-1}(\ncal_{\phi_{\lambda}}) = \int_{Z_{\phi_{\lambda}}} dS $$
is the Riemannian surface measure, 
where $dS$ 
denotes the Riemannian  volume element  on   the nodal set, i.e. the insert $iota_{n} dV_g$ of the unit normal
into the volume form of $(M, g)$.
 The main result is:

\begin{theo} \label{DONGLB} Let $(M, g)$ be a $C^{\infty}$ Riemannian manifold.   Then there exists a constant $C$ independent of $\lambda$ 
such that 
$$
C \lambda^{1-\frac{n-1}2} \leq  \hcal^{n-1}(\ncal_{\phi_{\lambda}}). $$
\end{theo}

\begin{rem}  In a recent article \cite{BlSo}, M. Blair and C. Sogge improve this result on manifolds
of non-positive curvature by showing that the right side divided by the left side tends to infinity. 
There exists a related proof using a comparison of diffusion processes  in \cite{Stei}.
The result generalizes in a not completely straighforward way   to  Schr\"odinger operators
$- \frac{\hbar^2}{2} \Delta_g + V$ for certain potentials $V$ \cite{ZZh} (see also \cite{Jin} for generalizations
of \cite{DF} to Schr\"odinger operators). The new issue is the separation
of the domain into classically allowed and forbidden regions. In \cite{HZZ} the density of zeros
in both regions is studied for random Hermite functions. \end{rem}

We sketch the proof of  Theorem~\ref{DONGLB}  from \cite{SoZ,SoZa}. The starting point is an identity  from \cite{SoZ} (inspired by an identity
in \cite{Dong}): 

\begin{prop} \label{DONGPROP}  For any $f \in C^2(M)$,
\begin{equation}\label{1}
\int_M |\phi_{\lambda}| \, (\Delta_g+\lambda^2)f \, dV_g = 2 \int_{\ncal_{\phi_{\lambda}} }|\nabla_g \phi_\lambda| \, f
\, dS,
\end{equation}
\end{prop}

When $f \equiv 1$ we obtain

\begin{cor} \label{DONGCOR} 
\begin{equation}\label{1a}
\lambda^2 \int_M |\phi_{\lambda}|  \, dV_g = 2 \int_{\ncal_{\phi_{\lambda}} }|\nabla_g \phi_\lambda| \, f
\, dS,
\end{equation}
\end{cor}
\bigskip

\begin{mainex}  Prove this identity by decomposing $M$ into a union of nodal domains. \end{mainex}
\bigskip

\noindent{\bf Hint: }
The nodal domains form  a partition of $M$, i.e. 
$$M=\bigcup_{j=1}^{N_+(\lambda)}D^+_j \, \cup \, \bigcup_{k=1}^{N_-(\lambda)} D^-_k\, \cup
{\mathcal N}_\lambda,$$ where the $D^+_j$ and $D^-_k$ are the
positive and negative nodal domains of $\phi_\lambda$, i.e, the
connected components of the sets $\{\phi_\lambda>0\}$ and
$\{\phi_\lambda<0\}$.

Let us assume for the moment that $0$ is a regular value for
$\phi_\lambda$, i.e., $\Sigma=\emptyset$.  Then each $D^+_j$ has
smooth boundary $\partial D^+_j$, and so if $\partial_\nu$ is the
Riemann outward normal derivative on this set, by the Gauss-Green
formula we have
\begin{equation} \begin{array}{ll} 
\int_{D^+_j}((\Delta+\lambda^2) f) \, |\phi_\lambda|\, dV
&=\int_{D^+_j}((\Delta+\lambda^2) f) \, \phi_\lambda \, dV
\\
&=\int_{D^+_j}f\, (\Delta+\lambda^2)\phi_\lambda dV-
\int_{\partial D^+_j}f \, \partial_\nu \phi_\lambda \, dS
\\
&=\int_{\partial D^+_j} f \, |\nabla \phi_\lambda|\, dS.
\end{array} \end{equation}

We use   that $-\partial_\nu \phi_\lambda = |\nabla
\phi_\lambda|$ since $\phi_\lambda=0$ on $\partial D_j^+$ and
$\phi_\lambda$ decreases as it crosses $\partial D_j^+$ from
$D^+_j$. A similar argument shows that
\begin{equation}\begin{array}{l} 
\int_{D_k^-} ((\Delta+\lambda^2)f )\, |\phi_\lambda|\, dV
=-\int_{D_k^-} ((\Delta+\lambda^2)f) \, \phi_\lambda \, dV\\ \\  =\int
f\,
\partial_\nu \phi_\lambda \, dS =\int_{\partial D^-_k} f\, |\nabla
\phi_\lambda| \, dS, \end{array}
\end{equation}
using in the last step that $\phi_\lambda$ increases as it crosses
$\partial D^-_k$ from $D^-_k$.

  If we sum these two identities
over $j$ and $k$, we get
\begin{multline*}
\int_M ((\Delta+\lambda^2)f) \, |\phi_\lambda|\, dV
=\sum_j\int_{D^+_j} ((\Delta+\lambda^2)f) \, |\phi_\lambda|\, dV
\\ +\sum_{k} \int_{D^-_k} ((\Delta+\lambda^2)f)\, |\phi_\lambda|\, dV
\\
=\sum_j \int_{\partial D_j^+} f\, |\nabla \phi_\lambda|\, dS\\  +
\sum_k \int_{\partial D_k^-} f\, |\nabla \phi_\lambda|\, dS =
2\int_{{\mathcal N}_\lambda} f\, |\nabla \phi_\lambda|\, dS,
\end{multline*}
using the fact that ${\mathcal N}_\lambda$ is the disjoint union
of the $\partial D^+_j$ and the disjoint union of the $\partial
D^-_k$.
\bigskip

The lower bound of Theorem \ref{DONGLB}  follows from the identity in Corollary \ref{DONGCOR}
and the following lemma:

\begin{lem} \label{lem} If $\la>0$ then 
\begin{equation}\label{4.0}\|\nabla_g \phi_\la\|_{L^\infty(M)}\lesssim \la^{1+\frac{n-1}2}\| \phi_\la\|_{L^1(M)}
\end{equation}
 \end{lem}
Here, $A(\lambda) \lesssim B(\lambda)$ means that there exists a constant independent of $\lambda$
so that $A(\lambda ) \leq C B(\lambda). $

By   Lemma \ref{lem} and Corollary \ref{DONGCOR}, we have
\begin{equation} \label{EST1}\begin{array}{lll} \la^2  \int_M|\phi_\la|\, dV
=2\int_{\ncal_\la}|\nabla_g \phi_\la|_g\, dS &\le &  2|\ncal_\la| \, \|\nabla_g \phi_\la\|_{L^\infty(M)}\\ &&\\
&\lesssim & 2|\ncal_\la| \; \la^{1+\frac{n-1}2}  \|\phi_\la\|_{L^1(M)} . \end{array}
 \end{equation}
Thus Theorem \ref{DONGLB} follows from the somewhat curious cancellation of $||\phi_{\lambda}||_{L^1}$
from the two sides of the inequality.
\bigskip

\noindent{\bf Problem} Show that Corollary \ref{DONGCOR} and Lemma \ref{lem} are   true  modulo $O(1)$
for quasi-modes of order zero (\S \ref{QMintro}, \S \ref{QM}).

\subsection{Proof of Lemma \ref{lem}}

\begin{proof}  
The main point is to construct a designer reproducing kernel $K_{\lambda}$ for $\phi_{\lambda}$:
Let  $\hat{\rho} \in C^\infty_0({\mathbb R})$ satisfy $\rho(0) = \int \hat{\rho} \, dt=1$. Define the operator
\begin{equation} \label{RHO}  \rho (\la - \sqrt{\Delta}) : L^2(M) \to L^2(M)   \end{equation}
by 
\begin{equation} \label{rhola} \rho(\la - \sqrt{\Delta})  f= \int_{\R} \hat{\rho}(t)  e^{it\la} e^{-it \sqrt{-\Delta}}f\, dt.\end{equation}
Then \eqref{RHO} is a function of $\Delta$ and has $\phi_{\lambda}$ as an eigenfunction
with eigenvalue $\rho(\lambda - \lambda) = \rho (0) = 1. $ Hence,

 $$\rho(\la - \sqrt{\Delta}) \phi_\la = \phi_\la. $$
\bigskip

\begin{mainex} Check that \eqref{rhola} 
has the spectral expansion,
\begin{equation}\label{5.0}
\rho(\la - \sqrt{\Delta})  f
=\sum_{j=0}^\infty  \rho(\la-\la_j)E_jf,
\end{equation}
where $E_jf$  is the  projection of $f$ onto the $\lambda_j$- eigenspace of $\sqrt{-\Delta_g}$. 
Conclude that  \eqref{rhola}
reproduces $\phi_\la$ if $ \rho(0) = 1$. 
\end{mainex}
\bigskip

We may choose $\rho$ further so that  $\hat{\rho}(t)=0$ for $t\notin [\epsilon/2,\epsilon]$.
\bigskip

{\bf CLAIM} If supp $\hat{\rho} \subset [\epsilon/2,\epsilon]$. then 
the kernel $K_\la(x,y)$ of $\rho(\la - \sqrt{\Delta}) $ for
$\epsilon$  sufficiently small satisfies
\begin{equation} \label{K1} |\nabla_g K_\la(x,y)|\le C \la^{1+\frac{n-1}2}. \end{equation}

The Claim proves the Lemma, because
 $$\begin{array}{lll} \nabla_x \phi_{\lambda}(x) & = & \nabla_x \rho (\la - \sqrt{\Delta}) \phi_\la (x)
\\ &&\\
& = &\int_M  \nabla_x K_{\lambda}(x, y) \phi_{\lambda}(y) dV(y) \\ &&\\
& \leq  & 
  C  \sup_{x, y} |\nabla_x K_{\lambda}(x, y)| \int_M |\phi_{\lambda}| dV\\ &&\\
& \leq & 
\la^{1+\frac{n-1}2} ||\phi_{\lambda}||_{L^1}  \end{array}$$ which implies the lemma.

The gradient estimate on $K_{\lambda}(x,y)$ is based on the following ``parametrix" for
the designer reproducing kernel:
\begin{prop} \label{Kaprop} 

\begin{equation} K_\la (x, y) = \lambda^{\frac{n-1}{2}} a_\la (x, y) e^{i \lambda r(x, y) }, \end{equation}
where $a_\la(x, y)$  is bounded with bounded derivatives in $(x, y)$ and where $r(x, y)$ is the Riemannian distance
between points.
\end{prop}

\begin{proof}

Let $U(t) = e^{- it \sqrt{\Delta}}$. We may write
\begin{equation} \label{RHOLA} \rho(\lambda - \sqrt{\Delta}) = \int_{\R} \hat{\rho}(t) e^{ it  \lambda} U(t, x, y) dt. \end{equation}
As reviewed in \S \ref{HADASECT},   for small $t$ and $x, y$ near the diagonal one may construct
the Hadamard parametrix, 
$$U(t, x, y) = \int_0^{\infty} e^{i \theta (r^2(x, y) - t^2) } At,(x, y, \theta) d \theta $$
modulo a smooth remainder (which may be neglected).

\begin{mainex} Explain why the remainder may be neglected. How many of the terms in the parametrix
construction does one need in the proof of Proposition \ref{Kaprop}?  (Hint: if one truncates the amplitude
after a finite number of terms in the Hadamard parametrix, the remainder lies in $C^k$ and then the 
contribution to \eqref{RHOLA} decays as $\lambda \to \infty$.) 
\end{mainex} \bigskip

Thus, 
$$K_{\lambda}(x, y) = \int_{\R} \int_0^{\infty} e^{i \theta (r^2(x, y) - t^2) }  e^{ it  \lambda} 
\hat{\rho}(t) At,(x, y, \theta) d \theta dt. $$

We change variables $\theta \to \lambda \theta$ to obtain
$$K_{\lambda}(x, y) = \lambda \int_{\R} \int_0^{\infty} e^{i \lambda [\theta (r^2(x, y) - t^2) + t] }  \hat{\rho}(t) 
At,(x, y, \lambda\theta) d \theta dt. $$
We then   apply stationary phase. The phase is
$$ \theta (r^2(x, y) - t^2)  + t $$
and the critical point equations are
$$r^2(x, y) = t^2, \;\;\; 2 t \theta = 1, \;\;\; (t \in (\epsilon, 2 \epsilon)). $$
The power of $\theta$ in the amplitude is $\theta^{\frac{n-1}{2}}$. The change
of variables thus puts  in $\lambda^{\frac{n + 1}{2}}. $But we get $\lambda^{-1}$
from stationary phase with two variables $(t, \theta)$.

The value of the phase at the critical point is $e^{ i t \lambda}= e^{ i \lambda r(x, y)}. $
The Hessian in $(t, \theta)$ is $2 t$ and it is invertible. Hence, 
$$K_{\lambda}(x, y) \simeq  \lambda^{\frac{n-1}{2}} e^{ i \lambda r(x, y)} 
  a(\lambda, x, y), $$
where
$$a \sim a_0 + \lambda^{-1} a_{-1} + \cdots $$
and 
$$a_0 = A_0(r(x, y), (x, y, \frac{2}{r(x, y)} ). $$

\end{proof}

Proposition \ref{Kaprop} implies  that  $|\nabla_g K_\la(x,y)|\le C \la^{1+\frac{n-1}2}$ by directly differentiating the 
expression. The extra power of $\lambda$ comes from the ``phase factor "
$ e^{i \lambda r(x, y) }$.

This concludes the proof of Lemma \ref{lem}.
\end{proof}


\begin{rem} 
There are many 	`reproducing kernels' if one only requires them to reproduce one eigenfunction. A very common
choice is the spectral projections operator
$$\Pi_{[\lambda, \lambda + 1]} (x, y) = \sum_{j: \lambda_j \in [\lambda, \lambda + 1]}
\phi_j(x) \phi_j(y) $$
for the interval $[\lambda, \lambda + 1]$. 
It reproduces {\it all} eigenfunction $\phi_k$ with $\lambda_k \in [\lambda, \lambda +1]. $
This reproducing kernel cannot be used in our application because $\Pi_{\lambda}(x, x) \simeq \lambda^{n-1}$, as  follows from
the local Weyl law. Similarly,  $\sup_{x, y} |\nabla_x \Pi_{\lambda}(x, y)| \simeq \lambda^n$. The reader may
check these statements on the spectral projections kernel for the standard sphere (\S \ref{SHAPP}).

\end{rem}

\subsection{Modifications}

Hezari-Sogge modified the proof  Proposition \ref{DONGPROP} 
in \cite{HS}  to prove

\begin{theo} \label{HS} For any $C^{\infty}$ compact Riemannian manifold,  the $L^2$-normalized eigenfunctions satisfy $$\hcal^{n-1}(\ncal_{\phi_{\lambda}}) \geq C \; \lambda \; ||\phi_{\lambda}||_{L^1}^2. $$
\end{theo}

They first apply   the Schwarz inequality to get
\begin{equation}\label{2}
\lambda^2 \int_M |\phi_\lambda| \, dV_g \le 2 (\hcal^{n-1}(\ncal_{\phi_{\lambda}}))^{1/2} \, \left(\, 
\int_{Z_{\phi_{\lambda}}}|\nabla_g \phi_\lambda|^2 \, dS\, \right)^{1/2}.
\end{equation}
They then  use the test function 
\begin{equation}\label{6}
f=\big(\, 1+\lambda^2 \phi_\lambda^2 + |\nabla_g \phi_\lambda|^2_g \, \bigr)^{\frac12}
\end{equation}
 in Proposition \ref{DONGPROP} to show that
\begin{equation} \int_{\ncal_{\phi_{\lambda}}}|\nabla_g \phi_\lambda|^2 \, dS \leq \lambda^3. \end{equation}
See also \cite{Ar} for the generalization to the nodal bounds to Dirichlet and Neumann
eigenfunctions of  bounded domains.

Theorem \ref{HS} shows that Yau's conjectured lower bound would follow for a sequence of eigenfunctions
satisfying $||\phi_{\lambda}||_{L^1} \geq C > 0$ for some positive constant $C$.

\subsection{\label{LBL1} Lower bounds on $L^1$ norms of eigenfunctions}

The following universal lower bound is optimal as $(M, g)$ ranges over all compact Riemannian
manifolds.

\begin{mainprop} \label{CS}
For any $(M, g)$ and any $L^2$-normalized eigenfunction,
$||\phi_{\lambda}||_{L^1} \geq C_g \lambda^{- \frac{n-1}{4}}$.
\end{mainprop}

\begin{rem}  There are few results on $L^1$ norms of eigenfunctions. The reason is probably
that $|\phi_{\lambda}|^2 dV$ is the natural probability measure associated to eigenfunctions.  It
is straightforward to show that the expected $L^1$ norm of random $L^2$-normalized spherical
harmonics  of degree $N$ and their generalizations to any $(M, g)$ is a positive constant $C_N$
with a uniform positive lower bound. One expects eigenfunctions in the ergodic case to have the
same behavior. \end{rem}

\begin{mainprob}  A difficult but interesting problem would be to show that $||\phi_{\lambda}||_{L^1}
\geq C_0 > 0$ on a compact hyperbolic manifold. A partial result in this direction would be useful.

\end{mainprob}

\subsection{\label{DUB} Dong's upper bound}

Let $(M, g)$ be a compact $C^{\infty}$ Riemannian manifold of
dimesion $n$, let $\phi_{\lambda}$ be an $L^2$-normalized
eigenfunction of the Laplacian,
$$\Delta \phi_{\lambda} = - \lambda^2 \phi_{\lambda},
$$
Let 
\begin{equation} \label{q} q =|\nabla \phi|^2 +   \lambda^{2}  \phi^2. \end{equation} In Theorem 2.2 of
\cite{D},  R. T. Dong proves the bound (for $M$ of any dimension $n$),
\begin{equation} \label{MAINDONG} \hcal^{n-1}(\ncal \cap \Omega) \leq \half \int_{\Omega} 
|\nabla \log q | + \sqrt{n } vol(\Omega) \lambda + vol(\partial \Omega).  \end{equation}
He also proves  (Theorem 3.3)  that on a surface,
\begin{equation} \label{Deltaq} \Delta \log q \geq - \lambda + 2\min(K, 0) + 4 \pi \sum_i (k_i - 1)
\delta_{p_i},  \end{equation}  where $\{p_i\}$ are the singular points and $k_i$
is the order of $p_i$. In Dong's notation, $\lambda > 0$. Using a weak Harnack inequality together with \eqref{Deltaq},
Dong proves (\cite{D}, (25)) that in dimension two,
\begin{equation} \label{GRADQ} \int_{B_R} |\nabla \log q| \leq C_g R \lambda + C'_g \lambda^2 R^3. \end{equation}
Combining with  \eqref{MAINDONG} produces  the upper bound $\hcal^{1}(\ncal \cap \Omega)  \leq \lambda^{3/2}$ in dimension 2.

\begin{mainprob} To what extent can one generalize these estimates to higher dimensions?

\end{mainprob}

 \subsection{Other level sets}

Although nodal sets are special, it is of interest to bound the Hausdorff surface measure of any level set
$\ncal_{\phi_{\lambda}}^c : = \{\phi_{\lambda} = c\}$.  Let $\sgn
(x) = \frac{x}{|x|}$.

\begin{prop} \label{BOUNDSc} For any $C^{\infty}$ Riemannian
manifold, and any $f \in C(M)$ we have,

\begin{equation} \label{DONGTYPEc}  \int_M f (\Delta + \lambda^2)\; |\phi_{\lambda} - c| \;dV
+ \lambda^2 c \int f \mbox{\sgn} (\phi_{\lambda} - c) dV = 2\;
\int_{\ncal^c_{\phi_{\lambda}}}   f |\nabla \phi_{\lambda}| dS.
\end{equation}

\end{prop}

This identity has similar implications for
$\hcal^{n-1}(\ncal^c_{\phi_{\lambda}})$ and for the
equidistribution of level sets.

\begin{cor}\label{cintro} For
$c \in {\mathbb R}$
$$\lambda^2\int_{\phi_\lambda\ge c}\phi_\lambda dV
= \int_{\ncal^c_{\phi_{\lambda}}}   |\nabla \phi_{\lambda}| dS
. $$ 

\end{cor}

One can obtain lower bounds on $\hcal^{n-1}(\ncal^c_{\phi_{\lambda}})$ as in the 
case of nodal sets. However the integrals of $|\phi_{\lambda}|$ no longer cancel out.
The numerator is smaller since one only integrates over $\{\phi_{\lambda} \geq c\}$. 
Indeed,  $\hcal^{n-1}(\ncal^c_{\phi_{\lambda}})$ must tend to zero as $c$ tends to
the maximum possible threshold $\lambda^{\frac{n-1}{2}}$ for $\sup_M |\phi_{\lambda}|$.

The Corollary follows by integrating $\Delta$ by parts,
 and by using the identity,
 \begin{equation} \label{c} \begin{array}{lll} \int_M |\phi_{\lambda} - c| + c \;\sgn(\phi_{\lambda} - c)
 \;dV & = & \int_{\phi_{\lambda} > c} \phi_{\lambda} dV -
 \int_{\phi_{\lambda} < c} \phi_{\lambda} dV \\ && \\ &=& 2\int_{\phi_{\lambda} > c} \phi_{\lambda} dV ,
 \end{array} \end{equation}
 since $0=\int_M \phi_\lambda dV=\int_{\phi_\lambda>c}\phi_\lambda dV
 +\int_{\phi_\lambda<c}\phi_\lambda dV$.

\begin{mainprob} A difficult problem would be to study  $\hcal^{n-1}(\ncal^c_{\phi_{\lambda}})$ as a function of $(c, \lambda)$
and try to find thresholds where the behavior changes. For random spherical harmonics,
$\sup_M |\phi_{\lambda}| \simeq \sqrt{\log \lambda} $ and one would expect the level
set volumes to be very small above this height except in special cases. 

\end{mainprob}

\subsection{\label{EXAMPLES} Examples}

The lower bound of Theorem \ref{DONGLB} is far from the lower bound conjectured by Yau, which
by Theorem \ref{DF} is correct at least in the real analytic case.  In this
section we go over the model examples to understand why the methds are not always getting sharp results. 

\subsubsection{\label{Flat tori section2} Flat tori}

We have,  $|\nabla \sin \langle k, x
\rangle|^2 = \cos^2 \langle k, x \rangle |k|^2$. Since $\cos
\langle k, x \rangle = 1$ when $\sin \langle k, x \rangle  = 0$
the integral is simply $|k|$ times the surface volume of the nodal
set, which is known to be of size $|k|$.  Also, we
have $\int_{{\bf T}} |\sin \langle k, x \rangle| dx \geq C$. Thus,
our method gives  the sharp lower bound
$\hcal^{n-1}(Z_{\phi_{\lambda}}) \geq C \lambda^{1}$ in this
example.

So the upper bound  is achieved in this example. Also, we
have $\int_{{\bf T}} |\sin \langle k, x \rangle| dx \geq C$. Thus,
our method gives  the sharp lower bound
$\hcal^{n-1}(Z_{\phi_{\lambda}}) \geq C \lambda^{1}$ in this
example.
Since $\cos
\langle k, x \rangle = 1$ when $\sin \langle k, x \rangle  = 0$
the integral is simply $|k|$ times the surface volume of the nodal
set, which is known to be of size $|k|$. 
\medskip

\subsubsection{Spherical harmonics on $S^2$}
\medskip

For background on spherical harmonics we refer to \S \ref{SHAPP}.

The  $L^1$ of $Y^N_0$  norm can be derived from
the asymptotics of Legendre polynomials
$$P_N(\cos \theta) = \sqrt{2} (\pi N \sin \theta)^{-\half} \cos
\left( (N + \half) \theta - \frac{\pi}{4} \right) + O(N^{-3/2}) $$
where the remainder is uniform on any interval $\epsilon < \theta
< \pi - \epsilon$. We have
$$||Y^N_0||_{L^1} = 4 \pi  \sqrt{\frac{(2 N +
1)}{2 \pi}} \int_0^{\pi/2} |P_N(\cos r)| dv(r) \sim C_0 > 0,$$
i.e. the $L^1$ norm is asymptotically a positive constant. Hence
$\int_{Z_{Y^N_0}}  |\nabla Y^N_0| ds \simeq C_0 N^2 $. In this
example $|\nabla Y_0^N|_{L^{\infty}} = N^{\frac{3}{2}}$ saturates
the sup norm bound.  The length of the nodal line of $Y_0^N$ is of order  $\lambda$, as one
sees from the rotational invariance and by the fact that $P_N$ has
$N$ zeros. The defect in the argument is that the bound  $|\nabla
Y_0^N|_{L^{\infty}} = N^{\frac{3}{2}}$ is only obtained on the
nodal components near the poles, where each component has  length
$\simeq \frac{1}{N}$.
\bigskip

\begin{mainex} Calculate the $L^1$ norms of ($L^2$-normalized) zonal spherical harmonics and
Gaussian beams.

\end{mainex}

The left image is a zonal spherical harmonic of degree $N$ on $S^2$: it has high peaksof height $\sqrt{N}$  at the north and
south poles. The right image is a Gaussian beam: its height along the equator is $N^{1/4}$ and then it has Gaussian
decay transverse to the equator.

\begin{center}
\includegraphics[scale=0.9]{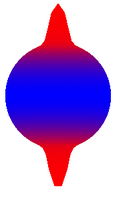}\includegraphics[scale=0.9]{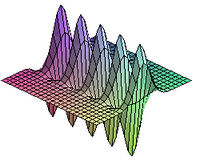}
\end{center}

\noindent{\bf Gaussian beams} \medskip

Gaussian beams are Gaussian shaped lumps which are concentrated on
$\lambda^{-\half}$ tubes $\tcal_{\lambda^{- \half}}(\gamma)$
around closed geodesics and have height $\lambda^{\frac{n-1}{4}}$.
We note that their $L^1$ norms decrease like
$\lambda^{-\frac{(n-1)}{4}}$, i.e. they saturate the $L^p$ bounds of \cite{Sog} for
small $p$.   In such cases we have $\int_{Z_{\phi_{\lambda}}}
|\nabla \phi_{\lambda}| dS \simeq \lambda^2
||\phi_{\lambda}||_{L^1} \simeq \lambda^{2 - \frac{n-1}{4}}. $ It
is likely that Gaussian beams are minimizers of the $L^1$ norm
among $L^2$-normalized eigenfunctions of Riemannian manifolds.
Also, the gradient bound $||\nabla \phi_{\lambda}||_{L^{\infty}} =
O(\lambda^{\frac{n + 1}{2}})$ is far off for Gaussian beams, the
correct upper bound being $\lambda^{1 + \frac{n-1}{4}}$.  If we
use these estimates on $||\phi_{\lambda}||_{L^1}$ and $||\nabla
\phi_{\lambda}||_{L^{\infty}}$,  our method gives
$\hcal^{n-1}(Z_{\phi_{\lambda}}) \geq C \lambda^{1 -
\frac{n-1}{2} }$,  while $\lambda$ is the correct lower bound for
Gaussian beams in the case of surfaces of revolution (or any real
analytic case). The defect is again that the gradient estimate is
achieved only very close to the closed geodesic of the Gaussian
beam. Outside of the tube  $\tcal_{\lambda^{- \half}}(\gamma)$ of
radius $\lambda^{- \half}$  around the geodesic, the Gaussian beam
and all of its derivatives decay like $e^{- \lambda d^2}$ where
$d$ is the distance to the geodesic. Hence
$\int_{Z_{\phi_{\lambda}}} |\nabla \phi_{\lambda}| dS \simeq
\int_{Z_{\phi_{\lambda}} \cap \tcal_{\lambda^{-
\half}}(\gamma)} |\nabla \phi_{\lambda}| dS. $ Applying  the
gradient bound for Gaussian beams  to the latter integral
 gives $\hcal^{n-1}(Z_{\phi_{\lambda}} \cap
\tcal_{\lambda^{- \half}}(\gamma)) \geq C \lambda^{1 -
\frac{n-1}{2}}$, which is sharp since the intersection
$Z_{\phi_{\lambda}} \cap \tcal_{\lambda^{- \half}}(\gamma)$
cuts across $\gamma$ in $\simeq \lambda$ equally spaced points (as
one sees from the Gaussian beam approximation).

\section{\label{QERsect} Quantum ergodic restriction theorem for Dirichlet or Neumann data}

QER (quantum ergodic restriction) theorems for Dirichlet data assert the quantum ergodicity of restrictions
$\phi_j |_H$
of eigenfunctions or their normal derivatives to hypersurfaces $H \subset M$.
In this section we briefly review the QER theorem for hypersurfaces of
\cite{TZ2, CTZ}. For lack of space, we must assume the reader's familiarity with quantum ergodicity
on the global manifold $M$. We refer to \cite{Ze2,Ze3,Ze6,Zw} for recent expositions.

\subsection{Quantum ergodic restriction theorems for Dirichlet data}

Roughly speaking, the QER theorem for Dirichlet data says that restrictions of eigenfunctions
to hypersurfaces $H \subset M$ for $(M, g)$ with ergodic geodesic flow  are quantum ergodic along $H$ as
long as $H$ is {\it asymmetric} for the geodesic flow. Here we note that a  tangent vector $\xi$ to $H$ of length $\leq 1$
is the projection to $T H$ of two unit tangent vectors $\xi_{\pm}$  to $M$. The $\xi_{\pm}   = \xi + r \nu$ where
$\nu$ is the unit normal to $H$ and $|\xi|^2 + r^2 = 1$. There are two possible signs of $r$ corresponding to
the two choices of ``inward'' resp. ``outward" normal. Asymmetry of $H$ with respect to the geodesic flow
$G^t$ means that the two orbits $G^t(\xi_{\pm})$ almost never return at the same time to the same place on $H$.
A generic hypersurface is asymmetric \cite{TZ2}. We refer to \cite{TZ2} (Definition 1) for the precise definition
of ``positive measure of microlocal reflection symmetry" of $H$. By asymmetry we mean that this measure is zero.

 We  write $h_j = \lambda_j^{-\half}$
and employ the calculus of semi-classical pseudo-differential operators \cite{Zw} where the
pseudo-differential
operators on $H$ are denoted by   $a^w(y, h D_y)$  or  $Op_{h_j}(a)$. The unit co-ball bundle of $H$ is denoted
by $B^* H$.

  \begin{theorem} \label{sctheorem} Let $(M, g)$ be a compact surface  with ergodic geodesic flow, and let  $H \subset
      M$ be a closed curve which is {\it asymmetric} with respect  to the geodesic flow.  Then
 there exists a  density-one subset $S$ of ${\mathbb N}$ such that
  for $a \in S^{0,0}(T^*H \times [0,h_0)),$
$$ \lim_{j \rightarrow \infty; j \in S} \langle Op_{h_j}(a)
 \phi_{h_j}|_{H},\phi_{h_j}|_{H} \rangle_{L^{2}(H)} = \omega(a), $$
 where
 $$  \omega(a) = \frac{4}{ vol(S^*M) } \int_{B^{*}H}  a_0( s, \sigma )  \,  (1 - |\sigma|^2)^{-\half}  \, ds d\sigma.$$
In particular this holds for multiplication operators $f$.
\end{theorem}

There is a similar result for normalized Neumann data.
  The normalized  Neumann  data of an eigenfunction along $H$ is denoted by
 \begin{equation}
\lambda_j^{-\half} D_{\nu} \phi_j |_{H}.
\end{equation}
Here, $ D_{\nu} = \frac{1}{i} \partial_{\nu}$ is a fixed choice of unit normal derivative.

We   define the microlocal lifts of the
Neumann data as the  linear functionals on semi-classical symbols $a \in S^{0}_{sc}(H)$ given by
$$\mu_h^N(a): = \int_{B^* H} a  \, d\Phi_h^N : = \langle Op_{H}(a) h D_{\nu} \phi_h
|_{H}, h D_{\nu} \phi_h |_{H}\rangle_{L^2(H)}.  $$

 \begin{theorem} \label{ND} \label{sctheoremNeu} Let $(M, g)$ be a compact surface  with ergodic geodesic flow, and let  $H \subset
      M$ be a closed curve which is {\it asymmetric} with respect  to the geodesic flow.  Then
 there exists a  density-one subset $S$ of ${\mathbb N}$ such that
  for $a \in S^{0,0}(T^*H \times [0,h_0)),$
$$ \lim_{h_j \rightarrow 0^+; j \in S} \mu_h^N(a) \to  \omega(a), $$
 where
 $$  \omega(a) = \frac{4}{ vol(S^*M) } \int_{B^{*}H}  a_0( s, \sigma )  \,  (1 - |\sigma|^2)^{\half}  \, ds d\sigma.$$
In particular this holds for multiplication operators $f$.
\end{theorem}

\subsection{\label{CDsect}  Quantum ergodic restriction theorems for Cauchy data}

Our
application is to the hypersurface $H = \mbox{Fix}(\sigma)$  \eqref{H} given by the fixed point set of the isometric
involution $\sigma$.  Such a hypersurface (i.e. curve)  fails to be asymmetric. 
 However there is a  quantum ergodic restriction theorem for Cauchy
data in \cite{CTZ}  which does apply and shows that the even eigenfunctions are quantum ergodic
along  $H$, hence along each component $\gamma$.

  The normalized Cauchy data of an eigenfunction along $\gamma$ is denoted by
 \begin{equation} \label{CD} CD(\phi_h)  := \{(\phi_h |_{\gamma}, \;
h D_{\nu} \phi_h |_{\gamma}) \}.
\end{equation}
Here, $ D_{\nu}$ is a fixed choice of unit normal derivative. The first component of the Cauchy
data is called the Dirichlet data and the second is called the Neumann data.

\begin{theorem} \label{useful} Assume that $(M, g)$ has an orientation reversing  isometric involution with
separating   fixed point set $H$. Let  $\gamma$ be a component of $H$.  Let
$\phi_{h}$ be the sequence of even ergodic eigenfunctions. Then,

$$\begin{array}{l}
 \langle Op_{\gamma}(a)  \phi_{h} |_{\gamma}, \phi_{h} |_{\gamma}
\rangle_{L^2(\gamma)} \\ \\ \rightarrow_{h \to 0^+} \frac{4}{ 2 \pi \mbox{Area}(M)} \int_{B^*\gamma} a_0(s,\sigma) (1 - | \sigma |^2)^{-1/2} d s d \sigma.
\end{array}$$
In particular, this holds when $Op_{\gamma}(a)$ is multiplication by a smooth function $f$.

\end{theorem}
Here we use the semi-classical notation $h_j = \lambda_{\phi}^{-\frac{1}{4}}$. 
ergodic along $\gamma$, but we do not use this result here.
We refer to \cite{TZ2, CTZ,Zw} for background and for  notation concerning pseudo-differential operators.

We further   define the microlocal lifts of the
Neumann data as the  linear functionals on semi-classical symbols $a \in S^{0}_{sc}(\gamma)$ given by
$$\mu_h^N(a): = \int_{B^* \gamma} a  \, d\Phi_h^N : = \langle Op_{\gamma}(a) h D_{\nu} \phi_h
|_{\gamma}, h D_{\nu} \phi_h |_{\gamma}\rangle_{L^2(\gamma)}.  $$
We
also  define the {\it renormalized microlocal lifts} of the Dirichlet
data by
$$\mu_h^D(a): = \int_{B^*\gamma } a \, d\Phi_h^{RD} : = \langle Op_{\gamma}(a) (1 +
h^2 \Delta_{\gamma}) \phi_{h} |_{\gamma},  \phi_{h}|_{\gamma} \rangle_{L^2(\gamma)}.
$$
Here, $h^2 \Delta_{\gamma}$ denotes the negative
tangential Laplacian $- h^2 \frac{d^2}{ds^2} $  for the induced metric on $\gamma$, so that the symbol $1 - |\sigma|^2$ of
the operator $(1+h^2 \Delta_{\gamma})$ vanishes on the tangent directions
 $S^*\gamma$ of $\gamma$.
Finally, we  define the microlocal lift $d \Phi_h^{CD}$ of the Cauchy data  to be the sum
\begin{equation} \label{WIGCD} d \Phi_h^{CD} := d \Phi_h^N + d \Phi_h^{RD}. \end{equation}

Let $B^* \gamma$ denote  the unit ``ball-bundle'' of $\gamma$ (which is the interval
$\sigma \in (-1,1)$ at each point $s$), $s$ denotes arc-length along $\gamma$ and $\sigma$ is the dual
symplectic coordinate. The first result of \cite{CTZ} relates QE (quantum ergodicity) on $M$ to quantum ergodicity
on a hypersurface $\gamma$. A sequence of eigenfunctions is QE globally on M if
$$\langle A \phi_{j_k}, \phi_{j_k} \rangle \to \frac{1}{\mu(S^*M)} \int_{S^* M} \sigma_A d \mu, $$
where $d\mu_L$ is Liouville measure, i.e. the measure induced on the co-sphere bundle by 
the symplectic volume measure and the Hamiltonian $H(x, \xi) = |\xi|_g$. 

\begin{theorem}
Assume that $\{\phi_h\}$ is a quantum ergodic sequence of eigenfunctions on $M$.  Then the sequence
$\{d \Phi_{h}^{CD} \}$ \eqref{WIGCD} of microlocal lifts of the Cauchy data of $\phi_h$ is quantum ergodic on $\gamma$ in the sense that for any
 $a \in S^0_{sc}(\gamma),$
$$\begin{array}{l}
\langle Op_H(a) h D_\nu \phi_h |_{\gamma} ,  h D_\nu \phi_h |_{\gamma} \rangle_{L^2(\gamma)}  + \langle Op_{\gamma}(a) (1 +
 h^2 \Delta_{\gamma}) \phi_{h} |_{\gamma}, \phi_{h} |_{\gamma}
\rangle_{L^2(\gamma)} \\ \\ \rightarrow_{h \to 0^+} \frac{4}{\mu(S^*
 M)} \int_{B^*\gamma} a_0(s, \sigma) (1 - | \sigma |^2)^{1/2} d s d\sigma
\end{array}$$
where $a_0$ is the principal symbol of $Op_{\gamma}(a)$.
\end{theorem}

When applied to even eigenfunctions under an orientation-reversing isometric involution with separating fixed
point set, the Neumann data vanishes, and we obtain
\begin{corollary} \label{COROLLARY} Let $(M,g)$ have an  orientation-reversing isometric involution with separating fixed point set
$H$ and let $\gamma$ be one of its components.
 Then for any  sequence of  even quantum ergodic eigenfunctions of $(M, g)$,

$$\begin{array}{l}
 \langle Op_{\gamma}(a) (1 +
 h^2 \Delta_{\gamma}) \phi_{h} |_{\gamma}, \phi_{h} |_{\gamma}
\rangle_{L^2(\gamma)} \\ \\ \rightarrow_{h \to 0^+} \frac{4}{\mu(S^*
 M)} \int_{B^*\gamma} a_0(s, \sigma) (1 - | \sigma |^2)^{1/2} d s d\sigma
\end{array}$$

\end{corollary}

For applications to zeros along $\gamma$,   we need a  limit formula for the
integrals $\int_{\gamma} f \phi_h^2 ds$, i.e.  a quantum ergodicity result for 
for Dirichlet data. We   invert the operator $(1 + h^2 \Delta_{\gamma})$ and obtain


\begin{theorem} \label{thm2}

Assume that $\{\phi_h\}$ is a quantum ergodic sequence on $M$.  Then,  there exists a sub-sequence
of density one as $h \to 0^+$ such that for all  $a \in S^{0}_{sc}(\gamma)$,
$$\begin{array}{ll} 
&\left< (1 + h^2 \Delta_{\gamma} + i0)^{-1} Op_{\gamma}(a)  h D_\nu \phi_h |_{H} , h D_\nu
\phi_h |_{\gamma} \right>_{L^2(\gamma)} + \left< Op_{\gamma}(a)  \phi_{h} |_{\gamma}, \phi_{h} |_{\gamma}
\right>_{L^2(\gamma)} \\
&\rightarrow_{h \to 0^+} \frac{4}{ 2 \pi \mbox{Area}(M)} \int_{B^*\gamma} a_0(s,\sigma) (1 - | \sigma |^2)^{-1/2} d s d \sigma.
\end{array} $$
 \end{theorem}

Theorem \ref{useful} follows from Theorem \ref{thm2} since the Neumann term drops out (as before) under
the hypothesis of Corollary \ref{COROLLARY}.

\section{\label{NODALINTSECT} Counting intersections of nodal sets and geodesics}

As discussed in the introduction \S \ref{QERINTRO}, the QER results can be used to obtain results
on intersections of nodal sets with geodesics in dimension two. In general, we do not know how to use
interesection results to obtain lower bounds on numbers of nodal domains unless we assume a symmetry
condition on the surface. But  begin with general results on intersection that do not assume any symmetries.

\begin{theorem}\label{theoN}
Let $(M, g) $ be a   $C^{\infty} $ compact negatively curved surface, and let $H$ be a closed curve which is asymmetric with respect to the geodesic flow. Then for any orthonormal eigenbasis $\{\phi_j\}$ of $\Delta$-eigenfunctions of $(M, g)$, there exists  a density $1$ subset $A$ of $\mathbb{N}$ such that
\[
\left\{\begin{array}{l}
\lim_{\substack{j \to \infty \\ j \in A}} \# \; \ncal_{\phi_j} \cap H = \infty
\\ \\
\lim_{\substack{j \to \infty \\ j \in A}} \# \;  \{x \in H: \partial_{\nu} \phi_j(x) = 0\} = \infty.
\end{array} \right.
\]
Furthermore, there are an infinite number of zeros where  $\phi_j |_H$ (resp. $\partial_{\nu} \phi_j |_H$) changes sign.
\end{theorem}

We now add the assumption of a symmetry as discussed in the introduction in \S \ref{INTERINTRO}.

\begin{theorem}\label{theoS}
Let $(M, g) $ be a compact negatively curved  $C^{\infty} $ surface with an orientation-reversing isometric involution $\sigma : M \to M$ with $\mbox{Fix}(\sigma)$ separating. Let $\gamma \subset \mbox{Fix}(\sigma)$.    Then for any orthonormal eigenbasis $\{\phi_j\}$ of $L_{even}^2(M)$, resp. $\{\psi_j\}$ of $L^2_{odd}(M)$,  one can find a density $1$ subset $A$ of $\mathbb{N}$ such that
\[\left\{ \begin{array}{l}
\lim_{\substack{j \to \infty \\ j \in A}} \# \; \ncal_{\phi_j} \cap \gamma = \infty\\ \\
\lim_{\substack{j \to \infty \\ j \in A}} \# \; \Sigma_{\psi_j} \cap \gamma = \infty.
\end{array} \right.\]
Furthermore, there are an infinite number of zeros where  $\phi_j |_H$ (resp. $\partial_{\nu} \psi_j |_H$) changes sign.

\end{theorem}

We now sketch the proof.

\subsection{Kuznecov sum formula on surfaces}

The first step is to use an old result \cite{Ze9} on the asymptotics of the `periods' $\int_{\gamma} f \phi_j ds$ of eigenfunctions
over closed geodesics when $f$ is a smooth function.

\begin{theorem}\label{K}  \cite{Ze9} (Corollary 3.3)  Let $f \in C^{\infty}(\gamma)$. Then
there exists a constant $c>0$ such that,
\[
\sum_{\lambda_j < \lambda}\left|\int_{\gamma} f \phi_j ds\right|^2 = c\left|\int_{\gamma} f ds\right|^2 \sqrt{\lambda} + O_f(1).
\]
\end{theorem}

There is a similar result for the normal derivative  $\partial_{\nu} $ of eigenfunctions along $\gamma$.

\begin{theorem}\label{KN}  Let $f \in C^{\infty}(\gamma)$. Then
there exists a constant $c>0$ such that,
\[
\sum_{\lambda_j < \lambda}\left|\lambda_j^{-1/2} \int_{\gamma} f \partial_{\nu} \phi_j ds\right|^2 = c\left|\int_{\gamma} f ds\right|^2 \sqrt{\lambda} + O_f(1).
\]
\end{theorem}

These `Kuznecov sum formulae' do not imply individual results about  asymptotic periods of the full
sequence of  eigenfunction. However, because
the terms are positive,  there must  exists a subsequence of eigenfunctions $\phi_j$ of
natural density one so that, for all  $f \in C^{\infty}(\gamma)$,
\begin{equation} \label{EST} \left\{ \begin{array}{l}
\left|\int_{\gamma} f \phi_j ds\right| \\ \\
\lambda_j^{-\half} \left|\int_{\gamma} f \partial_{\nu} \phi_j ds\right|  \end{array} \right.  =O_f( \lambda_j^{-1/4} (\log \lambda_j )^{1/2})
\end{equation}

Indeed, this follows by Chebychev's inequality. Denote by $N(\lambda)$ the number of eigenfunctions in $\{j~|~\lambda<\lambda_j<2\lambda \}$. Then for each $f$, 

\[
\frac{1}{N(\lambda)}|\{j~|~\lambda<\lambda_j<2\lambda,~\left|\int_{\gamma_i} f \phi_j ds\right|^2 \geq \lambda_j^{-1/2}\log \lambda_j \}| = O_f(\frac{1}{\log \lambda}).
\]
It follows that the  upper density of exceptions to \eqref{EST}  tends to zero. We then choose a countable
dense set $\{f_n\}$ and apply the diagonalization argument of \cite{Ze4} (Lemma 3)  or \cite{Zw} Theorem 15.5 step (2)) to conclude that there exists a density
one subsequence for which \eqref{EST} holds for all $f \in C^{\infty}(\gamma)$. The same holds for the normal
derivative.

We then use the argument sketched in \S \ref{QERINTRO} of the introduction. If we combine \eqref{EST}
with the QER result \eqref{qer},  we conclude  that a full density
subsequence of eigenfunctions of a negatively curved surface  must have  an unbounded number of sign changing zeros along $\gamma$. In the asymmetric case we use Theorem \ref{sctheorem} while in the symmetric case we use Theorem
\ref{thm2}  in the QER step.
\bigskip

\noindent{\bf Problem} The QER step and the Kuznecov extend to manifolds with boundary (although the Kuznecov
sum formula of \cite{Ze9} has so far only been proved in the boundaryless case). But the main obstacle to 
generalizing the results on intersections and nodal domains is that the logarithmic improvement in \eqref{NEGSUP}
has not been generalized to the boundary problems with chaotic billiards (to the author's knowledge).

\section{Counting nodal domains}

We now sketch the proof of Theorem \ref{theoJJ}. Granted that there are many zeros of even eigenfunctions
along each component $\gamma$ of $\mbox{Fix}(\sigma)$ and many zeros of normal derivatives of odd eigenfunctions,
the remaining point is to relate such zeros to nodal domains. The nodal sets cross the geodesic transversally but
may link up in complicated ways far from the separating set.  In dimension two, use use the Euler inequality
for graphs to relate numbers of intersection points and numbers of nodal domains. 
\bigskip

\noindent{\bf Problem}: The QER and Kuznecov sum formula results are valid in all dimensions. It is only
the topological step which is simpler in dimension two. Can one find any generalizations of this argument 
to higher dimensions?
\bigskip

In dimension two we can construct an embedded graph from the nodal set
  $\ncal_{\phi_{\lambda}}$ as follows.
\begin{enumerate}
\item For each embeded circle which does not intersect $\gamma$, we add a vertex.
\item Each singular point is a vertex.
\item If $\gamma \not\subset \ncal_{\phi_\lambda}$, then each intersection point in $\gamma \cap \ncal_{\phi_\lambda}$ is a vertex.
\item Edges are the arcs of $\ncal_{\phi_\lambda}$ ($\ncal_{\phi_\lambda} \cup \gamma$, when $\phi_\lambda$ is even) which join the vertices listed above.
\end{enumerate}

We thus obtain a graph  embeded into the surface $M$.
The {\it faces} $f$ of $G$ are the  connected components of $M \backslash V(G) \cup \bigcup_{e \in E(G)} e$.
The set of faces is denoted $F(G)$. An edge $e \in E(G)$ is {\it incident} to $f$ if the boundary of $f$ contains
an interior point of $e$. Every edge is incident to at least one and to at most two faces; if $e$ is incident
to $f$ then $e \subset \partial f$. The faces are not assumed to be cells and the sets $V(G), E(G), F(G)$ are
not assumed to form a CW complex.

Now let $v(\phi_\lambda)$ be the number of vertices, $e(\phi_\lambda)$ be the number of edges, $f(\phi_\lambda)$ be the number of faces, and $m(\phi_\lambda)$ be the number of connected components of the graph. Then by Euler's formula (Appendix F, \cite{g}),
\begin{equation}\label{euler}
v(\phi_\lambda)-e(\phi_\lambda)+f(\phi_\lambda)-m(\phi_\lambda) \geq 1- 2 g_M
\end{equation}
where $g_M$ is the genus of the surface.
We use this inequality to give a lower bound for the number of nodal domains for even and odd eigenfunctions.

In the odd case, we get a lower bound using the large number of singular points of the odd eigenfunctions. Note
that since an odd eigenfunction vanishes on $\mbox{Fix}(\sigma)$, the points where $\partial_{\nu} \psi_j = 0$
belong to the singular set. 

We claim that for  an odd eigenfunction $\psi_j$,
\[
N(\psi_j) \geq \#\left(\Sigma_{\psi_j}\cap \gamma\right) +2 - 2g_M,
\]

Proof: for an odd eigenfunction $\psi_j$, $\gamma \subset \ncal_{\psi_j}$.  Therefore $f(\psi_j)=N(\psi_j)$.
Let $n(\psi_j)=\#\Sigma_{\psi_j}\cap \gamma$ be the number of singular points on $\gamma$. These points correspond to vertices having degree at least $4$ on the graph, hence
\begin{equation} \begin{array}{ll}
0&= \sum_{x:vertices} \mathrm{deg}(x) -2e(\psi_j) \\ \\
&\geq 2\left(v(\psi_j)-n(\psi_j)\right)+4 n(\psi_j)-2e(\psi_j).
\end{array} \end{equation}
Therefore
\[
e(\psi_j)-v(\psi_j) \geq n(\psi_j),
\]
and plugging into \eqref{euler} with $m(\psi_j)\geq 1$, we obtain
\[
N(\psi_j) \geq n(\psi_j) +2 - 2g_M.
\]

Now consider  an even eigenfunction $\phi_j$.  In this case (following the idea of \cite{GRS}), we relate
the number of intersection points of the nodal set with $\gamma$ to the number of `inert nodal domains', i.e.
nodal domains which are invariant under the isometric involution, $\sigma U=U.$ 
We claim that,
\[
N(\phi_j) \geq \frac{1}{2}\#\left(\ncal_{\phi_j} \cap \gamma\right)+1-g_M.
\]

Proof: For an even eigenfunction $\phi_j$, let $N_{in}(\phi_j)$ be the number of  `inert' nodal domain $U$.  Let $N_{sp}(\phi_j)$ be the number of the rest (split nodal domains). From the assumption that $Fix(\sigma)$ is separating, inert domains   intersect $\gamma$ in a finite non-empty  union
of 
relative open intervals,  and $\mbox{Fix}(\sigma)$ divides each inert nodal domain into two connected components. 
Hence each inert nodal domain may correspond to two faces on the graph, depending on whether the nodal domain intersects $\gamma$ or not. Therefore $f(\phi_j)\leq 2N_{in}(\phi_j)+N_{sp}(\phi_j)$. 

Each point in $\ncal_{\phi_j} \cap \gamma$ corresponds to a vertex having degree at least $4$ on the graph.
The nodal intersections are non-singuar points.  Hence by the same reasoning as the odd case, we have
\[
N(\phi_j) \geq N_{in}+\frac{1}{2}N_{sp}(\phi_j) \geq \frac{f(\phi_j)}{2}\geq \frac{n(\phi_j)}{2} +1 - g_M
\]
where $n(\phi_j)=\#\ncal_{\phi_j} \cap \gamma$.

This concludes the proof of Theorem \ref{theoJJ}.

\section{Analytic continuation of eigenfuntions to the  complex domain} 

We next  discuss three results that use analytic continuation of eigenfunctions
to the complex domain.  First is the Donnelly-Fefferman volume bound Theorem \ref{DF}. 
We sketch a somewhat simplified proof which will appear in more detail in \cite{Ze0}. Second we discuss
the equidistribution theory of nodal sets in the complex domain in the ergodic case \cite{Ze5} and in the
completely integrable case \cite{Ze8}. Third, we discuss nodal intersection bounds. This includes bounds
on the number of nodal lines intersecting the boundary in \cite{TZ} for the Dirichlet or Neuman problem
in a plane domain, the number (and equi-distribution) of nodal intersections with geodesics in the complex
domain \cite{Ze6} and results on nodal intersections and nodal domains for the modular surface 

\subsection{Grauert tubes} 

. As examples, we have:

\begin{itemize}

\item  $M = \R^m/\Z^m$ is $M_{\C} = \C^m/\Z^m$.

\item   The unit sphere $S^n$ defined by  $x_1^2 + \cdots +
x_{n+1}^2 = 1$ in $\R^{n+1}$ is complexified as the complex
quadric $S^2_{\C} = \{(z_1, \dots, z_n) \in \C^{n + 1}: z_1^2 +
\cdots + z_{n+1}^2 = 1\}. $

\item  The hyperboloid model of hyperbolic space is the
hypersurface in $\R^{n+1}$ defined by
$$\Hh^n = \{ x_1^2 + \cdots
x_n^2 - x_{n+1}^2 = -1, \;\; x_n > 0\}. $$ Then,
$$H^n_{\C} = \{(z_1, \dots, z_{n+1}) \in \C^{n+1}:  z_1^2 + \cdots
z_n^2 - z_{n+1}^2 = -1\}. $$

\item Any real algebraic subvariety of $\R^m$ has a similar
complexification.

\item Any  Lie group $G$ (or symmetric space) admits a
complexification $G_{\C}$.
\end{itemize}

Let us consider examples of holomorphic continuations of
eigenfunctions:

\begin{itemize}

\item On  the flat torus $\R^m/\Z^m,$   the real eigenfunctions
are $\cos \langle k, x \rangle, \sin \langle k, x \rangle$ with $k
\in 2 \pi \Z^m.$ The complexified torus is $\C^m/\Z^m$ and the
complexified eigenfunctions are $\cos \langle k, \zeta \rangle,
\sin \langle k, \zeta \rangle$ with $\zeta  = x + i \xi.$

\item On the unit sphere $S^m$, eigenfunctions are restrictions of
homogeneous harmonic functions on $\R^{m + 1}$. The latter extend
holomorphically to holomorphic harmonic polynomials on $\C^{m +
1}$ and restrict to holomorphic function on $S^m_{\C}$.

\item On $\H^m$, one may use the hyperbolic plane waves $e^{ (i
\lambda + 1) \langle z, b \rangle}$, where  $\langle z, b \rangle$
is the (signed) hyperbolic distance of the horocycle passing
through $z$ and $b$ to $0$. They  may be holomorphically extended
to the maximal tube of radius $\pi/4$.

\item
 On  compact hyperbolic quotients ${\bf
H}^m/\Gamma$, eigenfunctions can be then represented by Helgason's
generalized Poisson integral formula \cite{H},
$$\phi_{\lambda}(z) = \int_B e^{(i \lambda + 1)\langle z, b
\rangle } dT_{\lambda}(b). $$ Here, $z \in D$ (the unit disc), $B
=
\partial D$, and $dT_{\lambda}
\in \dcal'(B)$ is the boundary value of $\phi_{\lambda}$, taken in
a weak sense along circles centered at the origin $0$.
 To  analytically continue $\phi_{\lambda}$ it suffices  to
analytically continue $\langle z, b\rangle. $ Writing the latter
as  $\langle \zeta, b \rangle,  $ we have:
\begin{equation} \label{HEL} \phi_{\lambda}^{\C} (\zeta) = \int_B e^{(i \lambda +
1)\langle \zeta, b \rangle } dT_{\lambda}(b). \end{equation}

\end{itemize}

 The ($L^2$-normalizations of the) modulus
squares 
\begin{equation} \label{HUSIMI} |\phi_j^{\C}(\zeta)|^2: M_{\epsilon}  \to \R_+ \end{equation}
are sometimes known as Husimi functions. They are holomorphic extensions
of $L^2$-normalized functions but are not themselves $L^2$ normalized on $M_{\epsilon}$.
However, as will be discussed below, their $L^2$ norms may on the Grauert tubes
(and their boundaries) can be determined.  One can then ask how the  mass
of the normalized Husimi function is distributed in phase space, or how the $L^p$ norms
behave.

\begin{center}
\includegraphics[scale=0.8]{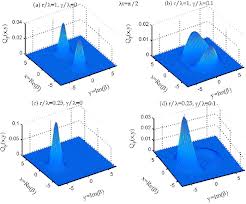} 
\end{center}

\subsection{Weak * limit problem for Husimi measures in the complex domain}
One of the general problems of quantum dynamics is to determine  all of the weak* limits of the sequence, 

$$\{\frac{|\phi_j^{\C}(z)|^2}{||\phi_j^{\C}||_{L^2(\partial M_{\epsilon})}}  d\mu_{\epsilon} \}_{j = 1}^{\infty}.$$
Here, $d\mu_{\epsilon}$ is the natural measure on $\partial M_{\epsilon}$ corresponding to the contact volume
form on $S^*_{\epsilon} M$.  Recall that a sequence $\mu_n$ of probability measures on a compact space $X$ is said to converge weak* to a measure $\mu$
if $\int_X f d\mu_n \to \int_X f d\mu$ for all $f \in C(X)$. 
We refer to Theorem \ref{W*HUSIMIERG}  for the ergodic case. In the integrable case one has localization
results, which are not presented here.

\subsection{\label{PSH} Background on currents and PSH functions}

We next consider logarithms of Husimi functions, which are PSH = (pluri-subharmonic) functions on $M_{\epsilon}$. 
A  function $f$ on a domain in a complex manifold is PSH if
$i \ddbar f $ is a {\it positive (1,1) current}.  That is, $i \ddbar f$ is a singular form of type $\sum_{i \bar{j}}
a_{i \bar{j}} dz^i \wedge d \bar{z}^k$ with $(a_{j \bar{k}})$ postive definte Hermitian.
If $f$ is a local holomorphic function, then $\log |f(z)|$ is PSH and $i \ddbar \log |f(z)| = [Z_f]$.  General
references are \cite{GH,HoC}.

In Theorem \ref{ERGCXZ}, 
 we regard the zero set  $[Z_f] $ as a {\it current of integration}, i.e. as a
linear functional on $(m-1, m-1)$ forms $\psi$
$$\langle [Z_{\phi_j} ], \psi \rangle = \int_{Z_{\phi_j} } \psi. $$
Recall that a {\it  current} is a linear functional (distribution) on smooth forms. We refer to \cite{GH} for 
background. On a complex manifold one has
$(p,q)$ forms with $p$ $dz_j$ and $q$ $d \bar{z}_k$'s. 
In \eqref{MEAS} we  use the K\"ahler hypersurface volume form $\omega_g^{m-1}$ (where $\omega_g = i \ddbar \rho$) to make $Z_{\phi_j} $ into a measure:
$$\langle [Z_{\phi_j} ], f \rangle = \int_{Z_{\phi_j} } f \omega_g^{m-1}, \;\;\; (f \in C(M)). $$

A sequence of $(1,1)$ currents $E_k$ converges weak* to a current $E$ if 
$\langle E_k, \psi \rangle \to \langle E, \psi \rangle$ for all smooth $(m-1, m-1)$ forms. Thus, for all $f$
$$ [Z_{\phi_j} ]  \rightarrow i \ddbar \sqrt{\rho} \iff \int_{Z_{\phi_{j}} } f \omega^{m-1} \to i \int_{M_{\epsilon}}
f  \ddbar \sqrt{\rho} \wedge \omega^{m-1, m-1}. $$

\subsection{\label{PLLSEC} Poincar\'e-Lelong formula}

One of the two key reasons for the gain in simplicity is that there exists a simple
analytical formula for the delta-function on the nodal set. 
The  {\it Poincar\'e-Lelong formula}  gives an exact formula  for the delta-function on the zero set of $\phi_j$ 
\begin{equation} \label{PLLb} i \ddbar \log  |\phi_j^{\C}(z)|^2 = [\ncal_{\phi_j^{\C}}]. \end{equation}
\bigskip
Thus, if $\psi $ is an $(n-1, n-1)$ form,
$$ \int_{\ncal_{\phi_j^{\C}}} \psi = \int_{M_{\epsilon}}  \psi \wedge i \ddbar \log  |\phi_j^{\C}(z)|^2. $$

\subsection{Pluri-subharmonic functions and compactness}

In the real domain, we have emphasized the problem of finding weak* limits
of the probability measures \eqref{SQUARE} and of their microlocal lifts
or Wigner measures in phase space. The same problem exists in the complex
domain for the sequence of Husimi functions \eqref{HUSIMI}. However, there
also exists a new problem involving the sequence of normalized logarithms
\begin{equation} \label{LOGS}\{ u_j: = \frac{1}{\lambda_{j}} \log |\phi_{j}^{\C}(z)|^2\}_{j = 1}^{\infty}. \end{equation}
A key fact is that this sequence is pre-compact in $L^p(M_{\epsilon})$ for all $p < \infty$
and even that 
 \begin{equation}\{ \frac{1}{\lambda_{j}} \nabla \log |\phi_{j}^{\C}(z)|^2\}_{j = 1}^{\infty}. \end{equation}  is pre-compact in $L^1(M_{\epsilon})$.

\begin{lem} \label{HARTOGS}  (Hartog's Lemma;  (see
\cite[Theorem~4.1.9]{HoI}): Let $\{v_j\}$ be a sequence of subharmonic functions in an
open set $X \subset \R^m$ which have a uniform upper bound on any
compact set. Then either $v_j \to -\infty$ uniformly on every
compact set, or else there exists a subsequence $v_{j_k}$ which is
convergent to some  $u \in L^1_{loc}(X)$. Further,  $\limsup_n
u_n(x) \leq u(x)$ with equality almost everywhere. For every
compact subset $K \subset X$ and every continuous function $f$,
$$\limsup_{n \to \infty} \sup_K (u_n - f) \leq \sup_K (u - f). $$
In particular, if $f \geq u$ and $\epsilon > 0$, then $u_n \leq f
+ \epsilon$ on $K$ for $n$ large enough. \end{lem}

\subsection{\label{LOGWEAK*} A general weak* limit problem for logarithms of Husimi functions}

The study of exponential  growth rates gives rise to a new kind new weak* limit problem  for complexified
eigenfunctions.

\begin{prob} Find the weak* limits $G$ on $M_{\epsilon}$ of  sequences
$$\frac{1}{\lambda_{j_k}} \log |\phi_{j_k}^{\C}(z)|^2 \to G?? $$
( The limits are actually in $L^1$ and not just weak. )
\end{prob}
\bigskip

See Theorems \ref{ALSOLOGLIM} ,  \ref{LOGLIMERG} and  \ref{ZEROWEAK}    for the solution to this problem 
(modulo sparse subsequences) in the  ergodic case.

Here is a  general Heuristic principle to pin down the possible $G$: 
If $\frac{1}{\lambda_{j_k} } \log |\phi_{j_k}^{\C}(z)|^2 \to G(z)$  then
$$ |\phi_{j_k}^{\C}(z)|^2 \simeq e^{\lambda_j G(z)} (1 + \mbox{SOMETHING SMALLER }) \; (\lambda_j \to \infty).$$
\bigskip

But $\Delta_{\C} |\phi_{j_k}^{\C}(z)|^2  = \lambda_{j_k}^2 |\phi_{j_k}^{\C}(z)|^2$, so we should have

\begin{conj} Any limit $G$ as above solves the Hamilton-Jacobi equation,
$$(\nabla_{\C} G)^2 = 1. $$
\end{conj}
\bigskip

(Note: The weak* limits of $\frac{|\phi_j^{\C}(z)|^2}{||\phi_j^{\C}||_{L^2(\partial M_{\epsilon})}}  d\mu_{\epsilon}$ must
be supported in $\{G = G_{\max}\}$ (i.e. in the set of maximum values).

\section{\label{POISSONSZEGOSECT} Poisson operator and Szeg\"o operators on Grauert tubes}

\subsection{Poisson operator and analytic Continuation of eigenfunctions}

The half-wave group of $(M, g)$ is the unitary group
$U(t)  = e^{i t \sqrt{\Delta}}$ generated by the square root of the positive Laplacian. Its Schwartz kernel
is a distribution on $\R \times M \times M$ with  the eigenfunction
expansion
\begin{equation} \label{Ut} U(t, x, y) = \sum_{j = 0}^{\infty} e^{i
t  \lambda_j} \phi_{ j}(x) \phi_j(y).  \end{equation}

By the Poisson operator we mean the analytic continuation of$U(t)$ to positive imaginary time,
\begin{equation} \label{POISSON} e^{- \tau \sqrt{\Delta}}  = U(i \tau) .\end{equation}
The eigenfunction expansion then converges absolutely to a real analytic function on $\R_+ \times M \times M$.

Let $A(\tau)$ denote the operator of analytic continuation of a
function on $M$ to the Grauert tube $M_{\tau}$. Since
\begin{equation} U_{\C} (i \tau) \phi_{\lambda} = e^{- \tau \lambda}
\phi_{\lambda}^{\C}, \end{equation} it  is simple to
see that \begin{equation}  \label{ATAU} A(\tau) = U_{\C}(i \tau) e^{\tau \sqrt{\Delta}} \end{equation} where
 $U_{\C}(i
\tau, \zeta, y)$ is the analytic continuation of the Poisson kernel in $x$ to
$M_{\tau}$. In terms of the eigenfunction expansion, one has
\begin{equation} \label{UI} U(i \tau, \zeta, y) = \sum_{j = 0}^{\infty} e^{-
\tau  \lambda_j} \phi_{ j}^{\C} (\zeta) \phi_j(y),\;\;\; (\zeta,
y) \in M_{\epsilon}  \times M.  \end{equation}  This is a very useful observation because $ U_{\C}(i \tau) e^{\tau \sqrt{\Delta}} $
is a Fourier integral operator with complex phase and can be related to the geodesic flow.  
The analytic continuability of
the Poisson operator to $M_{\tau}$  implies that  every
eigenfunction analytically continues to the same Grauert tube.

\subsection{Analytic continuation of the Poisson wave group}  The analytic continuation of the Possion-wave
kernel to $M_{\tau}$ in the $x$ variable is discussed in detail in \cite{Ze8} and ultimately derives from the
analysis by Hadamard of his parametrix construction.  We only briefly discuss  it here and refer to
\cite{Ze8} for further details.  In the case 
of Euclidean $\R^n$ and its wave kernel  $U(t, x, y) =
\int_{\R^n} e^{i t |\xi|} e^{i \langle \xi, x - y \rangle} d\xi$
which  analytically continues  to $t + i \tau, \zeta = x + i p \in
\C_+ \times \C^n$ as the integral
$$U_{\C} (t + i \tau , x + i p , y) = \int_{\R^n} e^{i (t + i \tau)  |\xi|} e^{i \langle \xi, x + i p - y
\rangle} d\xi. $$ The integral clearly converges absolutely for
$|p| < \tau.$

Exact formulae of this kind exist for $S^m$ and $\H^m$. For a
general real analytic Riemannian manifold, there exists an
oscillatry integral expression for the wave kernel of the form,
\begin{equation} \label{PARAONE} U (t, x, y) = \int_{T^*_y M} e^{i
t |\xi|_{g_y} } e^{i \langle \xi, \exp_y^{-1} (x) \rangle} A(t, x,
y, \xi) d\xi
\end{equation} where $A(t, x, y, \xi)$ is a polyhomogeneous amplitude of
order $0$.  The
holomorphic extension of (\ref{PARAONE}) to the Grauert tube
$|\zeta| < \tau$ in $x$ at time $t = i \tau$ then has the form
\begin{equation} \label{CXPARAONE} U_{\C} (i \tau,
\zeta, y) = \int_{T^*_y} e^{- \tau  |\xi|_{g_y} } e^{i \langle
\xi, \exp_y^{-1} (\zeta) \rangle} A(t, \zeta, y, \xi) d\xi
\;\;\;(\zeta = x + i p).
\end{equation}

\subsection{Complexified spectral projections}

The next step is  to  holomorphically extend the spectral projectors
$d\Pi_{[0, \lambda]}(x,y) = \sum_j \delta(\lambda - \lambda_j)
\phi_j(x) \phi_j(y) $ of $\sqrt{\Delta}$.
 The  complexified
diagonal  spectral projections
 measure is defined by
 \begin{equation} d_{\lambda} \Pi_{[0, \lambda]}^{\C}(\zeta, \bar{\zeta}) = \sum_j \delta(\lambda -
 \lambda_j) |\phi_j^{\C}(\zeta)|^2. \end{equation}
 Henceforth, we generally omit the superscript and write the
 kernel as $\Pi_{[0, \lambda]}^{\C}(\zeta, \bar{\zeta})$.
 This kernel is not a tempered distribution due to the exponential
 growth of $|\phi_j^{\C}(\zeta)|^2$. Since many asymptotic techniques
 assume spectral functions are of polynomial growth,  we simultaneously
 consider the damped spectral projections measure
  \begin{equation} \label{SPPROJDAMPED} d_{\lambda} P_{[0, \lambda]
  }^{\tau}(\zeta, \bar{\zeta}) = \sum_j \delta(\lambda -
 \lambda_j) e^{- 2 \tau \lambda_j} |\phi_j^{\C}(\zeta)|^2, \end{equation}
 which   is a temperate distribution as long as $\sqrt{\rho}(\zeta)
 \leq \tau. $ When we set $\tau = \sqrt{\rho}(\zeta)$ we omit the
 $\tau$ and put
  \begin{equation} \label{SPPROJDAMPEDz} d_{\lambda} P_{[0, \lambda]
  }(\zeta, \bar{\zeta}) = \sum_j \delta(\lambda -
 \lambda_j) e^{- 2 \sqrt{\rho}(\zeta) \lambda_j} |\phi_j^{\C}(\zeta)|^2. \end{equation}

The integral of the spectral measure over an interval $I$  gives
$$\Pi_{I}(x,y) = \sum_{j: \lambda_j \in I} \phi_j(x) \phi_j(y).$$
Its complexification gives the spectral projections kernel along the
anti-diagonal,
\begin{equation}\label{CXSP} \Pi_{I}(\zeta, \bar{\zeta}) =
 \sum_{j: \lambda_j \in I}
|\phi_j^{\C}(\zeta)|^2,  \end{equation}  and the integral of
(\ref{SPPROJDAMPED}) gives its temperate version
\begin{equation}\label{CXDSP}  P^{\tau}_{I}(\zeta, \bar{\zeta}) =
 \sum_{j: \lambda_j \in I}  e^{- 2 \tau \lambda_j}
|\phi_j^{\C}(\zeta)|^2,
\end{equation}
or in the crucial case of $\tau = \sqrt{\rho}(\zeta)$,
\begin{equation}\label{CXDSPa}  P_{I}(\zeta, \bar{\zeta}) =
 \sum_{j: \lambda_j \in I}  e^{- 2 \sqrt{\rho}(\zeta)\lambda_j}
|\phi_j^{\C}(\zeta)|^2,
\end{equation}

\subsection{Poisson operator as a complex Fourier integral
operator}

The damped spectral projection measure  $d_{\lambda} \; P_{[0,
\lambda]}^{\tau}(\zeta, \bar{\zeta})$ (\ref{SPPROJDAMPED}) is dual
under the real Fourier transform in the $t$ variable to the
restriction
\begin{equation} \label{CXWVGP} U (t + 2 i \tau, \zeta, \bar{\zeta}) = \sum_j
e^{(- 2 \tau + i t) \lambda_j} |\phi_j^{\C}(\zeta)|^2
\end{equation}  to the
anti-diagonal of the mixed Poisson-wave group. The
adjoint of the Poisson kernel $U(i \tau, x, y)$ also admits an
anti-holomorphic extension in the $y$ variable. The sum
(\ref{CXWVGP}) are the diagonal values of the complexified wave
kernel
\begin{equation} \label{EFORM}
\begin{array}{lll} U (t + 2 i \tau, \zeta, \bar{\zeta}') &  = &
\int_M  U (t + i \tau, \zeta, y)  E(i \tau, y, \bar{\zeta}'
) dV_g(x)\\
&& \\
&&  = \sum_j  e^{(- 2 \tau + i t) \lambda_j} \phi_j^{\C}(\zeta)
\overline{\phi_j^{\C}(\zeta')}.
\end{array}
\end{equation} We obtain
(\ref{EFORM}) by  orthogonality of the real eigenfunctions on $M$.

Since $U(t + 2 i \tau, \zeta, y)$ takes its values in the CR
holomorphic functions on $\partial M _{\tau}$, we consider the
Sobolev spaces $\ocal^{s + \frac{n-1}{4}}(\partial M _{\tau})$ of
CR holomorphic functions on the boundaries of the strictly
pseudo-convex domains $M_{\epsilon}$, i.e.
$${\mathcal O}^{s + \frac{m-1}{4}}(\partial M_{\tau}) =
W^{s + \frac{m-1}{4}}(\partial M_{\tau}) \cap \ocal (\partial
M_{\tau}), $$ where $W_s$  is the $s$th Sobolev space and where $
\ocal (\partial M_{\epsilon})$ is the space of boundary values of
holomorphic functions. The inner product on $\ocal^0 (\partial M
_{\tau} )$ is with respect to the Liouville measure
\begin{equation} \label{LIOUVILLEa} d\mu_{\tau} = (i \ddbar
\sqrt{\rho})^{m-1} \wedge d^c \sqrt{\rho}. \end{equation}

We then regard  $U(t + i \tau, \zeta, y)$ as the kernel of an
operator from $L^2(M) \to \ocal^0(\partial M_{\tau})$. It equals
its composition $ \Pi _{\tau} \circ U (t + i \tau)$ with the
\szego projector
$$ \Pi_{\tau} : L^2(\partial M_{\tau}) \to \ocal^0(
\partial M_{\tau})$$  for the tube $
M_{\tau}$, i.e.   the orthogonal projection onto boundary values
of holomorphic functions in the tube.

 This is a useful expression
for  the complexified wave kernel, because  $\tilde{\Pi}_{\tau}$
is a complex Fourier integral operator with a small wave front
relation. More precisely,  the real points of its  canonical
relation form  the graph $\Delta_{\Sigma}$ of the identity map on
the symplectic one
 $\Sigma_{\tau}
\subset T^*
\partial M_{\tau}$ spanned by the real one-form $d^c \rho$,
i.e. \begin{equation} \label{SIGMATAU} \Sigma_{\tau} = \{(\zeta; r
d^c \rho (\zeta)) ,\;\;\; \zeta \in \partial M_{\tau},\; r > 0 \}
 \subset T^* (\partial M_{\tau}).\;\;  \end{equation}
 We note that for each $\tau,$ there exists a symplectic equivalence $ \Sigma_{\tau} \simeq
 T^*M$ by the map $(\zeta, r d^c \rho(\zeta) )  \to
 (E_{\C}^{-1}(\zeta), r \alpha)$, where $\alpha = \xi \cdot dx$ is
 the action form (cf. \cite{GS2}).

The following result was first stated by  Boutet de Monvel\cite{Bou}  and has been proved in 
detail in \cite{Ze8,L,Ste}. 

\begin{theo}\label{BOUFIO}  $\Pi_{\epsilon} \circ U (i \epsilon): L^2(M)
\to \ocal(\partial M_{\epsilon})$ is a  complex Fourier integral
operator of order $- \frac{m-1}{4}$  associated to the canonical
relation
$$\Gamma = \{(y, \eta, \iota_{\epsilon} (y, \eta) \} \subset T^*M \times \Sigma_{\epsilon}.$$
Moreover, for any $s$,
$$\Pi_{\epsilon} \circ U (i \epsilon): W^s(M) \to {\mathcal O}^{s +
\frac{m-1}{4}}(\partial  M_{\epsilon})$$ is a continuous
isomorphism.
\end{theo}

In \cite{Ze8} we give the following sharpening of the sup norm estimates of \cite{Bou}:

\begin{prop} \label{PW} Suppose  $(M, g)$ is real analytic.  Then
$$ \sup_{\zeta \in M_{\tau}} |\phi^{\C}_{\lambda}(\zeta)| \leq C
  \lambda^{\frac{m+1}{2}} e^{\tau \lambda}, \;\;\;\; \sup_{\zeta \in M_{\tau}} |\frac{\partial \phi^{\C}_{\lambda}(\zeta)}{\partial \zeta_j}| \leq C
  \lambda^{\frac{m+3}{2}} e^{\tau \lambda}
$$
\end{prop}
The proof follows  easily from the fact that the complexified
Poisson kernel is a complex Fourier integral operator of finite
order. The estimates can be improved further. 

\subsection{\label{TOEP} Toeplitz dynamical construction of the wave group}

There exists an alternative  to the parametrix constructions of Hadamard-Riesz, Lax, H\"ormander and others 
which are reviewed in \S \ref{WAVEAPP}.  It is useful for  constructing the wave group $U(t)$ for large $t$,
when it is awkward to  use the group property $U(t/N)^N = U(t)$. As in Theorem \ref{BOUFIO} we denote
by $U(i \epsilon)$ the operator with kernel $U(i \epsilon, \zeta, y)$ with $\zeta \in \partial M_{\epsilon}, y \in M$. 
We also denote by $U^*(i \epsilon): \ocal(\partial M_{\epsilon}) \to L^2(M)$ the adjoint operator. Further, let
$$T_{g^t} : L^2(\partial M_{\epsilon}, d\mu_{\epsilon}) \to L^2(\partial M_{\epsilon}, d\mu_{\epsilon})$$
be the unitary translation operator
$$T_{g^t} f(\zeta) = f(g^t(\zeta)) $$
where $d\mu_{\epsilon}$ is the contact volume form on $\partial M_{\epsilon}$ and $g^t $ is the Hamiltonian
flow of $\sqrt{\rho}$ on $M_{\epsilon}$.  

\begin{prop} \label{TOEPROP}  There exists a symbol $\sigma_{\epsilon, t}$ such that
$$U(t) = U^*(i \epsilon)  \sigma_{\epsilon, t} T_{g^t} U(i \epsilon). $$
\end{prop}

The proof of this Proposition is to verify that the right side is a Fourier integral operator with canonical relation
the graph of the geodesic flow. One then constructs $\sigma_{\epsilon, t}$ so that the symbols match. The proof
is given in \cite{Ze6}. Related constructions are given in \cite{G1,BoGu}.

\section{\label{ERGODICNODAL} Equidistribution of complex nodal sets of real ergodic
eigenfunctions}

We now consider global results when hypotheses are made on the
dynamics of the geodesic flow. The main purpose of this section is to sketch the proof of Theorem \ref{ERGCXZ}.
 Use of the global wave operator brings into
play the relation between the geodesic flow and the complexified
eigenfunctions, and this allows one to prove gobal results on
nodal hypersurfaces that reflect the dynamics of the geodesic
flow. In some cases, one can determine not just the volume, but
the limit distribution of complex nodal hypersurfaces. Since we have discussed 
this result elsewhere \cite{Ze6} we only briefly review it here. 

The complex nodal hypersurface of an eigenfunction is   defined by
\begin{equation} Z_{\phi_{\lambda}^{\C}} = \{\zeta \in
M_{\epsilon_0} : \phi_{\lambda}^{\C}(\zeta) = 0 \}.
\end{equation}
As discussed in \S \ref{PSH}, there exists  a natural current of integration over the nodal
hypersurface in any Grauert tube $M_{\epsilon}$ with $\epsilon
< \epsilon_0$ , given by
\begin{equation}\label{ZDEF}  \langle [Z_{\phi_{\lambda}^{\C}}] , \phi \rangle =  \frac{i}{2 \pi} \int_{M_{\epsilon}} \ddbar \log
|\phi_{\lambda}^{\C}|^2 \wedge \phi =
\int_{Z_{\phi_{\lambda}^{\C}} } \phi,\;\;\; \phi \in \dcal^{ (m-1,
m-1)} (M_{\epsilon} ). \end{equation} In the second equality we
used the Poincar\'e-Lelong formula (\S \ref{PLLSEC}). We recall that $\dcal^{ (m-1,
m-1)} (M_{\epsilon} )$ stands for smooth test $(m-1,
m-1)$-forms with support in $M*_{\epsilon}.$

The nodal hypersurface $Z_{\phi_{\lambda}^{\C}}$ also carries a
natural volume form $|Z_{\phi_{\lambda}^{\C}}|$ as a complex
hypersurface in a K\"ahler manifold. By Wirtinger's formula, it
equals the restriction of $\frac{\omega_g^{m-1}}{(m - 1)!}$ to
$Z_{\phi_{\lambda}^{\C}}$. Hence, one can regard
$Z_{\phi_{\lambda}^{\C}}$ as defining  the  measure
\begin{equation} \label{MEAS} \langle |Z_{\phi_{\lambda}^{\C}}| , \phi \rangle
= \int_{Z_{\phi_{\lambda}^{\C}} } \phi \frac{\omega_g^{m-1}}{(m -
1)!},\;\;\; \phi \in C(B^*_{\epsilon} M).
\end{equation}
 We prefer to state results in terms of the
current $[Z_{\phi_{\lambda}^{\C}}]$ since it carries more
information.

We re-state Theorem \ref{ERGCXZ} as follows:

\begin{theo}\label{ZERO}  Let $(M, g)$ be  real analytic, and let $\{\phi_{j_k}\}$ denote a quantum ergodic sequence
of eigenfunctions of its Laplacian $\Delta$.  Let
$M_{\epsilon_0}$ be the maximal Grauert tube around $M$.  Let $\epsilon <
\epsilon_0$. Then:
$$\frac{1}{\lambda_{j_k}} [Z_{\phi_{j_k}^{\C}}] \to  \frac{i}{ \pi} \ddbar \sqrt{\rho}\;\;
\mbox{weakly in }\;\;  \dcal^{' (1,1)} (M_{\epsilon}), $$ in
the sense that,   for any continuous test form $\psi \in \dcal^{
(m-1, m-1)}(M_{\epsilon})$, we have
$$\frac{1}{\lambda_{j_k}} \int_{Z_{\phi_{j_k}^{\C}}} \psi \to
 \frac{i}{ \pi} \int_{M_{\epsilon} } \psi \wedge \ddbar
\sqrt{\rho}. $$ Equivalently, for any  $\phi \in C(M_{\epsilon})$,
$$\frac{1}{\lambda_{j_k}} \int_{Z_{\phi_{j_k}^{\C}}} \phi \frac{\omega_g^{m-1}}{(m -
1)!}  \to
 \frac{i}{ \pi} \int_{M_{\epsilon} } \phi  \ddbar
\sqrt{\rho}  \wedge \frac{\omega_g^{m-1}}{(m - 1)!} . $$
\end{theo}


\subsection{Sketch of the proof}

The first step is to find a nice way to express $\phi_j^{\C}$ on $M_{\C}$. Very often, when we analytically
continue a function, we lose control over its behavior. 
The trick is to observe that the complexified wave group analytically continues the eigenfunctions. 
Recall that $U(t) = \exp i t \sqrt{\Delta}$. Define the Poisson operator as
$$U(i \tau) = e^{- \tau \sqrt{\Delta}}. $$
Note that
$$U(i \tau) \phi_j = e^{- \tau \lambda_j} \phi_j. $$
The next step is to analytically continue $U(i \tau, x, y)$ in $x \to z \in M_{\tau}$.

The complexified  Poisson
kernel is define by 
$$U(i \tau, \zeta, y) = \sum_{j = 0}^{\infty} e^{- \tau \lambda_j} \phi^{\C}_j(\zeta)
\phi_j(y). $$

It is holomorphic in  $\zeta \in M_{\tau}$, i.e. 
when $\sqrt{\rho}(\zeta) < \tau$. But the main point is that it remains a Fourier integral
operator after analytic continuation:

\begin{theo} (Hadamard, Boutet de Monvel, Z, M. Stenzel, G. Lebeau)  $U(i \epsilon, z, y): L^2(M)
\to H^2(\partial M_{\epsilon})$ is a  complex Fourier integral
operator of order $- \frac{m-1}{4}$  quantizing the complexified
exponential map $\exp_y i \epsilon \frac{\eta}{|\eta|}: S^*_{\epsilon} \to \partial M_{\epsilon} .$

\end{theo}

We first observe that 
$$U(i \tau) \phi_{\lambda_j} = e^{- \tau \lambda_j}
\phi_{\lambda_j}^{\C}. $$
This follows immediately by integrating
$$U(i \tau, \zeta, y) = \sum_{k = 0}^{\infty} e^{- \tau \lambda_k} \phi^{\C}_k(\zeta)
\phi_k(y)$$
against $\phi_j$ and using orthogonality.

 But we know that  $U(i \tau) \phi_{\lambda_j}$ is a Fourier integral operator. It is a fact that
such an operator can only
change $L^2$ norms by powers of $\lambda_j$. So 
$$|| U(i \tau) \phi_{\lambda_j} ||^2_{L^2(\partial M_{\epsilon})} $$
has polynomial growth in $\lambda_j$ and therefore we have,
 $$||\phi_{\lambda_j} ||^2_{L^2(\partial M_{\epsilon})} =  \lambda^{some power} e^{ \tau \lambda_j}.$$
The power is relevant because we are taking the normalized logarithm.

The first step is to prove quantum ergodicity of the complexified
eigenfunctions:

\begin{theo} \label{W*HUSIMIERG} Assume the geodesic flow of $(M, g)$ is ergodic. Then
$$ \frac{|\phi_{j_k}^{\epsilon}(z)|^2}{||\phi_{j_k}^{\epsilon} ||_{L^2(\partial
M_{\epsilon})}^2} \to 1,\;\; \mbox{weak  * on }\;\; C
(\partial M_{\epsilon}),
$$ along a density one subsquence of $\lambda_j$. I.e. for any continuous $V$,
$$\int_{\partial M_{\epsilon}} V \frac{|\phi_{j_k}^{\epsilon}(z)|^2}{||\phi_{j_k}^{\epsilon} ||_{L^2(\partial
M_{\epsilon})}^2}  d vol \to \int_{\partial M_{\epsilon}}  V dvol. $$
\end{theo}
Thus, Husimi measures tend to 1 weakly as measures. 
We then apply Hartogs' Lemma \ref{HARTOGS} to obtain, 

\begin{lem} \label{ALSOLOGLIM} We have: For all but a sparse subsequence of eigenvalues, $$\frac{1}{\lambda_{j_k}} \log  \frac{|\phi_{j_k}^{\epsilon}(z)|^2}{||\phi_{j_k}^{\epsilon} ||_{L^2(\partial
M_{\epsilon})}^2} \to 0, \;\; \mbox{in} \;L^1( M_{\epsilon}). $$
\end{lem}

This is almost obvious from the QE theorem.  The limit is $\leq 0$ and it were $< 0$ on a set of positive measure
it would contradict 
$$ \frac{|\phi_{j_k}^{\epsilon}(z)|^2}{||\phi_{j_k}^{\epsilon} ||_{L^2(\partial
M_{\epsilon})}^2} \to 1. $$

 Combine Lemma \ref{ALSOLOGLIM} with
Poincare- Lelong:
$$  \frac{1}{\lambda_{j_k}} [Z_{j_k} ] = i\ddbar \log |\phi_{j_k}^{\C}|^2. $$
We get
$$ \frac{1}{\lambda_{j_k}} \ddbar \log |\phi_{j_k}^{\C}|^2 \sim
\frac{1}{\lambda_{j_k}}  \ddbar \log ||\phi_{j_k}^{\C}||^2_{L^2(\partial M_{\epsilon})} \;\; \mbox{weak * on }\;\;  M_{\epsilon}.$$

To complete proof we need to prove:
\begin{equation} \label{NORMCONV}  \frac{1}{\lambda_j} \log ||\phi_j^{\C}||_{\partial
M_{\epsilon}}^2 \to 2 \epsilon . \end{equation}
But $U(i \epsilon) = e^{- \epsilon \lambda_j} \phi_j^{\C}$, hence
$||\phi_{\lambda}^{\C}||_{L^2(\partial M_{\epsilon} )}^2$
equals $e^{2 \epsilon \lambda_j}$ times
$$\langle U(i \epsilon) \phi_{\lambda}, U(i \epsilon)
\phi_{\lambda} \rangle = \langle U(i \epsilon)^*
U(i \epsilon) \phi_{\lambda}, \phi_{\lambda} \rangle. $$ 

 But  $U(i \epsilon)^* U(i \epsilon)$ is a
pseudodifferential operator of order $\frac{n-1}{2}$. Its symbol
$|\xi|^{- \frac{n-1}{2}}$ doesn't contribute to the logarithm.

We now provide more details on each step. 

\subsection{Growth properties of complexified eigenfunctions}

In this section we prove Lemma \ref{ALSOLOGLIM}  in more detail. 
We state it in combination with \eqref{NORMCONV}.

\begin{theo} \label{LOGLIMERG} If the geodesic flow is ergodic, then for all but a sparse subsequence of
$\lambda_j$,\;
$$\frac{1}{\lambda_{j_k} } \log |\phi_{j_k}^{\C}(z)|^2  \to \sqrt{\rho} \; \mbox{in}\; L^1(M_{\epsilon}).$$
\bigskip

\end{theo}

The Grauert tube function is a maximal PSH function with bound $\leq \epsilon$ on $M_{\epsilon}$. Hence
Theorem \ref{LOGLIMERG}   says that ergodic eigenfunctions have the maximum
exponential growth rate possible for any  eigenfunctions.

A key object in the proof  is the sequence of functions
$U_{\lambda}(x, \xi) \in C^{\infty}(M_{\epsilon} )$ defined by
\begin{equation} \label{DEFS} \left\{ \begin{array}{l}  U_{\lambda}(x, \xi) : =
\frac{\phi_{\lambda}^{\C}(x, \xi)}{\rho_{\lambda}(x, \xi)},\;\;\;
(x, \xi) \in
 M_{\epsilon} , \;\;\; \mbox{where}\\ \\
\rho_{\lambda} (x, \xi) :=  ||\phi_{\lambda}^{\C} |_{\partial
M_{|\xi|_g} } ||_{L^2(\partial M_{|\xi|_g})}
\end{array} \right.
\end{equation}
Thus,  $\rho_{\lambda}(x, \xi)$ is the   $L^2$-norm of the
restriction of  $\phi_{\lambda}^{\C}$  to  the sphere bundle
$\{\partial M_{\epsilon } \}$ where $\epsilon = |\xi|_g$.
$U_{\lambda}$ is of course not holomorphic, but its restriction to
each sphere bundle is CR holomorphic there, i.e.
\begin{equation} \label{LITTLEU} u_{\lambda}^{\epsilon}  =
U_{\lambda} |_{\partial M_{\epsilon} }  \in \ocal^0 (\partial
M_{\epsilon}).
\end{equation}

Our first result gives an ergodicity property of holomorphic
continuations of ergodic eigenfunctions.

\begin{lem} \label{ERGOCOR} Assume that $\{\phi_{j_k}\}$ is a quantum ergodic sequence of $\Delta$-eigenfunctions
on $M$ in the sense of (\ref{QEDEF}).  Then for each $0 < \epsilon
<\epsilon_0$,
$$|U_{j_k}|^2  \to  \frac{1}{\mu_1(S^* M)} \sqrt{\rho}^{-m + 1},\;\; \mbox{weakly in}\;\; L^1(
M_{\epsilon}, \omega^m) . $$
\end{lem}

We note that $\omega^m = r^{m-1} dr d\omega dvol(x)$ in polar
coordinates, so the right side indeed lies in $L^1$. The actual
limit function is otherwise irrelevant. The next step is to use a
compactness argument  to obtain
strong convergence of the normalized logarithms of the sequence
$\{|U_{\lambda}|^2\}$. The first statement of the following lemma
immediately implies the second.

\begin{lem} \label{ZEROWEAK} Assume that $|U_{j_k}|^2  \to  \frac{1}{\mu_{\epsilon}(\partial M_{\epsilon})}  \sqrt{\rho}^{-m + 1
},\;\; \mbox{weakly in}\;\; L^1(M_{\epsilon}, \omega^m). $
Then:

\begin{enumerate}

\item  $\frac{1}{\lambda_{j_k}}  \log |U_{j_k}|^2 \to 0$ strongly
in $L^1(M_{\epsilon}). $

\item $\frac{1}{\lambda_j} \ddbar \log |U_{j_k}|^2 \to 0,\;\;
\mbox{weakly in} \;\; \dcal^{'(1,1)}(M_{\epsilon}). $

\end{enumerate} \end{lem}

Separating out the numerator and denominator of $|U_j|^2$, we
obtain that
 \begin{equation} \label{SEPARATE} \frac{1}{\lambda_{j_k}} \ddbar  \log |\phi_{j_k}^{\C}|^2 - \frac{2}{\lambda_{j_k}} \ddbar
 \log \rho_{\lambda_{j_k}} \to 0,\;\;\; (\lambda_{j_k} \to \infty). \end{equation}  The next lemma shows that the second term
has a weak limit:

\begin{lem} \label{NORM}  For $0 < \epsilon < \epsilon_0$,
$$\frac{1}{\lambda{j_k}} \log \rho_{\lambda{j_k}}(x, \xi)  \to \; \sqrt{\rho},\;\;\;\mbox{ in }\;\;L^1(M_{\epsilon})\;\; \mbox{as}\;\; \lambda{j_k} \to
\infty.
$$
Hence, $$\frac{1}{\lambda_{j_k}} \ddbar
 \log \rho_{\lambda_{j_k}} \to \ddbar \sqrt{\rho}\;\;\; (\lambda_j \to \infty) \;\; \mbox{weakly in }\;\;  \dcal'(M_{\epsilon}).$$

\end{lem}

It follows that the left side of (\ref{SEPARATE}) has the same
limit, and that will complete the proof of
 Theorem \ref{ZERO}.

\subsection{Proof of Lemma \ref{ERGOCOR} and Theorem \ref{W*HUSIMIERG}}

  We begin by proving a
 weak limit formula for the CR holomorphic functions $u_{\lambda}^{\epsilon}$  defined in
 (\ref{LITTLEU}) for fixed $\epsilon$.

\begin{lem} \label{ERGO} Assume that $\{\phi_{j_k}\}$ is a quantum ergodic sequence.
Then for each $0 < \epsilon < \epsilon_0$,
$$|u_{j_k}^{\epsilon}|^2 \to
\frac{1}{\mu_{\epsilon}(\partial M_{\epsilon})},\;\;
\mbox{weakly in}\;\; L^1(\partial M_{\epsilon},
d\mu_{\epsilon}).
$$
That is, for any   $a \in C(\partial M_{\epsilon})$,
$$\int_{\partial M_{\epsilon} } a(x, \xi) |u_{j_k}^{\epsilon}((x, \xi)|^2
d\mu_{\epsilon} \to \frac{1}{\mu_{\epsilon}(\partial
M_{\epsilon} )} \int_{\partial M_{\epsilon}} a(x, \xi)
d\mu_{\epsilon}.$$

\end{lem}

\begin{proof}

It suffices to consider $a \in C^{\infty} (\partial
M_{\epsilon})$.  We then consider the Toeplitz operator
$\Pi_{\epsilon} a \Pi_{\epsilon}$ on $\ocal^0(\partial
M_{\epsilon})$. We have,
\begin{equation}\label{EQUIV} \begin{array}{lll}  \langle \Pi_{\epsilon} a \Pi_{\epsilon}
u_j^{\epsilon}, u_j^{\epsilon} \rangle & =  & e^{2 \epsilon
\lambda_j}\;  ||\phi_{\lambda}^{\C}||_{L^2(\partial M_{\epsilon})}^{-2} \langle \Pi_{\epsilon} a \Pi_{\epsilon} U(i \epsilon)
 \phi_j, U(i \epsilon) \phi_j
\rangle_{L^2(\partial M_{\epsilon})}\\  & & \\
& = &  e^{2 \epsilon \lambda_j}\;
||\phi_{\lambda}^{\C}||_{L^2(\partial M_{\epsilon})}^{-2}
\langle U(i \epsilon)^* \Pi_{\epsilon} a \Pi_{\epsilon} U(i
\epsilon) \phi_j, \phi_j \rangle_{L^2(M)}.
\end{array}
\end{equation}

It is not hard to see that $ U(i \epsilon)^* \Pi_{\epsilon} a
\Pi_{\epsilon} U(i \epsilon) $ is a pseudodifferential operator on
$M$ of order $- \frac{m-1}{2}$ with principal symbol $\tilde{a}
|\xi|_g^{ - \frac{m-1}{2}}$, where $\tilde{a}$ is the (degree $0$)
homogeneous extension of $a$ to $T^*M - 0$.  The normalizing
factor $ e^{2 \epsilon \lambda_j}\;
||\phi_{\lambda}^{\C}||_{L^2(\partial B^*_{\epsilon} M)}^{-2}$ has
the same form with $a = 1$. Hence, the expression on the right
side of  (\ref{EQUIV}) may be written as
\begin{equation} \frac{\langle U(i \epsilon)^* \Pi_{\epsilon} a \Pi_{\epsilon} U(i
\epsilon) \phi_j, \phi_j \rangle_{L^2(M)}}{\langle U(i \epsilon)^*
\Pi_{\epsilon} U(i \epsilon) \phi_j, \phi_j \rangle_{L^2(M)}}.
\end{equation}

  By the standard quantum ergodicity result on compact
Riemannian manifolds with ergodic geodesic flow (see \cite{Shn,
Ze4, CV} for proofs and references)  we have
\begin{equation} \label{QE} \frac{\langle U(i \epsilon)^* \Pi_{\epsilon} a \Pi_{\epsilon} U(i
\epsilon) \phi_{j_k}, \phi_{j_k}\rangle_{L^2(M)}}{\langle U(i
\epsilon)^* \Pi_{\epsilon} U(i \epsilon) \phi_{j_k}, \phi_{j_k}
\rangle_{L^2(M)}} \to \frac{1}{\mu_{\epsilon}(\partial
M_{\epsilon})} \int_{\partial M_{\epsilon}} a d
\mu_{\epsilon}.
\end{equation}
More precisely, the numerator is asymptotic to the right side
times $\lambda^{- \frac{m-1}{2}}$, while the denominator has the
same asymptotics when $a$ is replaced by $1$. We also use that
$\frac{1}{\mu_{\epsilon}(\partial M_{\epsilon})}
\int_{\partial M_{\epsilon}} a d \mu_{\epsilon}$ equals the
analogous average of $\tilde{a}$ over $\partial M_{\epsilon}$. Taking the ratio produces
(\ref{QE}).

Combining (\ref{EQUIV}), (\ref{QE}) and the fact that
$$ \langle \Pi_{\epsilon} a \Pi_{\epsilon}
u_j^{\epsilon}, u_j^{\epsilon} \rangle = \int_{\partial
B^*_{\epsilon} M} a |u_j^{\epsilon}|^2 d \mu_{\epsilon}$$
completes the proof of the lemma.

\end{proof}

We now complete the proof of Lemma \ref{ERGOCOR}, i.e. we prove
that
\begin{equation} \label{WEAKLIMA} \int_{M_{\epsilon}} a
|U_{j_k}|^2 \omega^m \to \frac{1}{\mu_{\epsilon}(\partial M_{\epsilon})}
\int_{M_{\epsilon}} a \sqrt{\rho}^{-m + 1} \omega^m
\end{equation} for any $a \in C(M_{\epsilon}).$
It is only necessary to relate the surface  Liouville measures $d\mu_r$
(\ref{LIOUVILLE}) to the K\"ahler volume measure. One may write
$d\mu_r = \frac{d}{dt} |_{t = r} \chi_t \omega^m$, where $\chi_t$
is the characteristic function of $M_t= \{ \sqrt{\rho} \leq t\}$.
By homogeneity of $|\xi|_g$, $\mu_{r}(\partial M_{r}) =
r^{m-1} \mu_{\epsilon}(\partial M_{\epsilon}) $. If  $a \in C(M_{\epsilon})$,
then $\int_{M_{\epsilon}} a \omega^m = \int_0^{\epsilon}
\{\int_{\partial M_{r}} a
 d \mu_{r}\} dr.  $ By Lemma \ref{ERGO}, we  have

\begin{equation} \label{LIOUVSYM} \begin{array}{lll} \int_{M_{\epsilon}} a |U_{j_k}|^2
\omega^m = \int_0^{\epsilon}\{ \int_{\partial M_{r}} a
|u_{j_k}^r|^2  d \mu_{r} \}  d r &\to & \int_0^{\epsilon}
\{\frac{1}{\mu_{r} (\partial B^*_{r})} \int_{\partial M_{r}} a
d \mu_r \}
dr \\ & & \\
& = &  \frac{1}{\mu_{\epsilon}(\partial
 M_{\epsilon})}\int_{M_{\epsilon}} a r^{-m + 1} \omega^m, \\ & & \\
& \implies & w^*-\lim_{\lambda \to \infty} |U_{j_k}|^2 =
 \frac{1}{\mu_1(\partial M_{\epsilon})} \sqrt{\rho}^{-m + 1}.
\end{array}
\end{equation}

\subsection{Proof of Lemma \ref{NORM}}

In fact, one has
$$\frac{1}{\lambda} \log \rho_{\lambda}(x, \xi)  \to \sqrt{\rho},\;\;\;\mbox{uniformly in }\;\;M_{\epsilon}\;\; \mbox{as}\;\; \lambda \to
\infty.
$$

\begin{proof}

Again using  $U(i \epsilon) \phi_{\lambda} = e^{-\lambda \epsilon}
\phi_{\lambda}^{\C}$, we have:
 \begin{equation} \label{ME} \begin{array}{lll} \rho^2_{\lambda}(x, \xi) &
 =&
 \langle \Pi_{\epsilon} \phi_{\lambda}^{\C}, \Pi_{\epsilon} \phi_{\lambda}^{\C} \rangle_{L^2(\partial B^*_{\epsilon} M)} \;\;(\epsilon = |\xi|_{g_x})\\ & & \\
 & = &
 e^{  2 \lambda
\epsilon}
 \langle  \Pi_{\epsilon} U(i \epsilon) \phi_{\lambda}, \Pi_{\epsilon} U(i \epsilon)
 \phi_{\lambda}
\rangle_{L^2(\partial B^*_{\epsilon} M)} \\ & & \\ & = &   e^{ 2
\lambda \epsilon} \langle U(i \epsilon)^* \Pi_{\epsilon} U(i
\epsilon) \phi_{\lambda},  \phi_{\lambda} \rangle_{L^2(M)}.
\end{array}
\end{equation}
Hence,
\begin{equation} \label{LOGAR} \frac{2}{\lambda} \log
\rho_{\lambda} (x, \xi)= 2 |\xi|_{g_x} + \frac{1}{\lambda} \log
\langle U(i \epsilon)^* \Pi_{\epsilon} U(i \epsilon)
\phi_{\lambda}, \phi_{\lambda} \rangle. \end{equation}
The second term on the right side is the matrix element of a pseudo-differential operator,
hence is bounded by some power of $\lambda$. Taking the logarithm gives a remainder of
order $\frac{\log \lambda}{\lambda}$.

\end{proof}

\subsection{Proof of Lemma \ref{ZEROWEAK}}

\begin{proof} We wish to prove that  $$\psi_j := \frac{1}{\lambda_j}
\log |U_j|^2 \to 0 \;\; \mbox{in}\; L^1(M_{\epsilon}).
$$
As we have said, this is almost obvious from Lemmas \ref{ERGOCOR} and \ref{ERGO}. 
 If the conclusion is not true, then
there exists a subsquence $\psi_{j_k}$ satisfying
$||\psi_{j_k}||_{L^1(B^*_{\epsilon} M)} \geq \delta > 0.$
To obtain a contradiction, we use Lemma \ref{HARTOGS}. 
 

To see that the hypotheses are satisfied in our example,
it suffices to prove these statements on each surface $\partial
M_{\epsilon}$ with uniform constants independent of
$\epsilon$. On the surface $\partial M_{\epsilon}$, $U_j =
u^{\epsilon}_j$. By the Sobolev inequality in
$\ocal^{\frac{m-1}{4}}(\partial M_{\epsilon})$, we have
$$\begin{array}{lll} \sup_{(x, \xi) \in \partial M_{\epsilon})} |u_j^{\epsilon} (x, \xi)| & \leq &
\lambda_j^m ||u_j^{\epsilon} (x, \xi)||_{L^2(\partial
M_{\epsilon})} \\ & & \\
& \leq & \lambda_j^m.
\end{array}$$
 Taking
the logarithm, dividing by $\lambda_j$, and combining with the
limit formula of Lemma \ref{NORM} proves (i) - (ii).

We now settle the dichotomy above by proving that the sequence
$\{\psi_j\}$ does not tend uniformly to $-\infty$ on compact sets.
That would imply that $\psi_j \to - \infty$ uniformly on the
spheres $\partial M_{\epsilon}$ for each $\epsilon <
\epsilon_0$. Hence, for each $\epsilon$, there  would exist $K
> 0$ such that for $k \geq K$,
\begin{equation}\frac{1}{\lambda_{j_k}} \log | u_{j_k}^{\epsilon} (z)| \leq -
1.\label{firstposs}\end{equation} However, (\ref{firstposs})
implies that
$$ | u_{j_k}(z)| \leq e^{- 2 \lambda_{j_k}}\;\;\;\;\forall z \in
\partial M_{\epsilon}\;, $$ which is inconsistent with the hypothesis that $|
u_{j_k}^{\epsilon} (z)| \rightarrow 1$ in $\dcal'(\partial
M_{\epsilon})$.

Therefore,
 there must exist a subsequence, which we continue to denote by
$\{\psi_{j_k}\}$, which converges in $L^1(M_{\epsilon_0})$ to some
$\psi \in L^1(M_{\epsilon_0}).$ Then,
$$\psi (z) = \limsup_{k \rightarrow \infty} \psi_{j_k}\leq 2 |\xi|_g\;\;\;\;\;\;{\rm (a.e)}\;.$$ Now let
$$\psi^*(z):= \limsup_{w \rightarrow z}  \psi (w) \leq 0 $$ be the
upper-semicontinuous regularization of $\psi$. Then $\psi^*$ is
plurisubharmonic on $M_{\epsilon}$ and $\psi^* = \psi$ almost
everywhere.

If  $\psi^* \leq 2 |\xi|_g - \delta$ on a set $U_{\delta}$  of positive measure, then $\psi_{j_k}(\zeta) \leq
-\delta/2$ for $\zeta \in U_{\delta},\ k\geq K$; i.e.,
\begin{equation} |\psi_{j_k}(\zeta)|\leq e^{- \delta
\lambda_{j_k}},\;\;\;\;\; \zeta\in U_{\delta}, \;\;k \geq K.
\end{equation} This  contradicts the weak convergence to $1$ and concludes the proof.

\end{proof}




\section{\label{NODALGEOS} Intersections of nodal sets and analytic curves  on real
analytic surfaces}

It is often possible to obtain more refined results on nodal sets
by studying their intersections with some fixed (and often
special) hypersurface. This has been most successful in dimension
two.  
In \S \ref{PLANEDOMAIN}  we discuss  upper bounds on the number of  intersection points of the nodal set with the bounary
of a real analytic plane domain and more general `good' analytic curves.  
To obtain lower bounds or asymptotics, we need to add some dynamical hypotheses. In case of 
ergodic geodesic flow, we can obtain equidistribution theorems for intersections of nodal
sets and geodesics on surfaces. The dimensional restriction
is due to the fact that the results are partly based on the quantum ergodic restriction theorems of \cite{TZ,TZ2},
which concern restrictions of eigenfunctions to hypersurfaces. Nodal sets and geodesics have complementary
dimensions and intersect in points, and therefore it makes sense to count the number of intersections. But
we do not yet have a mechanism for studying restrictions to geodesics when $\dim M \geq 3$.

\subsection{\label{PLANEDOMAIN} Counting nodal lines which touch the boundary in analytic
plane domains}
In this section, we review the results of \cite{TZ} giving upper bounds
on the number of intersections of the nodal set with the boundary of an analytic (or more
generally piecewise analytic) 
plane domain.  One may expect that the results of
this section can also be generalized to higher dimensions by measuring codimension two nodal hypersurface volumes within
the boundary. 

Thus we would like to  count the number of 
nodal lines (i.e. components of the nodal set)  which touch the boundary. Here we assume that $0$
is a regular value so that components of the nodal set are either loops in the interior (closed nodal
loops)  or curves
which touch the boundary in two points (open nodal lines).  It is known that  for generic piecewise analytic
plane domains, zero is a regular value of all the eigenfunctions
$\phi_{\lambda_j}$, i.e. $\nabla \phi_{\lambda_j} \not= 0$ on
$\ncal_{\phi_{\lambda_j}}$ \cite{U}; we then call the nodal set
regular.
Since the boundary lies in the nodal set for Dirichlet boundary
conditions, we remove it from the nodal set before counting
components. Henceforth, the number of components of the nodal set
in the Dirichlet case means the number of components of
$\ncal_{\phi_{\lambda_j}} \backslash \partial \Omega.$

We now sketch the proof of
Theorems \ref{INTREALBDYint}  in the case of Neumann boundary conditions.
By a piecewise
analytic domain $\Omega^2 \subset \R^2$,  we mean a
compact domain with piecewise analytic boundary, i.e. $\partial
\Omega$ is a union of a finite number of piecewise analytic curves
which intersect only at their common endpoints.
Such domains are often studied as archtypes of domains with ergodic billiards
and quantum chaotic eigenfunctions, in particular the 
 Bunimovich stadium or Sinai billiard. 

For the Neumann problem, the boundary nodal points are the same as
the zeros of the boundary values $\phi_{\lambda_j} |_{\partial
\Omega}$ of the eigenfunctions. The number of boundary nodal
points  is thus twice the number of open nodal lines. Hence in the
Neumann case, the  Theorem follows from:

\begin{theo}\label{BNP}  Suppose that $\Omega \subset \R^2$ is a piecewise real analytic  plane domain.
 Then the number $n(\lambda_j) = \# \ncal_{\phi_{\lambda_j}} \cap \partial
 \Omega$ of zeros of the boundary values  $\phi_{\lambda_j} |_{\partial
\Omega}$ of the $j$th Neumann  eigenfunction satisfies
$n(\lambda_j) \leq C_{\Omega}
 \lambda_j$, for some $C_{\Omega} > 0$.
\end{theo}
This is a more precise version of   Theorem \ref{INTREALBDYint}  since it does not assume that $0$ is
a regular value. 
We prove Theorem \ref{BNP}  by analytically
continuing the boundary values of the eigenfunctions and counting  {\it complex zeros and critical points} of analytic
continuations of Cauchy data of eigenfunctions. When $\partial
\Omega \in C^{\omega}$,  the eigenfunctions can be holomorphically
continued to an open tube domain  in $\C^2$ projecting over an
open neighborhood $W$ in $\R^2$ of $\Omega$ which is independent
of the eigenvalue.  We denote by $\Omega_{\C} \subset \C^2$ the
points $\zeta = x + i \xi \in \C^2 $ with $x \in \Omega$.  Then
$\phi_{\lambda_j}(x)$ extends to a holomorphic function
$\phi_{\lambda_j}^{\C}(\zeta)$ where $x \in W$ and where $ |\xi|
\leq \epsilon_0$ for some $\epsilon_0 > 0$.  

Assuming $\partial \Omega$ real analytic, we   define the
(interior) complex nodal set by
$$\ncal_{\phi_{\lambda_j}}^{\C} = \{\zeta \in  \Omega_{\C}:
\phi_{\lambda_j}^{\C}(\zeta) = 0 \}.  $$ 

\begin{theo} \label{mainthm} Suppose that $\Omega \subset \R^2$ is a piecewise real analytic  plane
domain, and denote by $(\partial \Omega)_{\C}$ the union of the
complexifications of its real analytic boundary components.

\begin{enumerate}

\item  Let
  $n(\lambda_j, \partial \Omega_{\C} ) = \# Z_{\phi_{\lambda_j}}^{\partial \Omega_{\C}}$ be the number
of complex zeros on the complex boundary. 
 Then there exists a constant $C_{\Omega} > 0$ independent of the radius of
  $(\partial \Omega)_{\C}$ such that
$n(\lambda_j, \partial \Omega_{\C})
 \leq C_{\Omega} \lambda_j. $


\end{enumerate}

\end{theo}

The  theorems on real nodal lines and critical points  follow from
the fact that real zeros and critical points are also complex
zeros and critical points, hence
\begin{equation} n(\lambda_j) \leq n(\lambda_j, \partial \Omega_{\C} )
. \end{equation}
All of the results are sharp, and are already obtained for certain
sequences of eigenfunctions  on a disc (see \S \ref{EXAMPLES}). 

 To prove \ref{mainthm}, we   represent the analytic continuations of the boundary values of the
eigenfunctions  in terms of  layer potentials.
  Let $G(\lambda_j, x_1, x_2)$ be any `Green's function' for the Helmholtz equation
 on $\Omega$, i.e. a solution
 of $(- \Delta - \lambda_j^2) G(\lambda_j, x_1, x_2) = \delta_{x_1}(x_2)$ with $x_1, x_2  \in \bar{\Omega}$.
By Green's formula,
\begin{equation}\label{GREENSFORMULA}  \phi_{\lambda_j}(x, y) = \int_{\partial \Omega} \left(\partial_{\nu}
G(\lambda_j,   q, (x, y)) \phi_{\lambda_j}(q) - G(\lambda_j,   q,
(x, y))
\partial_{\nu} \phi_{\lambda_j}(q) \right) d\sigma(q), \end{equation}
 where $(x, y) \in \R^2$,   where  $d\sigma$ is
arc-length measure on $\partial \Omega$  and where
$\partial_{\nu}$ is the normal derivative by the interior unit
normal. Our aim is to analytically continue this formula.

  In the case of
   Neumann eigenfunctions $\phi_\lambda$ in $\Omega,$
\begin{equation} \label{green1}
\phi_{\lambda_j}(x, y)= \int_{\partial \Omega}
\frac{\partial}{\partial \nu_{q } }G(\lambda_j, q,
(x, y)) u_{\lambda_j}(q) d\sigma (q),\;\; (x, y)
\in \Omega^o\;\; (\mbox{Neumann}).
\end{equation}

 To obtain
concrete representations we need to choose $G$. We choose the
real ambient Euclidean Green's function $S$ 
\begin{equation} \label{potential1}
S(\lambda_j, \xi, \eta;  x, y)  = - Y_0(\lambda_jr((x, y); (\xi,
\eta))),
\end{equation}
where $r = \sqrt{z z^*}$ is  the distance function (the square
root of $r^2$ above)  and where $Y_{0}$ is the Bessel function of
order zero of the second kind.
The Euclidean Green's function has
the form
\begin{equation} \label{GAB} S(\lambda_j,
\xi, \eta; x,y) = \linebreak A(\lambda_j,   \xi, \eta; x, y) \,
\log \frac{1}{r} + B(\lambda_j, \xi, \eta; x, y),\end{equation}
where $ A$ and $B$ are entire functions  of $r^2$.
The coefficient $ A = J_0(\lambda_jr)$ is known
as the Riemann function.

 By  the `jumps' formulae, the double layer
potential  $\frac{\partial}{\partial \nu_{\tilde{q}}} S(\lambda_j,
 \tilde{q}, (x, y))$ on $\partial \Omega \times \bar{\Omega}$  restricts to
 $\partial \Omega \times  \partial \Omega$ as $\frac{1}{2}
\delta_q(\tilde{q})  + \frac{\partial}{\partial \nu_{\tilde{q}}}
S(\lambda_j,  \tilde{q}, q)$ (see e.g. \cite{TI,TII}). Hence in the
Neumann case the boundary values $u_{\lambda_j}$ satisfy,
\begin{equation} \label{green1c} u_{\lambda_j}(q)=  2 \int_{\partial \Omega}
\frac{\partial}{\partial \nu_{\tilde{q}}} S(\lambda_j,\tilde{q},
q) u_{\lambda_j}(\tilde{q}) d\sigma(\tilde{q})\;\;
(\mbox{Neumann}).
\end{equation}
 We have,
\begin{equation}  \begin{gathered}
\frac{\partial}{\partial \nu_{\tilde{q}}} S(\lambda_j, \tilde{q},
q) =  - \lambda_j Y_1 (\lambda_j r) \cos \angle(q -\tilde{q},
\nu_{\tilde{q}}).
\end{gathered}\label{Neumann-F}\end{equation}

It is equivalent, and sometimes  more convenient, to use the
(complex valued) Euclidean outgoing Green's function  $\Ha_0(k
z)$, where $\Ha_0 = J_0 + i Y_0$ is the Hankel function of order
zero. It has the same form as (\ref{GAB}) and only differs by the
addition of the even entire function $J_0$ to the $B$ term. If  we
use the Hankel free outgoing Green's function, then in place of
(\ref{Neumann-F}) we have the kernel
\begin{equation} \label{HANKELINT} \begin{array}{lll} N(\lambda_j, q(s), q(s'))
&=& \frac{i}{2} \partial_{\nu_y} \Ha_{0}(\lambda_j|q(s) - y|)|_{y = q(s')} \\ &&\\
&=& -\frac{i}{2} \lambda_j\Ha_{1} (\lambda_j|q(s) - q(s')|) \cos \angle(q(s')
-q (s), \nu_{q(s')}),
\end{array} \end{equation}
and in place of (\ref{green1c}) we have the
formula
\begin{equation} \label{int1a}
u_{\lambda_j}(q(t))= \int_{0}^{2\pi}
N(\lambda_j, q(s), q(t)) \, u_{\lambda_j}(q(s))
ds.
\end{equation}

The next step is to analytically  continue the  layer potential
representations (\ref{green1c}) and (\ref{int1a}). The main point
is to express the analytic continuations of Cauchy data of Neumann
and Dirichlet eigenfunctions in terms of the real Cauchy data.
For brevity, we only consider (\ref{green1c}) but essentially the
same arguments apply to the free outgoing representation
(\ref{int1a}).

 As mentioned above,
both $ A(\lambda_j,\xi, \eta, x, y)$ and $B(\lambda_j,\xi, \eta,
x, y)$ admit analytic continuations. In the case of $A$, we use  a
traditional notation $R(\zeta, \zeta^*, z, z^*)$ for  the analytic
continuation and for simplicity of notation we omit the dependence
on $\lambda_j$.

The details of the analytic continuation are complicated when the curve is the boundary, and
they simplify when the curve is interior. So we only continue the sketch of the proof in the interior case.

As above, the  arc-length parametrization of
 $C$ is denoed by by
$q_{C}:[0,2\pi] \rightarrow C$ and the corresponding arc-length
parametrization  of the boundary, $\partial \Omega,$ by $q:[0,2\pi
] \rightarrow \partial \Omega$.  Since the boundary and $C$ do not
intersect, the logarithm $\log r^{2}(q(s);q_{C}^{C}(t))$ is well
defined and the
  holomorphic continuation of equation  (\ref{int1a}) is given by:
\begin{equation} \label{int1}
\phi_{\lambda_j}^{\C}(q_C^{\C}(t))= \int_{0}^{2\pi}
 N(\lambda_j,q(s), q_C^{\C}(t)) \, u_{\lambda_j}(q(s)) d\sigma(s),
\end{equation}

From the basic  formula  (\ref{HANKELINT}) for $N(\lambda_j,q,q_{C})$
 and the standard
integral formula for the Hankel function $\Ha_{1}(z)$,
 one easily gets an asymptotic expansion in $\lambda_j$ of the form:
\begin{equation} \label{potential3}
N( \lambda_j,q(s), q_{C}^{\C}(t))  = e^{i \lambda_j r(q(s);
q^{\C}_{C}(t)) }  \, \sum_{m=0}^{k} a_{m}(q(s), q^{\C}_{C}(t)) \, \lambda_j^{1/2 -m}  \end{equation}
$$+ O(e^{i
\lambda_jr( q(s); q^{\C}_{C}(t)) } \,
\lambda_j^{1/2-k-1}).$$
Note that the expansion in (\ref{potential3}) is valid since
for interior curves,
$$C_{0} := \min_{ (q_C(t), q(s)) \in C \times \partial \Omega } |q_C(t) - q(s)|^{2}  >0.$$
 Then,   $\Re r^{2}(q(s);q_{C}^{C}(t)) >0$ as long as
 \begin{equation} \label{holbranch}
 | \Im q^{\C}_{C}(t)|^{2} < C_{0}.
 \end{equation}
  So, the principal square root of $r^{2}$ has a well-defined holomorphic extension to the tube (\ref{holbranch}) containing $C$.  We have denoted this square root by $r$ in (\ref{potential3}).

 Substituting (\ref{potential3})
in the analytically continued single layer potential integral formula (\ref{int1}) proves that
for $t \in A(\epsilon)$ and $\lambda_j>0$ sufficiently large,
\begin{equation} \label{int2}
\phi^{\C}_{\lambda_j}(q_{C}^{\C}(t)) = 2\pi \lambda_j^{1/2} \int_{0}^{2\pi}
 e^{i\lambda_j r (q(s): q_{C}^{\C}(t)) }
 a_{0}(q(s), q_{C}^{\C}(t)) ( 1 + O(\lambda_j^{-1})  \, ) \, u_{\lambda_j}(q(s))  d\sigma(s).
\end{equation}
Taking absolute values of the integral on the RHS in (\ref{int2})
and applying the  Cauchy-Schwartz inequality proves

\begin{lem} \label{mainlemma1}
For $t \in [0,2\pi ] + i[-\epsilon,\epsilon]$ and $\lambda_j>0$ sufficiently large
$$|\phi^{\C}_{\lambda_j}(q_{C}^{\C}(t))| \leq C_{1} \lambda_j^{1/2} \exp \, \lambda_j \left( \max_{q(s) \in \partial \Omega} \Re  \,  i r(q(s);q_{C}^{\C}(t))   \right)  \cdot  \| u_{\lambda_j} \|_{L^{2}(\partial \Omega)}.$$
\end{lem}

From the pointwise upper bounds in Lemma \ref{mainlemma1}, it is immediate that
\begin{equation} \label{DFupper}
\log \, \max_{q^{\C}_{C}(t) \in Q^{\C}_{C}(A(\epsilon))} |\phi_{\lambda_j}^{\C}(q^{\C}_{C}(t) )| \leq C_{\max} \lambda_j + C_{2} \log \lambda_j+  \log \| u_{\lambda_j} \|_{L^{2}(\partial \Omega)},
\end{equation}
where,
$$C_{\max} =  \max_{(q(s),q^{\C}_{C}(t)) \in \partial \Omega \times Q^{\C}_{C}(A(\epsilon))} \Re  \,  i r (q(s);q^{\C}_{C}(t)).$$

Finally, we use that  $\log
\|u_{\lambda_j}\|_{L^{2}(\partial \Omega)} = O( \lambda_j) $ by the assumption that  $C$ is   a good curve and apply Proposition
\ref{DFnew} to get that $n(\lambda_j,C) = O(\lambda_j).$

 The  following estimate,  suggested by Lemma 6.1 of
Donnelly-Fefferman \cite{DF}, gives an upper bound on the number
of zeros in terms of the growth of the family:

\begin{prop} \label{DFnew} Suppose that $C$ is a good real analytic curve
in the sense of (\ref{GOOD}).  Normalize $u_{\lambda_j}$ so that
$||u_{\lambda_j}||_{L^2(C)} = 1$. Then, there exists a constant $C(\epsilon) >0$ such that
for any $\epsilon >0$,
 $$n(\lambda_j, Q_{C}^{\C}( A(\epsilon/2) ) ) \leq C(\epsilon)  \max_{ q_{C}^{\C}(t) \in   Q_{C}^{\C}( A( \epsilon) ) }  \log
|u_{\lambda_j}^{\C}(q_{C}^{\C}(t))| . $$ \end{prop}

\begin{proof} Let $G_{\epsilon}$ denote  the Dirichlet Green's function of the `annulus'
$Q_{C}^{\C}(A(\epsilon)) $. Also, let $\{a_k\}_{k = 1}^{n(\lambda_j, Q_{C}^{\C}(A(\epsilon/2)) )}$
denote the zeros of $u_{\lambda_j}^{\C}$ in the sub-annulus
$Q_{C}^{\C}(A(\epsilon/2))$. Let $U_{\lambda_j} =
\frac{u^{\C}_{\lambda_j}}{||u^{\C}_{\lambda_j}||_{Q_{C}^{\C} (A(\epsilon))}}$ where
$||u||_{Q_{C}^{\C} (A(\epsilon))} = \max_{\zeta \in Q_{C}^{\C} (A(\epsilon))} |u(\zeta)|. $ Then,
$$\begin{array}{lll} \log |U_{\lambda_j}(q_{C}^{\C}(t))| & = &  \int_{ Q_{C}^{\C} ( ( A(\epsilon/2) ) )}
G_{\epsilon}(q_{C}^{\C}(t), w) \ddbar \log |u^{\C}_{\lambda_j}(w)| +
H_{\lambda_j}(q_{C}^{\C}(t)) \\ && \\
& = & \sum_{a_k \in Q_{C}^{\C}( A(\epsilon/2) ): u_{\lambda_j}^{\C}(a_k) = 0} G_{\epsilon}
(q^{\C}_{C}(t), a_k) + H_{\lambda_j}(q_{C}^{\C}(t)), \end{array}$$  since  $\ddbar \log |u^{\C}_{\lambda_j}(w)|  = \sum_{a_k\in C_{\C}: u_{\lambda_j}^{\C}(a_k) = 0} \delta_{a_{k}}$. Moreover, the function
$H_{\lambda_j}$ is  sub-harmonic  on $Q_{C}^{\C} (A(\epsilon))$ since
$$\ddbar H_{\lambda_j} = \ddbar \log |U_{\lambda_j}(q_{C}^{\C}(t))|  - \sum_{a_k \in Q_{C}^{\C}(A(\epsilon/2)): u_{\lambda_j}^{\C}(a_k) = 0}  \ddbar G_{\epsilon}
(q^{\C}_{C}(t), a_k)$$
$$  = \sum_{a_k\in Q_{C}^{\C} (A(\epsilon)) \backslash
 Q_{C}^{\C}(A(\epsilon/2)) } \delta_{a_k} > 0. $$
So, by the maximum principle for subharmonic functions,
$$\max_{Q_{C}^{\C} (A(\epsilon))} H_{\lambda_j} (q_{C}^{\C}(t)) \leq \max_{\partial Q_{C}^{\C} (A(\epsilon))} H_{\lambda_j} (q_{C}^{\C}(t))
= \max_{\partial Q_{C}^{\C} (A(\epsilon))} \log |U_{\lambda_j}(q_{C}^{\C}(t))| = 0. $$   It
follows that
\begin{equation} \log |U_{\lambda_j}(q_{C}^{\C}(t))| \leq \sum_{a_k \in Q_{C}^{\C}(A(\epsilon/2) ): u_{\lambda_j}^{\C}(a_k) = 0} G_{\epsilon}
(q_{C}^{\C}(t), a_k),
\end{equation}
hence that
\begin{equation} \max_{q_{C}^{\C}(t) \in Q_{C}^{\C}( A(\epsilon/2) )} \log
|U_{\lambda_j}(q_{C}^{\C}(t))| \leq \left( \max_{z, w \in Q_{C}^{\C}( A(\epsilon/2) )}
G_{\epsilon}(z,w) \right) \;\; n(\lambda_j, Q_{C}^{\C}( A(\epsilon/2) )).
\end{equation}
Now $G_{\epsilon}(z,w) \leq \max_{w \in Q_{C}^{\C}( \partial A(\epsilon) )}
G_{\epsilon}(z,w) = 0$ and $G_{\epsilon}(z,w) < 0$ for $z,w \in
Q_{C}^{\C}( A(\epsilon/2) )$. It follows that there exists a constant $\nu(\epsilon) <
0$ so that $ \max_{z, w \in Q_{C}^{\C}( A(\epsilon/2) )} G_{\epsilon}(z,w) \leq
\nu(\epsilon). $ Hence,
\begin{equation} \max_{q^{\C}_{C}(t) \in Q^{\C}_{C}( A(\epsilon/2) )} \log
|U_{\lambda_j}(Q^{\C}_{C}(t))| \leq \nu(\epsilon)  \;\; n(\lambda_j, Q^{\C}_{C}( A(\epsilon/2) )).
\end{equation}
Since both sides are negative, we obtain
\begin{equation}\label{mainbound} \begin{array}{lll}  n(\lambda_j,
Q_{C}^{\C}( A(\epsilon/2) ))  \leq \frac{1}{|\nu(\epsilon)|}  \left| \max_{q_{C}^{\C}(t) \in
Q_{C}^{\C}( A(\epsilon/2) )} \log |U_{\lambda_j}(q_{C}^{\C}(t))| \right| \\ && \\
\leq \frac{1}{|\nu(\epsilon)|} \left( \max_{q_{C}^{\C}(t) \in
Q_{C}^{\C} (A(\epsilon))} \log |u^{\C}_{\lambda_j}(q_{C}^{\C}(t))|
- \max_{q_{C}^{\C}(t) \in Q_{C}^{\C}( A(\epsilon/2) )} \log
|u^{\C}_{\lambda_j}(q_{C}^{\C}(t))| \right)\\ && \\
 \leq \frac{1}{|\nu(\epsilon)|} \;\;  \max_{q_{C}^{\C}(t) \in Q_{C}^{\C} (A(\epsilon))} \log
|u^{\C}_{\lambda_j}(q_{C}^{\C}(t))|,
\end{array} \end{equation}
where in the last step we use that $\max_{q^{\C}_{C}(t) \in Q^{\C}_{C}( A(\epsilon/2) )}
\log |u^{\C}_{\lambda_j}(q_{C}^{\C}(t))| \geq 0$, which  holds since
$|u_{\lambda_j}^{\C}| \geq 1$ at some point in $Q_{C}^{\C}( A(\epsilon/2) )$. Indeed,   by our
normalization, $\|u_{\lambda_j}\|_{L^{2}(C)} =1$, and so there must already exist points on the real curve $C$ with $|u_{\lambda_j}| \geq 1$. Putting $C(\epsilon) = \frac{1}{|\nu (\epsilon)|}$ finishes the proof.
\end{proof}
\vspace{1mm}

This completes the proof of Theorem \ref{GOODTH}.

\subsection{\label{PL} Application to Pleijel's conjecture}

I. Polterovich
\cite{Po} observed that  Theorem \ref{INTREALBDYint}  can be used to prove an old conjecture of A. Pleijel
regarding Courant's nodal domain theorem, which says that the
number $n_k$ of nodal domains (components of  $\Omega \backslash
Z_{\phi_{\lambda_k}}$)
 of the $k$th eigenfunction  satisfies $n_k \leq k$.  Pleijel improved this result for Dirichlet eigefunctions
 of plane domains:
For any plane domain with Dirichlet boundary conditions,
$\limsup_{k \to \infty} \frac{n_k}{k} \leq \frac{4}{j_1^2} \simeq
0. 691...$, where $j_1$ is the first zero of the $J_0$ Bessel
function.  He conjectured that the same result should be true for
a free membrane, i.e. for Neumann boundary conditions. This was
recently proved in the real analytic case
 by I. Polterovich \cite{Po}. His argument is roughly the following: Pleijel's original argument applies to
 all nodal domains which do not touch the boundary, since the eigenfunction is a Dirichlet eigenfunction in such
  a nodal domain. The argument does not apply to nodal domains which touch the boundary, but by the Theorem above the number
  of such domains is negligible for the Pleijel bound.

\subsection{Equidistribution of intersections of nodal lines and geodesics on surfaces}

We fix $(x, \xi) \in S^* M$ and let 
\begin{equation} \label{GAMMAX} \gamma_{x, \xi}: \R \to M, \;\;\;\gamma_{x, \xi}(0) = x, \;\;
\gamma_{x, \xi}'(0) = \xi \in T_x M  \end{equation}  denote the corresponding parametrized geodesic. 
Our goal is to  determine the asymptotic  distribution of intersection points of $\gamma_{x, \xi}$ with the nodal 
set of a highly eigenfunction. As usual, we cannot cope with this problem in the real domain and therefore analytically
continue it to the complex domain. Thus, we consider the intersections 
$$\ncal^{\gamma_{x, \xi}^{\C}}_{\lambda_j} = Z_{\phi_{_j}^{\C}}  \cap \gamma_{x, \xi}^{\C} $$  of the complex nodal set  with the (image of the) complexification of a generic geodesic
 If  \begin{equation} \label{SEP} S_{\epsilon} = \{(t
+ i \tau \in \C: |\tau| \leq \epsilon\} \end{equation} then $\gamma_{x, \xi}$ admits an  analytic
continuation
\begin{equation} \label{gammaXCX} \gamma_{x, \xi}^{\C}: S_{\epsilon} \to M_{\epsilon}.  \end{equation}
In other words, we consider the zeros of the pullback, 
$$\{\gamma_{x, \xi}^* \phi_{\lambda}^{\C} = 0\} \subset S_{\epsilon}. $$

 We encode the discrete 
set by the measure
\begin{equation} \label{NCALCURRENT} [\ncal^{\gamma_{x, \xi}^{\C}}_{\lambda_j}] = \sum_{(t + i \tau):\; \phi_j^{\C}(\gamma_{x, \xi}^{\C}(t + i \tau)) = 0} \delta_{t + i \tau}.
\end{equation}

We would like to show that for generic geodesics, the complex zeros on the complexified geodesic condense
on the real points and become uniformly distributed with respect to arc-length. This does not always occur:
as in our discussion of QER theorems, if 
$\gamma_{x, \xi}$ is the fixed point set of an isometric involution,  then ``odd" eigenfunctions under the 
involution will vanish on the geodesic. The additional hypothesis is that QER holds for $\gamma_{x, \xi}$.
 The following
is  proved  (\cite{Ze3}):

\begin{theo}\label{theo} Let $(M^2, g)$ be a real analytic Riemannian surface  with ergodic
geodesic flow. Let $\gamma_{x, \xi}$ satisfy the QER hypothesis.   Then there
exists a subsequence of eigenvalues $\lambda_{j_k}$ of density one
such that for any $f \in C_c(S_{\epsilon})$,
$$\lim_{k \to \infty} \sum_{(t + i \tau):\; \phi_j^{\C}(\gamma_{x, \xi}^{\C}(t + i \tau)) = 0} f(t + i
\tau ) = \int_{\R} f(t) dt.  $$
\end{theo}

In other words,
$$\mbox{weak}^* \lim_{k \to \infty} \frac{i}{\pi \lambda_{j_k}} [\ncal^{\gamma_{x, \xi}^{\C}}_{\lambda_j}] =
 \delta_{\tau = 0}, $$ in the sense of  weak* convergence on
$C_c(S_{\epsilon})$. Thus, the complex nodal set intersects the
(parametrized) complexified geodesic in a discrete set  which is
asymptotically (as $\lambda \to \infty$) concentrated along the
real geodesic  with respect to its arclength.

This concentration- equidistribution result is a `restricted'
version of the result of \S \ref{ERGODICNODAL}. As noted there, the limit distribution of
complex nodal sets in the ergodic case is a singular current $dd^c \sqrt{\rho}$. The motivation
for restricting to geodesics is that restriction magnifies the singularity of this current.  In the
case of a geodesic, the singularity is magnified to a delta-function; for other curves there is additionally
a smooth background measure.

The assumption of ergodicity is crucial. For instance, in the 
case of a flat torus, say $\R^2/L$ where $L \subset \R^2$ is a generic
lattice, the real eigenfunctions are
$\cos \langle \lambda, x \rangle, \sin \langle \lambda,x \rangle$
where $\lambda \in L^*$, the dual lattice, with eigenvalue $-
|\lambda|^2$.  Consider  a geodesic
$\gamma_{x,\xi}(t) = x + t \xi$. Due to the flatness, the
restriction $\sin \langle \lambda, x_0 + t \xi_0 \rangle$  of the
eigenfunction to a geodesic is an eigenfunction of the Laplacian
$-\frac{d^2}{dt^2}$ of  submanifold metric along the geodesic with
eigenvalue $- \langle \lambda, \xi_0 \rangle^2$. The
complexification of the restricted eigenfunction is $\sin \langle
\lambda,  x_0 + (t + i \tau) \xi_0 \rangle |$ and its exponent of
its growth is $\tau |\langle \frac{\lambda}{|\lambda|} , \xi_0
\rangle|$, which can have a wide range of values as the eigenvalue
moves along different rays in $L^*$. The limit current is $i
\ddbar$ applied to the limit and thus also has many limits

The proof  involves several  new principles which played no role in the global
result of \S \ref{ERGODICNODAL}  and which are specific to geodesics. However, the first steps in the proof
are the same as in the global case. 
 By the Poincar\'e-Lelong formula, we may express the
current of summation over the intersection points in \eqref{NCALCURRENT} in the form,
\begin{equation}\label{PLLc}  [\ncal^{\gamma_{x, \xi}^{\C}}_{\lambda_j}] = i \ddbar_{t + i
\tau} \log \left| \gamma_{x, \xi}^* \phi_{\lambda_j}^{\C} (t + i \tau)
\right|^2. \end{equation}  
Thus, the main point of the proof is to determine the asymptotics of  $\frac{1}{\lambda_j} \log \left| \gamma_{x, \xi}^* \phi_{\lambda_j}^{\C} (t + i \tau)
\right|^2$.  When we freeze $\tau$ we put
\begin{equation} \label{gammatau} \gamma_{x, \xi}^{\tau} (t) = \gamma^{\C}_{x, \xi}(t + i \tau). \end{equation}

\begin{prop} \label{MAINPROPa} (Growth saturation) If $\{\phi_{j_k}\}$ satisfies QER along any arcs of $\gamma_{x, \xi}$, then in $L^1_{loc} (S_{\tau}),  $ we have   $$\lim_{k \to \infty} \frac{1}{\lambda_{j_k}} \log \left| \gamma_{x, \xi}^{\tau *} \phi_{\lambda_{j_k}}^{\C} (t + i \tau)
\right|^2 = |\tau|. $$

\end{prop}

Proposition \ref{MAINPROPa} immediately implies Theorem \ref{theo} since we can apply $\ddbar$ to the
$L^1$ convergent sequence $\frac{1}{\lambda_{j_k}} \log \left| \gamma_{x, \xi}^* \phi_{\lambda_{j_k}}^{\C} (t + i \tau)
\right|^2 $ to obtain $\ddbar |\tau|$.

 The upper bound in Proposition \ref{MAINPROPa} follows immediately from the known global estimate 
$$\lim_{k \to \infty} \frac{1}{\lambda_j} \log |\phi_{j_k}(\gamma_{x, \xi}^{\C}(\zeta)| \leq |\tau|$$
on all of $\partial M_{\tau}$.
Hence the difficult point is to prove   that this growth rate is actually obtained
upon restriction to $\gamma_{x, \xi}^{\C}$.  This requires new kinds of arguments related to the QER
theorem.

\begin{itemize}

\item Complexifications of restrictions of eigenfunctions to geodesics have incommensurate Fourier modes,
i.e. higher modes are exponentially larger than lower modes.

\item The quantum ergodic restriction theorem in the real domain  shows that the Fourier
coefficients of the top allowed modes are `large' (i.e. as large as the lower modes).  Consequently, the $L^2$
norms of the complexified eigenfunctions along arcs of $\gamma_{x, \xi}^{\C}$ achieve the lower bound
of Proposition \ref{MAINPROPa}. 

\item Invariance of Wigner measures along the geodesic flow implies that the Wigner measures of restrictions
of complexified eigenfunctions to complexified geodesics should tend to constant multiples of Lebesgue measures
$dt$ for each $\tau > 0$. Hence the eigenfunctions everywhere on $\gamma_{x, \xi}^{\C}$ achieve the growth
rate of the $L^2$ norms.

\end{itemize}

These principles are most easily understood in the case of periodic geodesics. We let
$\gamma_{x, \xi}: S^1 \to M$ parametrize the geodesic with arc-length (where $S^1 = \R/ L \Z$ where
$L$ is the length of $\gamma_{x, \xi}$).

\begin{lem} \label{L2NORMintro} Assume that $\{\phi_j\}$ satsifies QER along the
periodic geodesic $\gamma_{x, \xi}$. Let $||\gamma_{x, \xi}^{\tau*} \phi_j^{\C}||^2_{L^2(S^1)}$ be the $L^2$-norm
of the complexified restriction of $\phi_j$ along $\gamma_{x, \xi}^{\tau}$. Then,
$$\lim_{\lambda_j \to \infty} \frac{1}{\lambda_j} \log ||\gamma_{x, \xi}^{\tau*} \phi_j^{\C}||^2_{L^2(S^1)}
= |\tau| .$$
\end{lem}

To prove Lemma \ref{L2NORMintro}, we study the      orbital Fourier series of $\gamma_{x, \xi}^{\tau*} \phi_j$
and of its complexification. The orbital Fourier coefficients are 
$$\nu_{\lambda_j}^{x, \xi}(n) = \frac{1}{L_{\gamma}} \int_0^{L_{\gamma}} \phi_{\lambda_j}(\gamma_{x, \xi}(t)) e^{- \frac{2 \pi i n t}{L_{\gamma}}} dt, $$
and the orbital Fourier series is 
\begin{equation} \label{PER} \phi_{\lambda_j}(\gamma_{x, \xi}(t) )= \sum_{n \in \Z}  \nu_{\lambda_j}^{x, \xi}(n)  e^{\frac{2 \pi i n t}{L_{\gamma}}}. 
\end{equation}
Hence the analytic continuation of $\gamma_{x, \xi}^{\tau*} \phi_j$  is given by 
\begin{equation} \label{ACPER} \phi^{\C}_{\lambda_j}(\gamma_{x, \xi}(t + i \tau) )= \sum_{n \in \Z}  \nu_{\lambda_j}^{x, \xi}(n)  e^{\frac{2 \pi i n (t + i \tau)}{L_{\gamma}}}. \end{equation}
By the Paley-Wiener theorem for Fourier series, the  series converges absolutely and uniformly for $|\tau| \leq \epsilon_0$. 
By ``energy localization" only the modes with $|n| \leq \lambda_j$ contribute substantially to the $L^2$ norm. We
then observe that the Fourier modes decouple, since they have different exponential growth rates. We use the
QER hypothesis in the following way:

\begin{lem} \label{FCSAT}  Suppose that $\{\phi_{\lambda_j}\}$ is QER along the periodic geodesic $\gamma_{x, \xi}$.
Then for all $\epsilon > 0$, there exists $C_{\epsilon} > 0$ so that
$$\sum_{n: |n| \geq (1 - \epsilon) \lambda_j}  |\nu_{\lambda_j}^{x, \xi}(n)|^2 \geq  C_{\epsilon}. $$

\end{lem}

 Lemma \ref{FCSAT}  implies Lemma \ref{L2NORMintro} since it implies that for any $\epsilon > 0$,
$$\sum_{n: |n| \geq (1 - \epsilon) \lambda_j}  |\nu_{\lambda_j}^{x, \xi}(n)|^2 e^{-2 n \tau}  \geq  C_{\epsilon} e^{2\tau(1 - 
\epsilon) \lambda_j}. $$

To go from asymptotics of $L^2$ norms of restrictions to Proposition \ref{MAINPROPa} we then
use the third principle:

\begin{prop} \label{LL} (Lebesgue limits)  If
 $\gamma_{x, \xi}^* \phi_j \not=  0$ (identically),  then for all $\tau > 0$ the
 sequence 
$$U_j^{x, \xi, \tau} = \frac{\gamma_{x, \xi}^{\tau *} \phi_j^{\C}}{||{\gamma_{x, \xi}^{\tau *} \phi_j^{\C}}||_{L^2(S^1)} }$$
  is QUE  with limit measure given by normalized  Lebesgue measure on $S^1$. \end{prop}
The proof of Proposition \ref{MAINPROPa} is completed by combining Lemma \ref{L2NORMintro}  and
Proposition \ref{LL}.   Theorem \ref{theo} follows easily from Proposition \ref{MAINPROPa}.

The proof for non-periodic geodesics is considerably more involved, since one cannot use Fourier analysis
in quite the same way.

\subsection{Real zeros and complex analysis}

\begin{mainprob}

An important but apparently rather intractable problem is, how to obtain information
on the real zeros from knowledge of the complex nodal distribution? There are several
possible approaches:

\begin{itemize}

\item Try to intersect the nodal current with the current of integration over the real
points $M \subset M_{\epsilon}$. I.e. try to slice the complex nodal set with the
real domain.
\bigskip

\item Thicken the real slice slightly by studying the behavior of the nodal
set in $M_{\epsilon}$ as $\epsilon \to 0$. The sharpest version is to try to re-scale the nodal set by a factor of $\lambda^{-1}$ to zoom in on the zeros
which are within $\lambda^{-1}$ of the real domain. They may not be real but at least
one can control such ``almost real" zeros.  Try to understand (at least in real dimension 2)
how the complex nodal set  `sprouts' from the real nodal set. How do the connected
components of the real nodal set fit together in the complex nodal set? 
\bigskip

\item Intersect the nodal set with geodesics. This magnifies the singularity along the real 
domain and converts nodal sets to isolated points. 

\end{itemize}
\end{mainprob}

\section{\label{Lp} $L^p$ norms of eigenfuncions}

In \S \ref{LBL1} we pointed out that lower bounds on $||\phi_{\lambda}||_{L^1}$ lead to improved
lower bounds on Hausdorff measures of nodal sets. In this section we consider  general
$L^p$-norm problems for eigenfunctions.

\subsection{Generic upper bounds on $L^p$ norms}

We have already explained that the pointwise Weyl law \eqref{PLWL} and remainder jump estimate \eqref{03} leads to the general
sup norm bound for $L^2$-normalized eigenfunctions, 
\begin{equation} \label{SUPNORM} ||\phi_{\lambda}||_{L^{\infty}} \leq C_g \lambda^{\frac{m-1}{2}}, \;\; (m = \dim M). 
\end{equation}
The upper bound  is  achieved  by  zonal spherical harmonics. In \cite{Sog} (see also \cite{Sogb, Sogb2})
C.D. Sogge proved general $L^p$ bounds: 

\begin{theo}\label{SOGTH}  (Sogge, 1985) 
\begin{equation}\label{i.8}
\sup_{\phi\in
V_\lambda}\frac{\|\phi\|_p}{\|\phi\|_2}=O(\lambda^{\delta(p)}),
\quad 2\le p\le \infty
\end{equation}

 where
\begin{equation} \delta(p)=
\begin{cases}
n(\tfrac12-\tfrac1p)-\tfrac12, \quad \tfrac{2(n+1)}{n-1}\le p\le
\infty
\\
\tfrac{n-1}2(\tfrac12-\tfrac1p),\quad 2\le p\le
\tfrac{2(n+1)}{n-1}.
\end{cases}
\end{equation}

\end{theo}

The upper bounds are sharp in the class of all $(M, g)$ and are saturated on
the round sphere: \bigskip

\begin{itemize}

\item For $p>\tfrac{2(n+1)}{n-1}$, 
zonal (rotationally invariant) spherical harmonics saturate the $L^p$ bounds. Such eigenfunctions
also occur on surfaces of revolution. \bigskip

\item For  $L^p$ for $2\leq p\leq \tfrac{2(n+1)}{n-1}$  the bounds
are saturated by highest weight spherical harmonics, i.e.  Gaussian beam along a
stable elliptic geodesic. Such eigenfunctions also occur on surfaces of revolution.

\end{itemize}

The zonal has high $L^p$ norm due to its high peaks on balls of radius $\frac{1}{N}$. The balls are so small
that they do not have high $L^p$ norms for small p. The Gaussian beams are not as high but they are relatively
high over an entire geodesic.

\subsection{Lower bounds on $L^1$ norms}

The $L^p$ upper bounds are the only known tool for obtaining lower bounds on $L^1$ norms. 
We now prove   Proposition \ref{CS}:
\begin{proof}

Fix a function $\rho\in {\mathcal S}({\mathbb R})$ having the properties
that $\rho(0)=1$ and $\hat \rho(t)=0$ if $t\notin [\delta/2,\delta]$, where
$\delta>0$ is smaller than the injectivity radius of $(M,g)$.
If we then set
$$T_\lambda f = \rho(\sqrt{-\Delta}-\lambda)f,$$
we have that $T_\lambda \phi_\lambda=\phi_\lambda$.  Also,
by Lemma 5.1.3 in \cite{Sogb}, $T_\lambda$ is an oscillatory
integral operator of the form
$$T_{\lambda} f(x) = \lambda^{\frac{n-1}{2}} \int_M e^{i \lambda
r(x,y)} a_\lambda(x, y) f(y) dy,$$
with $|\partial^\alpha_{x,y}a_\lambda(x,y)|\le C_\alpha$.
Consequently,
$||T_{\lambda}
\phi_{\lambda}||_{L^\infty} \le C \lambda^{\frac{n-1}{2}}
||\phi_{\lambda}||_{L^1}$, with $C$ independent of $\lambda$,
and so
$$1 = ||\phi_{\lambda}||_{L^2}^2 = \langle T \phi_{\lambda},
\phi_{\lambda} \rangle \leq ||T \phi_{\lambda}||_{L^\infty}
||\phi_{\lambda}||_{L^1} \leq  C \lambda^{\frac{n-1}{2}}
||\phi_{\lambda}||_{L^1}^2. $$

We can give another proof based on the eigenfunction
estimates  (Theorem \ref{SOGTH}),  which say that
$$\|\phi_\lambda\|_{L^p}\le C \lambda^{\frac{(n-1)(p-2)}{4p}},
\quad 2<p\le \tfrac{2(n+1)}{n-1}.
$$
If we pick such a $2<p<\tfrac{2(n+1)}{n-1}$, then by H\"older's inequality, we have
$$1=\|\phi_\lambda\|_{L^2}^{1/\theta}\le \|\phi_\lambda\|_{L^1} \, \|\phi_\lambda\|_{L^p}^{\frac1\theta-1}\le  \|\phi_\lambda\|_{L^1}\bigl(\, C\lambda^{\frac{(n-1)(p-2)}{4p}}\, \bigr)^{\frac1\theta-1},
\quad \theta=\tfrac{p}{p-1}(\tfrac12-\tfrac1p)=\tfrac{(p-2)}{2(p-1)},$$
which  implies $\|\phi_\lambda\|_{L^1}\ge c\lambda^{-\frac{n-1}4}$, since
$(1-\tfrac1\theta) \tfrac{(n-1)(p-2)}{4p}=\tfrac{n-1}4$.
\end{proof}

We remark that this lowerbound for $\|\phi_\lambda\|_{L^1}$ is sharp on the standard sphere, since
$L^2$-normalized highest weight spherical harmonics of degree $k$ with eigenvalue
$\lambda^2 = k(k+n-1)$ have $L^1$-norms which are bounded above
and below by $k^{(n-1)/4}$ as $k\to \infty$.  Similarly, the $L^p$-upperbounds that we used in the
second proof of this $L^1$-lowerbound is also sharp because of these functions.

\bigskip

\subsection{Riemannian manifolds with maximal eigenfunction growth}

Although the general sup norm bound \eqref{SUPNORM} is achieved by some sequences of
eigenfunctions on some Riemannian manifolds (the standard sphere or a surface of revolution),
it is very rare that $(M, g)$ has such sequences of eigenfunctions. We say that such $(M, g)$
have maximal eigenfunction growth. In a series of articles \cite{SoZ,STZ, SoZ2}, ever more
stringent conditions are given on such $(M, g)$. We now go over the results.

Denote the eigenspaces by
$$V_{\lambda} = \{\phi: \Delta \phi = -\lambda^2 \phi\}.$$
 We measure the growth rate of $L^p$ norms by 
 \begin{equation} L^{p}(\lambda, g) = \sup_{\phi\in V_{\lambda}: ||\phi||_{L^2}
= 1 }
 ||\phi||_{L^{p}}. \end{equation}
\bigskip

\begin{defin} Say that $(M, g)$ has maximal $L^p$ eigenfunction growth if it possesses a sequence
of eigenfunctions $\phi_{\lambda_{j_k}}$ which saturates the $L^p$ bounds. When $p = \infty$ we say
that it has maximal sup norm growth. \end{defin}

\begin{prob}

\begin{itemize}

\item Characterize $(M, g)$  with maximal $L^{\infty}$ eigenfunction growth. The same sequence
of eigenfunctions should saturate all $L^p$ norms with $p \geq p_n:=   \tfrac{2(n+1)}{n-1}$. 
\bigskip

\item Characterize $(M, g)$  with maximal $L^{p}$ eigenfunction growth for $2 \leq p \leq  \tfrac{2(n+1)}{n-1}. $
\bigskip

\item 
Characterize $(M, g)$ for which $||\phi_{\lambda}||_{L^1} \geq C > 0$. 

\end{itemize}

\end{prob}

In \cite{SoZ},  it was shown that $(M, g)$ of maximal $L^{p}$ eigenfunction growth for $p \geq p_n$
have self-focal points. The terminology is non-standard and several different terms are used.

\begin{defin}

 We call a point $p$ a {\it self-focal point} or {\it blow-down point} if all geodesics leaving
$p$ loop back to $p$ at a common time $T$. That is, $\exp_p  T\xi = p$ (They do not have to be closed geodesics.)

We call a point $p$ a partial self-focal point if there exists a positive measure in $S^*_x M$
of directions $\xi$ which loop back to $p$. 

\end{defin}

The poles of  a surface of revolution are self-focal and all geodesics close up smoothly (i.e. are
closed geodesics). The umbilic points of an ellipsoid are self-focal but only two directions give
smoothly closed geodesics (one up to time reversal).

\begin{center}
\includegraphics[scale=0.6]{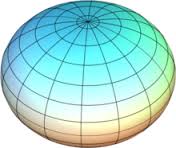}\hspace{1cm} \includegraphics[scale=0.8 ]{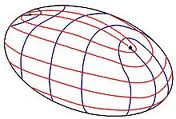}
\end{center}

In \cite{SoZ} is proved:

\begin{theo}   \label{SoZthm} Suppose $(M, g)$ is a $C^{\infty}$ Riemannian manifold with
maximal eigenfunction growth, i.e. having a sequence $\{\phi_{\lambda_{j_k}}\}$ of eigenfunctions
which achieves (saturates) the bound $||\phi_{\lambda_{j_k}}||_{L^{\infty}} \geq C_0 \lambda_{j_k}^{(n-1)/2}$
for some $C_0 > 0$ depending only on $(M, g)$.

Then there  must exist a point $x \in M$ for which the set
\begin{equation}  {\mathcal L}_x = \{ \xi \in S^*_xM : \exists T: \exp_x T \xi = x\} \end{equation}
of directions of geodesic loops at $x$ has positive 
measure in $S^*_x M$.  Here, $\exp$ is the exponential map, and the measure
$|\Omega|$ of a set $\Omega$ is the one  induced by the metric
$g_x$ on $T^*_xM$. For instance, the poles $x_N, x_S$ of a surface
of revolution $(S^2, g)$ satisfy $|{\mathcal L}_x| = 2 \pi$.
\end{theo}

Theorem \ref{SoZthm}, Theorem  \ref{TL}, as well as the
 results of \cite{SoZ,
STZ},  are proved by studying the remainder term $R(\lambda, x)$  in the pointwise Weyl law,
\begin{equation} \label{LWL} N(\lambda, x) = \sum_{j: \lambda_j \leq \lambda} |\phi_j(x)|^2 = C_m \lambda^m + R(\lambda, x). \end{equation}  The first term $N_W(\lambda) = C_m \lambda^m $ is called the Weyl term.
It is classical that the remainder is of one lower order, $R(\lambda, x) = O(\lambda^{m-1})$.
The  relevance of the remainder  to maximal eigenfunction growth is through the following  well-known Lemma
(see e.g.  \cite{SoZ}):

\begin{lem}\label{Reig} Fix $x\in M$. Then if $\lambda\in
\text{spec }\sqrt{-\Delta}$
\begin{equation}\label{eig3}
\sup_{\phi\in
V_\lambda}\frac{|\phi(x)|}{\|\phi\|_2}=\sqrt{R(\lambda,x)-R(\lambda-0,x)}.
\end{equation} 

\end{lem}
\noindent Here, for a right continuous function $f(x)$ we denote by $f(x + 0) -
f(x - 0)$ the jump of $f$ at $x$.
Thus, Theorem \ref{SoZthm} follows from

\begin{theo}\label{maintheorem} Let $R(\lambda,x)$ denote
the remainder for the local Weyl law at $x$.  Then
\begin{equation}\label{M20}
R(\lambda,x)=o(\lambda^{n-1}) \, \, \text{if } \, \, |\lcal_x|=0.
\end{equation}
Additionally, if $|{\lcal}_x|=0$ then, given $\varepsilon>0$,
there is a neighborhood ${\ncal}$ of $x$ and a $\Lambda=
<\infty$, both depending on $\varepsilon$ so that
\begin{equation}\label{M3}
|R(\lambda,y)|\le \varepsilon \lambda^{n-1}, \, \, y\in {\ncal},
\, \, \lambda\ge \Lambda.
\end{equation}
\end{theo}

\subsection{Theorem \ref{TL}}
However, Theorem \ref{SoZthm} is not sharp: on  a tri-axial ellipsoid (three distinct axes), the  umbilic points are
self-focal points. But the eigenfunctions which maximize the sup-norm only have $L^{\infty}$ norms
of order $\lambda^{\frac{n-1}{2}}/\log \lambda$. 
An improvement is given in \cite{STZ}.

Recently, Sogge and the author have further improved the result in the case of real analytic $(M,g)$
In this case $|\lcal_x| > 0$ implies that $\lcal_x = S^*_x M$ and the geometry simplifies. In
\cite{SoZ2}, we prove Theorem \ref{TL}, which we restate in terms of the jumps of the remainder:

\begin{theo} \label{REM}  Assume that  $U_x$ has no invariant
$L^2$ function for any $x$.  Then  $$N(\lambda + o(1), x) - N(\lambda, x) =
o(\lambda^{n-1}), \; \mbox{ uniformly in x}. $$
Equivalently, 
\begin{equation} \label{REMEST}  R(\lambda + o(1), x) - R(\lambda, x) =
o(\lambda^{n-1}) \end{equation}
uniformly in $x$.
  \end{theo}


Before discussing the proof we note that the conclusion gives very stringent conditions on $(M, g)$. 
First, there  are topological restrictions on manifolds possessing a self-focal point. 
If  $(M, g)$ has a focal point $x_0$ then the rational cohomology $H^*(M, \Q)$ has a single generator (Berard-Bergery). 
But even in this case there are many open problems:
\bigskip

\begin{prob}
All known examples of $(M, g)$ with maximal eigenfunction growth have completely integrable geodesic
flow, and indeed quantum integrable Laplacians. Can one prove that maximum eigenfunction growth only
occurs in the integrable case? Does it only hold if there exists a point $p$ for which $\Phi_p = Id$?

A related purely geometric problem: Do there exist  $(M, g)$ with $\dim M \geq 3$ possessing
 self-focal points wth $\Phi_x  \not= Id$. I.e. do there exist
generalizations of umbilic points of ellipsoids in dimension two. There do not seem to exist any known
examples; higher dimensional ellipsoids do not seem to have such points.

\end{prob}


Despite these open questions, Theorem \ref{REM} is in a sense sharp. If there exists a self-focal
point $p$ with a smooth invariant function, then one can construct a  {\it quasi-mode} of order zero which lives
on the flow-out Lagrangian
$$\Lambda_p : = \bigcup_{t \in [0, T]} G^t S^*_p M $$
where $G^t$ is the geodesic flow and $T$ is the minimal common return time. The `symbol' is the flowout
of the  smooth
invariant density.

\begin{prop} \label{CONVERSE} Suppose that $(M, g)$ has a point $p$ which is a self-focal point whose first return map $\Phi_x$
at the return time $T$ 
is the identity map of $S^*_p M$. Then there exists a quasi-model of order zero associated to the sequence
$\{\frac{2}{\pi} T k + \frac{\beta}{2}: k = 1, 2, 3, \dots\}$ which concentrates microlocally on the flow-out
of $S^*_p M$. (See \S \ref{CONVERSE} for background and more precise information).

\end{prop}


\subsection{Sketch of proof of Theorem \ref{TL}}

We first outline the proof. A key issue is the uniformity of remainder estimates of $R(\lambda, x)$
as $x$ varies. Intuitively it is obvious that the main points of interest are the self-focal points. 
But at this time of writing, we cannot exclude the possibility, even in the real analytic setting, that there are an infinite number of such
points with twisted return maps. Points which isolated from the set of self-focal points are easy to deal with,
but there may be non-self-focal points which lie in the closure of the self-focal points. 
We introduce some notation.

\begin{defin} \label{POINTS} We say that $x \in M$

\begin{itemize}

\item is an $\lcal$ point ($x \in \lcal$)  if 
$\lcal_x = \pi^{-1}(x) \simeq S^*_x M$. Thus, $x$ is a self-focal point. 

\item is a $\ccal \lcal$ point $(x \in \ccal \lcal$)  if $x \in \lcal$ and $\Phi_x = Id$. Thus,
all of the loops at $x$ are smoothly closed.

\item is a $\tcal \lcal$ point $(x \in \tcal \lcal)$ if $x \in \lcal$ but $\Phi_x \not= Id$, i.e. $x$
is a twisted self-focal point. Equivalently,   $\mu_x
\{\xi \in \lcal_x: \Phi_x(\xi) = \xi \} = 0; $ All directions
are loop directions, but  almost none are directions of smoothly closed loops.

\end{itemize}

\end{defin}

To prove Theorem \ref{TL} we may (and henceforth will) assume that $\ccal \lcal = \emptyset$. Thus,
$\lcal = \tcal \lcal$. We also let $\overline{\lcal}$ denote the closure of the set of self-focal points. At this
time of writing, we do not know how to exclude that $\overline{\lcal} = M$, i.e. that the set of self-focal
points is dense. 

\begin{prob} Prove (or disprove) that if $\ccal \lcal = \emptyset$ and if $(M, g)$ is real analytic, then $\lcal$ is a finite
set. \end{prob}

We also need to further specify times of returns. It is well-known and easy to prove that if all $\xi \in \pi^{-1} (x)$
are loop directions, then the time $T(x, \xi)$ of first return is constant on $\pi^{-1}(x)$. This is because an analytic function
is constant on its critical point set. 

\begin{defin} \label{QUANTDEF} We say that $x \in M$

\begin{itemize}

\item is a $\tcal \lcal_T$ point $(x \in \tcal \lcal_T$)  if  $x \in \tcal \lcal$ and if  $ T(x,
\xi) \leq T$ for all $\xi \in \pi^{-1}(x) $.  We
denote the set of such points by $\tcal \lcal_T$.

\end{itemize}

\end{defin}

\begin{lem} \label{FINITE} If $(M, g)$ is real analytic, then $\tcal \lcal_T$ is a finite set. \end{lem}

There are several ways to prove this. One is to consider the set of all loop points,
$$ \ecal = \{(x, \xi) \in TM: \exp_x \xi = x \},$$
where as usual we identify vectors and co-vectors with the metric. Then at a self-focal point $p$, 
$\ecal \cap T_p M$ contains a union of spheres of radii $k T(p)$, $k = 1, 2, 3, \dots$. The condition
that $\Phi_p \not= I$ can be used to show that each sphere is a component of $\ecal$, i.e. is isolated
from the rest of $\ecal$. Hence in the compact set $B_T^* M = \{(x, \xi): |\xi| \leq T\}$, there can only
exist a finite number of such components. Another way to prove it is to show that any limit point 
$x$, with  $p_j \to x$ and  $p_j \in \tcal \lcal_T$  must be a $\tcal \lcal_T$ point whose first
return map is the identity.  Both proofs involve the study of Jacobi fields along the looping geodesics.










To outline the proof,
let $\hat{\rho}
\in C_0^{\infty}$ be an even function  with $\hat{\rho}(0) = 1$, $\rho(\lambda) > 0$
for all $\lambda \in \R$, and $\hat{\rho}_T(t) = \hat{\rho}(\frac{t}{T})$. 
The classical cosine Tauberian method to determine  Weyl asymptotics  is to study 

One starts from the smoothed spectral expansion \cite{DG,SV}
\begin{equation} \label{rhoTFORM1}  \begin{array}{lll} \rho_T * d N(\lambda, x)   & = & \int_{\R} \hat{\rho}(\frac{t}{T}) e^{i \lambda t} U(t, x, x) dt \\&&\\ & = &
 a_0 \lambda^{n-1}
+ a_1 \lambda^{n-2} +  \lambda^{n-1}\sum_{j =
1}^{\infty}  R_j(\lambda, x, T) + o_{T}(\lambda^{n - 1}), 
\end{array} \end{equation}
with uniform remainder in $x$. The sum over $j$ is a sum over charts needed to parametrize the
canonical relation of $U(t, x, y)$, i.e. the graph of the geodesic flow. By the usual parametrix construction  for $U(t) = e^{i t\sqrt{\Delta}}$, one proves that 
there exist   phases $\tilde{t}_j$  and amplitudes $a_{j0}$ such that
\begin{equation} \label{Rj}  \begin{array}{lll}  R_j(\lambda, x, T)   
&\simeq &  \lambda^{n-1}  \int_{S^*_x
M} e^{i \lambda\tilde{t}_j(x, \xi) }   \left( (\hat{\rho}_T a_{j0} ) \right) |d \xi| + O(\lambda^{n-2}). \end{array} \end{equation}
As in \cite{DG,Saf,SV}  we  use polar coordinates in $T^* M$, and stationary
phase in $dt dr$ to reduce to integrals over $S^*_x M$.   The phase $\tilde{t}_j $ is the value of the phase
$\phi_j(t, x, x, \xi)$ of $U(t, x, x)$ at the critical point. The loop directions are those $\xi$ such
that $\nabla_{\xi} \tilde{t}_j(x, \xi) = 0$.

\begin{mainex} Show that $\rho_T * dN(\lambda, x)$ is a semi-classical Lagrangian distribution
in the sense of \S \ref{LAGAPP}.  What is its principal symbol?

\end{mainex}

To illustrate the notation, we consider   a 
 flat torus $\R^n/\Gamma$ with $\Gamma \subset \R^n$ a full rank lattice. As is well-known, the wave kernel
then has the form
$$U(t, x, y) = \sum_{\gamma \in \Gamma}  \int_{\R^n} e^{i \langle x - y - \gamma, \xi \rangle} e^{it |\xi|} d \xi. $$
Thus,  the indices $j$ may be taken to 
be the lattice points $\gamma \in \Gamma$, and
$$\begin{array}{lll} \rho_T * d N(\lambda, x)  & = &  \sum_{\gamma \in \Gamma} \int_{\R} \int_0^{\infty} \int_{S^{n-1}}
\hat{\rho}(\frac{t}{T}) e^{i  r \langle \gamma, \omega \rangle} e^{i t r}  e^{- i t \lambda} r^{n-1} dr dt d \omega
\end{array}$$
We change variables $r \to \lambda r$ to get a full phase $\lambda (r \langle \gamma, \omega \rangle
+ t r - t)$ . The stationary phase points in $(r, t)$ are $\langle \gamma, \omega \rangle = t$ and
$r = 1$.  Thus,
$$\tilde{t}_{\gamma}(x, \omega) = \langle \gamma, \omega \rangle. $$
 
The geometric interpretation of $t_{\gamma}^*(x, \omega)$ is that it is the value of $t$ for which the geodesic
$\exp_x t \omega = x + t \omega$  comes closest to the representative   $x + \gamma$ of $x$
in the $\gamma$th chart.   Indeed, 
 the line $x + t \omega$  is `closest'  to $x + \gamma$  when
 $t \omega$ closest to $\gamma$, since  $|\gamma - t \omega|^2 = |\gamma|^2 - 2t \langle \gamma, \omega \rangle + t^2. $
On a general $(M, g)$ without conjugate points, 
$$ \tilde{t}_{\gamma}(x, \omega) =  \langle \exp_x^{-1} \gamma x, \omega \rangle. $$

\subsection{Size of the remainder at a self-focal point}

The first key observation is that \eqref{Rj} takes a special form at a self-focal point. 
At a self-focal point $x$ define $U_x$ as in \eqref{UX}.  
Also define
\begin{equation} \label{UXLAMBDA} U_x^{\pm} (\lambda) = e^{i \lambda T_x^{\pm}} U_x^{\pm}.  \end{equation}
The following observation is due to Safarov \cite{Saf} (see also \cite{SV}).

\begin{lem}

Suppose that $x$ is a self-focal point. If  $\hat{\rho} = 0$ in a
neighborhood of $t = 0$ then
\begin{equation} \label{SAFFORM} \rho_T' * N(\lambda, x) = \lambda^{n-1} \sum_{k \in
\Z \backslash 0} \int_{S^*_x M}  \hat{\rho}(\frac{k
T(\xi)}{T}) \overline{U_{x}(\lambda)^k \cdot 1} d \xi +
O(\lambda^{n-2}).
\end{equation}
\end{lem}


Here is the main result showing that $R(\lambda, x)$ is small at the self-focal points if there
do not exist invariant $L^2$ functions. $T_x^{(k)}(\xi)$ is the $k$th return time of $\xi$ for $\Phi_x$.

\begin{prop}\label{MAINPROP} Assume that $x$ is a self-focal point and that $U_x$ has no invariant $L^2$ function. Then,  for all $\eta
> 0$, there exists $T$ so that \begin{equation} \label{INEQ3}  \frac{1}{T} \left| \int_{S^*_x M} \sum_{k =
0}^{\infty} \hat{\rho} (\frac{T_x^{(k)}(\xi)}{T}) U_x^k \cdot 1 |d\xi|
\right| \leq \eta.   \end{equation}
\end{prop}

This is a simple application of the von Neumann mean ergodic theorem to the unitary operator $U_x$.
Indeed, $\frac{1}{N} \sum_{k = 0}^N U_x ^k \to P_x,$ where $P_x: L^2(S^*_x M) \to L^2_0(S^*_x M)$ is
the orthogonal projections onto the invariant $L^2$ functions for $U_x$. By our assumption, $P_x = 0$.

Proposition \ref{MAINPROP} is not apriori uniform as $x$ varies over self-focal points, since there is no
obvious relation between $\Phi_x$ at one self-focal point and another. It would of course be uniform
if we knew that there only exist a finite number of self-focal points. As mentioned above, this is
currently unknown. However, there is a second mechanism
behind Proposition \ref{MAINPROP}. Namely, if the first common return time $T_x^{(1)}(\xi)$ is larger than $T$,
then there is only one term $k = 0$  in the sum with the cutoff $\hat{\rho}_T$ and the sum is $O(\frac{1}{T})$. 

\subsection{Decomposition of the remainder into almost loop directions and far from loop directions}

We now consider non-self-focal points. Then the function $\tilde{t}_j(x,\xi)$ has almost no critical
points in $S^*_x M$.

 Pick  
  $f \geq 0 \in C_0^{\infty}(\R)$ which equals $1$
on $|s| \leq 1$ and zero for $|s| \geq 2$ and split up the $j$th term      into two terms using   $f(\epsilon^{-2} |\nabla_{\xi} \tilde{t}_j|^2) $ and
$1 - f(\epsilon^{-2} |\nabla_{\xi} \tilde{t}_j|^2) $:
\begin{equation} \label{DECOMPRj}  R_j(\lambda, x, T)  = R_{j1}(\lambda, x, T, \epsilon) + R_{j2}(\lambda, x, T, \epsilon), \end{equation}
where
\begin{equation} \label{Rj1}  \begin{array}{l} R_{j1}(\lambda, x, T, \epsilon) : =  \int_{S^*_x
M} e^{i \lambda \tilde{t}_j } f(\epsilon^{-1} |\nabla_{\xi}
\tilde{t}_j(x, \xi)|^2)
(\hat{\rho} (T_x(\xi))) a_{0}(T_x(\xi), x,  \xi)  d \xi
\end{array} \end{equation}
The second term  $R_{j2}$ comes from the  $1 -
f(\epsilon^{-2} |\nabla_{\xi} T_x(\xi)|^2) $ term. By one integration by parts, one easily has

\begin{lem} \label{LOG1}   For all $T>0$ and  $\epsilon \geq \lambda^{-\half} \log \lambda$ we have
$$\sup_{x \in M} |R_{2}(\lambda, x, T,\epsilon)| \leq C (\epsilon^{2} \lambda)^{-1} .$$
 \end{lem}

The $f$ term
involves the contribution of the almost-critical points of $\tilde{t}_j$.  They are estimated
by the measure of the almost-critical set.

\begin{lem}\label{AC}  There exists a  uniform positive constant  $C$ so that for all $(x, \epsilon)$,
\begin{equation} |R_{j1}(x; \epsilon) | \leq C   \mu_x\left(
\{\xi: 0 < |\nabla_{\xi} \tilde{t}_j(x, \xi)|^2 < \epsilon^2\}
\right), \end{equation}
\end{lem}

\subsection{Points in $M\backslash \overline{\tcal \lcal}$}

If $x$ is isolated from $\overline{\tcal \lcal}$ then there is a uniform bound on the size of the remainder near $x$.

\begin{lem} Suppose that $x \notin \overline{\tcal \lcal}$. Then given $\eta > 0$ there  exists a
ball $B(x, r(x, \eta))$ with radius $r(x, \eta) > 0$ and $\epsilon > 0$ so that
$$\sup_{y \in B(x, r(x, \eta))} |R(\lambda, y, \epsilon) | \leq \eta. $$
\end{lem}

Indeed, we pick $r(x, \eta)$ so that the closure of $B(x, r(x, \eta))$   is disjoint from $\overline{\tcal \lcal}$. Then
the one-parameter family of functions   $F_{\epsilon}(y) =  \to  \mu_y\left(
\{\xi \in S^*_y M: 0 < |\nabla_{\xi} \tilde{t}_j(y, \xi)|^2 < \epsilon^2\}
\right)$ is decreasing to zero as $\epsilon \to 0$ for each $y$. By Dini's theorem, the family tends to zero
uniformly on $\overline{B(x, r(x, \eta))}$.

\subsection{Perturbation theory of the remainder}

So far, we have good remainder estimates at each self-focal point and in balls around points isolated
from the self-focal points. We still need to deal with the uniformity issues as $p$ varies among self-focal
points and  points in  $\overline{\tcal \lcal}$.

We now compare remainders at nearby points. Although $R_j(\lambda, x, T)$ is oscillatory, the estimates
on $R_{j 1}$ and the ergodic estimates do not use the oscillatory factor $e^{i \lambda \tilde{t}}$, which 
in fact is only used in Lemma \ref{LOG}. Hence we compare absolute remainders $|R|(x, T)|$, i.e. where we
take the absolute under the integral sign. They are independent of $\lambda$. The integrands of the remainders
vary smoothly with the base point and only involve integrations over different fibers $S^*_x M$ of 
$S^* M \to M$. 

\begin{lem} We have,
$$| |R|(x,  T) - |R|(y, T) |\leq C e^{a T}  \mbox{dist}(x,y). $$

\end{lem}

Indeed, we write the difference as the integral of its derivative. The derivative involves the change in
$\Phi^n_x$ as $x$ varies over iterates up to time $T$ and therefore is estimated by the sup norm $e^{a T}$ of
the first derivative of the geodesic flow up to time $T$. If we choose a ball of radius $\delta e^{- a T}$ around a focal point, we obtain'
\begin{cor} For any $\eta >0$, $T > 0$  and any focal point $p \in \tcal \lcal$ there exists $r(p, \eta)$ so that
$$\sup_{y \in \overline{B(p, r(p, \eta))} }|R(\lambda, y, T)| \leq \eta. $$
\end{cor}

To complete the proof of Theorem \ref{TL} we prove 
\begin{lem} Let $ x \in \overline{\tcal \lcal} \backslash \tcal \lcal$. Then for any $\eta >0$
there exists $r(x, \eta) > 0$ so that 
$$\sup_{y \in \overline{B(x, r(x, \eta))}} |R(\lambda, y, T)| \leq \eta. $$ \end{lem}

Indeed, let $p_j \to x$ with $T(p_j) \to \infty$. Then the remainder is given at each $p_j$ by
the left side of  \eqref{SAFFORM}. But for any fixed $T$, the first term of \eqref{SAFFORM}  has 
at most one term for $j$ sufficiently large. Since the remainder is continuous, the remainder at $x$
is the limit of the remainders at $p_j$ and is therefore $O(T^{-1}) + O(\lambda^{-1})$. 

  By the perturbation estimate, one has the same remainder estimate in a sufficiently small ball around $x$. 

\subsection{Conclusions}

\begin{itemize}

\item No eigenfunction $\phi_j(x)$ can be maximally large at a point $x$ which is $\geq \lambda_j^{-\half} \log \lambda_j$
away from the self-focal points. \bigskip

\item When there are no invariant measures, $\phi_j$  also cannot be large at a self-focal point. \bigskip

\item If $\phi_j$  is not large at any self-focal point, it is also not large nera a self-focal point.

\end{itemize}




\section{\label{GEOFLOWAPP} Appendix on the phase space and the geodesic flow}

Classical mechanics takes place in phase space $T^* M$ (the cotangent bundle). Let $(x, \xi)$ be Darboux  coordinates on $T^* M$, i.e. $x_j$ are local coordinates on $M$ and $\xi_j$ are the functions on $T^*M$ which pick out components
with respect to $dx_j$.  Hamilton's equations for the Hamiltonian $H(x, \xi) = \half |\xi|^2 + V(x): T^* M \to \R$ are
$$\frac{dx}{dt} = \frac{\partial H}{\partial \xi}, \;\;\;  \frac{d \xi}{dt} =- \frac{\partial H}{\partial \xi}. $$
The Hamilton flow is $\Phi^t(x, \xi) = (x_t, \xi_t)$ with $(x_0, \xi_0) = (x, \xi)$.
Let us recall the symplectic interpreation of Hamilton's equations.

The cotangent bundle $T^*M$ of any manifold $M$ carries a
canonical action form $\alpha$ and symplectic form $\omega = d
\alpha$.  Given any local coordinates $x_j$ and associated frame
$dx_j$ for $T^*M$, we put: $\alpha = \sum_j \xi_j dx_j$. The form
is independent of the choice of coordinates and is called the
action form. \bigskip

It is invariantly defined as follows: At a covector $\xi \in
T_x^*M$ define $\alpha_{x, \xi}(X) = \xi(\pi_* X)$ where $X$ is a
tangent vector to $T^*M$ at $(x, \xi)$ and $\pi: T^*M \to M$ is
the projection $(x, \xi) \to x.$ Another way to define $\alpha$ is
that for any 1-form $\eta$, viewed as a section of $T*M$, $\eta^*
\alpha = \eta. $ \bigskip

By definition $\omega = d \alpha = \sum_{j = 1}^n dx_j \wedge
d\xi_j.$

A symplectic form is a non-degenerate closed 2-form. Thus it is an
anti-symmetric bilinear form on tangent vectors at each point $(x,
\xi) \in T^*M$.  We note that
\begin{itemize}

\item $\omega(\frac{\partial}{\partial x_j},
\frac{\partial}{\partial x_k}) = 0$; one says that the vector
space Span$\{\frac{\partial}{\partial x_j}\}$ at each $(x, \xi)
\in T^*M$ is Lagrangian;

\item  $\omega(\frac{\partial}{\partial x_j},
\frac{\partial}{\partial \xi_k}) = \delta_{j k}$; one says that
$\frac{\partial}{\partial x_j}, \frac{\partial}{\partial \xi_j}$
are symplectically paired;

\item  $\omega(\frac{\partial}{\partial \xi_j},
\frac{\partial}{\partial \xi_k}) = 0$.  Thus, the vector space
Span$\{\frac{\partial}{\partial \xi_j}\}$ is Lagrangian.

\end{itemize}

A Hamiltonian is a smooth function $H(x, \xi)$ on $T^*M$. We say
it is homogeneous of degree $p$ if $H(x, r \xi) = r^p H(x, \xi)$
for $r > 0$. \bigskip

The Hamiltonian vector field $\Xi_H$ of $H$ is the symplectic
gradient of $H$. That is, one takes $dH$, a 1-form on $T^*M$ and
uses the symplectic form $\omega$ to convert it to a vector field.
That is,
$$\omega(\Xi_H, \cdot) = d H. $$
We claim:
$$\Xi_H = \sum_{j = 1}^n (\frac{\partial H}{\partial \xi_j}
\frac{\partial}{\partial x_j} - \frac{\partial H}{\partial x_j}
\frac{\partial}{\partial \xi_j}).$$

We note that $dH = \frac{\partial H}{\partial \xi_j} d \xi_j  +
\frac{\partial H}{\partial x_j} dx_j$, so the equation
$\omega(\Xi_H, \cdot) = d H $ is equivalent to:
\begin{itemize}


\item $\omega(\frac{\partial}{\partial \xi_j}, \cdot ) = -  dx_j;$

\end{itemize}

The flow of the Hamiltonian vector field is the one-parameter
group $$\Phi^t: T^*M \to T^*M $$ defined by
$$\Phi^t(x_0, \xi_0) = (x_t, \xi_t)$$
where $(x_t, \xi_t)$ solve the ordinary differential equation:
$$\left\{ \begin{array}{l} \frac{dx_j}{dt} = \frac{\partial
H}{\partial \xi_j}; \\ \\
\frac{d \xi_j}{dt} = - \frac{\partial H}{\partial x_j}.
\end{array} \right. ,  $$
with initial conditions $x(0) = x_0, \xi(0) = \xi_0.$

The symplectic form induces the following Lie bracket on
functions:
$$\{f, g\}(x) = \Xi_f (g) = \sum_{j = 1}^n (\frac{\partial f}{\partial \xi_j}
\frac{\partial g}{\partial x_j} - \frac{\partial f}{\partial x_j}
\frac{\partial g}{\partial \xi_j}).$$ We see that $\{f, g\} = -
\{g, f\}.$

Consider the coordinate functions $x_j, \xi_k$ on $T^*M$.
Exercise: Show:
\begin{itemize}

\item $\{x_j, x_k \} = 0$;

\item $\{x_j, \xi_k \} = \delta_{jk}$;

\item $\{\xi_j, \xi_k \} = 0. $

\end{itemize}

These are the ``canonical commutation relations".

\

Classical  phase space =  cotangent bundle $T^*M$ of $M$, equipped
with its canonical symplectic form $\sum_i dx_i \wedge d\xi_i$.
The metric defines the Hamiltonian $H(x,\xi) = |\xi|_g =
\sqrt{\sum_{ij = 1}^n g^{ij}(x) \xi_i \xi_j}$ on $T^*M$, where
$g_{ij} = g(\frac{\partial}{\partial x_i},\frac{\partial}{\partial
x_j}) $, $[g^{ij}]$ is the inverse matrix to $[g_{ij}]$.
Hamilton's equations:
$$\left\{\begin{array}{l} \frac{dx_j}{dt} = \frac{\partial
H}{\partial \xi_j} \\ \\
\frac{d\xi_j}{dt} = - \frac{\partial H}{\partial x_j}. \end{array}
\right. $$

Its flow is the `geodesic flow'

$$G^t: S^*_g M \to S^*_g M $$  restricted to the energy surface $\{H = 1\} : = S^*_g M$.












\section{\label{WAVEAPP} Appendix: Wave equation and Hadamard parametrix}

The Cauchy problem for the wave equation on $\R \times M$ ($\dim M = n$) is the initial value problem
(with Cauchy data $f, g$ ) 
$$\left\{ \begin{array}{l} \Box u(t, x) = 0, \\ \\
u(0, x) = f, \;\; \frac{\partial}{\partial t} u(0, x) = g(x), \end{array} \right.. $$

The solution operator of the Cauchy problem (the ``propagator") is the wave group,
$$\ucal(t) = \begin{pmatrix} \cos t \sqrt{\Delta} & \frac{\sin t \sqrt{\Delta}}{\sqrt{\Delta}} \\ & \\
\sqrt{\Delta} \sin t \sqrt{\Delta}  & \cos t \sqrt{\Delta} \end{pmatrix}. $$
The solution of the Cauchy problem with data $(f, g)$ is $\ucal(t) \begin{pmatrix} f  \\ g \end{pmatrix}. $

\begin{itemize}

\item Even part
$\cos t\sqrt{\Delta}$ which solves the initial value
problem
\begin{equation} \left\{ \begin{array}{ll} (\frac{\partial}{\partial t}^2 - \Delta) u = 0& \\
u|_{t=0} = f & \frac{\partial}{\partial t} u |_{t=0} = 0
\end{array}\right .\end{equation} 

\item  Odd part $\frac{\sin t\sqrt{\Delta}}{\sqrt{\Delta}}$ is the
operator solving
\begin{equation} \left\{ \begin{array}{ll} (\frac{\partial}{\partial t}^2 - \Delta) u = 0& \\
u|_{t=0} = 0 & \frac{\partial}{\partial t} u |_{t=0} = g
\end{array}\right .\end{equation} 

\end{itemize}

The forward  half-wave group is the solution operator of the Cauchy problem
$$(\frac{1}{i } \frac{\partial}{\partial t} - \sqrt{-\Delta} ) u = 0, \;\;\;  u(0, x) = u_0. $$

The solution is given by 
$$u(t, x) = U(t) u_0(x), $$
with $$U(t) = e^{i t \sqrt{-\Delta}}$$
the unitary group  on $L^2(M)$ generated by the self-adjoint elliptic operator $\sqrt{-\Delta}$.

A fundamental solution of the wave equation is a solution of 
$$\Box E(t, x, y) = \delta_0(t) \delta_x(y). $$
The right side is the Schwartz kernel of the identity operator on $\R \times M$. 

There exists a unique fundamental solution with support in the forward light cone, called
the advanced (or forward) propagator. It is given by
$$E_+(t) = H(t) \frac{\sin t \sqrt{\Delta}}{\sqrt{\Delta}}, $$
where $H(t) = {\bf 1}_{t \geq 0}$ is the Heaviside step function.

\subsection{Hormander parametrix}

We would like to construct a parametrix of the form
$$\int_{T^*_x M} e^{i \langle \exp_y^{-1} x, \eta \rangle} e^{i t |\eta|_y}
A(t, x, y, \eta)d \eta.$$
This is a homogeneous Fourier integral operator kernel (see \S \ref{LAGAPP}).

H\"ormander actually constructs one of the form
$$\int_{T^*_x M} e^{i \psi(x, y, \eta)} e^{i t |\eta|}
A(t, x, y, \eta)d \eta,$$ where $\psi$ solves the Hamilton Jacobi
Cauchy problem,
$$\left\{ \begin{array}{l} q(x, d_x \psi(x, y, \eta)) = q(y, \eta),
\\ \\
\psi(x, y, \eta)  = 0 \iff \langle x - y, \eta \rangle = 0, \\ \\
d_x \psi(x, y, \eta)  = \eta, \;\;(\mbox{for} \; x = y \end{array}
\right. $$

The question is whether $\langle \exp_y^{-1} x, \eta \rangle$
solves the equations for $\psi$. Only the first one is unclear. We
need to understand $\nabla_x \langle \exp_y^{-1} x, \eta \rangle.
$ We are only interested in the norm of the gradient at $x$ but it
is useful to consider the entire expression. If we write $\eta =
\rho \omega$ with $|\omega|_y = 1$, then $\rho$ can be eliminated
from the equation by homogeneity. We fix $(y, \eta) \in S^*_y M$
and consider $\exp_y: T_y M \to M$.  We wish to vary $\exp_y^{-1}
x(t)$ along a curve. Now the level sets of $\langle \exp_y^{-1} x,
\eta \rangle$ define a notion of  local `plane waves' of $(M, g)$
near $y$. They are actual hyperplanes normal to $\omega$ in flat
$\R^n$ and in any case are far different from distance spheres.
Having fixed $(y, \eta)$, $\nabla_x \langle \exp_y^{-1} x, \omega
\rangle$ are  normal to the plane waves defined by $(y, \eta)$. To
determine the length we need to see how $\nabla_x \langle
\exp_y^{-1} x, \omega \rangle$  changes in directions normal to
plane waves.

The level sets of $\langle \exp_y^{-1} x, \eta \rangle$ are images
under $\exp_y$ of level sets of $\langle \xi, \eta \rangle = C$ in
$T_y M$. These are parallel hyperplanes normal to $\eta$. The
radial geodesic in the direction $\eta$ is of course normal to the
exponential image of the hyperplanes. Hence, this radial geodesic
is parallel to $\langle \exp_y^{-1} x, \eta \rangle$ when $\exp_y
t \eta = x$. It follows that $|\nabla_x \langle \exp_y^{-1} x,
\eta \rangle$ at this point equals $\frac{\partial}{\partial t}
\langle \exp_y^{-1} \exp_y t \frac{\eta}{|\eta|}, \eta \rangle = t
|\eta|_y.$ Hence $|\nabla_x \langle \exp_y^{-1} x, \eta \rangle|_x
= 1$ at such points.

\subsection{\label{HADASECT} Wave group: $r^2 - t^2$}


We now review the construction of a Hadamard parametrix,
\begin{equation} U(t) (x,y) = \int_0^{\infty} e^{ i \theta (r^2(x,y) - t^2)} \sum_{k =
0}^{\infty} W_k (x,y) \theta^{\frac{n-3}{2} - k}
d\theta\;\;\;\;\;\;\;\;\;(t < {\rm inj}(M,g)) \end{equation} where
$U_o(x,y) = \Theta^{-\half}(x,y)$ is the volume 1/2-density, where
the higher coefficients are determined by transport equations, and
where  $\theta^r$ is regularized at $0$ (see below). This formula
is only valid for times $t < inj(M,g)$ but using the group
property of $U(t)$ it determines the wave kernel for all times. It
shows that for fixed $(x,t)$ the kernel $U(t)(x,y)$ is singular
along the distance sphere $S_t(x)$ of radius $t$ centered at $x$,
with singularities propagating along geodesics. It only represents
the singularity and in the analytic case only converges in a
neighborhood of the characteristic conoid.

Closely related but somewhat simpler is the even part of the wave
kernel, $\cos t\sqrt{\Delta}$ which solves the initial value
problem
\begin{equation} \left\{ \begin{array}{ll} (\frac{\partial}{\partial t}^2 - \Delta) u = 0& \\
u|_{t=0} = f & \frac{\partial}{\partial t} u |_{t=0} = 0
\end{array}\right .\end{equation} Similar, the odd part of the
wave kernel, $\frac{\sin t\sqrt{\Delta}}{\sqrt{\Delta}}$ is the
operator solving
\begin{equation} \left\{ \begin{array}{ll} (\frac{\partial}{\partial t}^2 - \Delta) u = 0& \\
u|_{t=0} = 0 & \frac{\partial}{\partial t} u |_{t=0} = g
\end{array}\right .\end{equation} These kernels only really
involve $\Delta$ and may be constructed by the Hadamard-Riesz
parametrix method. As above they have the form
\begin{equation}  \int_{0}^{\infty} e^{i \theta (r^2-t^2)}
\sum_{j=0}^{\infty} W_j(x,y) \theta_{reg}^{\frac{n-1}{2} - j}
d\theta
 \;\;\;\mbox{mod}\;\;C^{\infty}  \end{equation}
where $W_j$ are the Hadamard-Riesz coefficients determined
inductively by the transport equations
\begin{equation}\begin{array}{l}
 \frac{\Theta'}{2 \Theta} W_0 + \frac{\partial W_0}{\partial r} = 0\\ \\
4 i r(x,y) \{(\frac{k+1}{r(x,y)} +  \frac{\Theta'}{2 \Theta})
W_{k+1} + \frac{\partial W_{k + 1}}{\partial r}\} = \Delta_y W_k.
\end{array}\end{equation} The solutions are given by:
\begin{equation}\label{HD} \begin{array}{l} W_0(x,y) = \Theta^{-\half}(x,y) \\ \\
W_{j+1}(x,y) =  \Theta^{-\half}(x,y) \int_0^1 s^k \Theta(x,
x_s)^{\half} \Delta_2 W_j(x, x_s) ds
\end{array} \end{equation}
where $x_s$ is the geodesic from $x$ to $y$ parametrized
proportionately to arc-length and where $\Delta_2$ operates in the
second variable.

According to \cite{GS},  page 171,
$$\int_0^{\infty} e^{i \theta \sigma} \theta_+^{\lambda} d \lambda
= i e^{i \lambda \pi/2} \Gamma(\lambda + 1) (\sigma + i
0)^{-\lambda - 1}.
$$

One has,
\begin{equation}  \int_{0}^{\infty} e^{i \theta (r^2-t^2)}
 \theta_+^{\frac{d-3}{2} - j}
d\theta = i e^{i (\frac{n-1}{2} - j) \pi/2} \Gamma(\frac{n-3}{2} -
j + 1) (r^2-t^2 + i0)^{j-\frac{n -3}{2} - 2}
\end{equation}
Here there is a problem  when $n$ is odd since $
\Gamma(\frac{n-3}{2} - j + 1)$ has poles at the negative integers.

One then uses
$$\Gamma(\alpha + 1 - k) = (-1)^{k+1} (-1)^{[\alpha]}
\frac{\Gamma(\alpha +1 - [\alpha]) \Gamma ([\alpha] + 1 -
\alpha)}{\alpha + 1} \frac{1}{\alpha - [\alpha]} \frac{1}{\Gamma(k
- \alpha)}. $$

We note that
$$\Gamma(z) \Gamma(1 - z) = \frac{\pi}{\sin \pi z}. $$

Here and above $t^{-n}$ is the distribution defined by $t^{-n} =
Re (t + i0)^{-n}$ (see \cite{Be}, [G.Sh., p.52,60].)  We recall
that $(t+i0)^{-n} =  e^{-i\pi \frac{n}{2}}\frac{1}{\Gamma(n)}
\int_0^{\infty} e^{itx} x^{n-1} dx$.

We also need that $(x + i 0)^{\lambda}$ is entire and
$$(x + i 0)^{\lambda} = \left\{ \begin{array}{ll} e^{i \pi
\lambda} |x|^{\lambda}, & x < 0 \\ & \\
x_+^{\lambda}, & x > 0. \end{array} \right.$$

The imaginary part cancels the singularity of $\frac{1}{\alpha -
[\alpha]}$ as $\alpha \to \frac{d-3}{2}$ when $n = 2m + 1$. There
is no singularity in even dimensions. In odd dimensions the real
part is $\cos \pi \lambda x_-^{\lambda} + x_+^{\lambda}$ and we
always seem to have a pole in each term!

But in any dimension, the imaginary part is well-defined and we
have
\begin{equation} \frac{\sin t\sqrt{\Delta}}{\sqrt{\Delta}}(x,y) =
 C_o sgn(t) \sum_{j=0}^{\infty}(-1)^j w_j(x,y)\frac{
(r^2-t^2)_{-}^{j-\frac{n - 3}{2} - 1}}{4^j \Gamma(j -
\frac{n-3}{2})}
 \;\;\;\mbox{mod}\;\;C^{\infty} \end{equation}

By taking the time derivative we also have,
\begin{equation} \cos t\sqrt{\Delta}(x,y) =
 C_o |t| \sum_{j=0}^{\infty}(-1)^j w_j(x,y)\frac{
(r^2-t^2)_{-}^{j-\frac{n -3}{2} - 2}}{4^j \Gamma(j - \frac{n-3}{2}
- 1)}
 \;\;\;\mbox{mod}\;\;C^{\infty} \end{equation}
where $C_o$ is a universal constant and where $W_j=\tilde{C}_o
e^{-ij\frac{\pi}{2}} 4^{-j}w_j(x,y),$.

\subsection{Exact formula in spaces of constant curvature}

The Poisson kernel of $\R^{n+1}$ is the kernel of $e^{- t
\sqrt{\Delta}}$, given by
$$\begin{array}{lll} K(t, x, y)& =& t^{-n}  (1 + |\frac{x -
y}{t}|^2)^{-\frac{n+1}{2}}\\ & & \\
& = & t \;   (t^2 + |x - y|^2)^{-\frac{n+1}{2}}.
\end{array}$$
It is defined only for $t > 0$, although formally it appears to be
odd.

Thus, the kernel of $e^{i t \sqrt{\Delta}}$ is
$$\begin{array}{lll} U(t, x, y)
& = & (i t)  \;   ( |x - y|^2 - t^2 )^{-\frac{n+1}{2}}.
\end{array}$$

One would conjecture that the Poisson kernel of any Riemannian
manifold would have the form \begin{equation} K(t, x, y) =  t \;
\sum_{j = 0}^{\infty} (t^2 + r(x, y)^2)^{-\frac{n+1}{2} + j}
U_j(x, y)
\end{equation}
for suitable $U_j$.

\subsection{$\Ss^n$}

One can determine the kernel of $e^{i t \sqrt{\Delta}}$ on $\Ss^n$
from the Poisson kernel of the unit ball $B \subset \R^{n+1}$. We
recall that the Poisson integral formula for the unit ball is:
$$u(x) = C_n \int_{\Ss^n} \frac{1 - |x|^2}{|x - \omega'|^2}
f(\omega') dA(\omega'). $$

Write $x = r \omega$ with $|\omega| = 1$ to get:
$$P(r, \omega, \omega') = \frac{1 - r^2}{(1 - 2 r \langle \omega,
\omega' \rangle + r^2)^{\frac{n+1}{2}}}. $$

A second formula for $u(r \omega)$ is
$$u(r, \omega) = r^{A - \frac{n-1}{2}} f(\omega),$$
where $A = \sqrt{\Delta + (\frac{n-1}{4})^2}. $ This follows from
by writing the equation $\Delta_{\R^{n+1}} u = 0$ as an Euler
equation: $$\{r^2 \frac{\partial^2}{\partial r^2}  + n
r\frac{\partial}{\partial r }- \Delta_{\Ss^n}\} u = 0. $$

Therefore, the Poisson operator $e^{- t A}$  with  $r = e^{-t}$ is
given by
$$\begin{array}{lll} P(t, \omega, \omega') &= &  C_n \frac{\sinh t }{(\cosh t - \cos
r(\omega, \omega'))^{\frac{n+1}{2}}} \\ & & \\
& = &  C_n \frac{\partial}{\partial t}  \frac{1}{(\cosh t - \cos
r(\omega, \omega'))^{\frac{n-1}{2}}}. \end{array}$$

Here, $r(\omega, \omega')$ is the distance between points of
$\Ss^n$.

We analytically continue the expressions to $t > 0$ and obtain the
wave kernel as a boundary value:

$$\begin{array}{l} e^{i t A} =
 \lim_{\epsilon \to 0^+} C_n i \sin t (\cosh \epsilon
\cos t -  i \sinh \epsilon \sin t - \cos r(\omega,
\omega')^{-\frac{n +1}{2}}\\ \\
= \lim_{\epsilon \to 0^+} C_n i   \sinh (it - \epsilon) (\cosh (i
t - \epsilon ) - \cos r(\omega, \omega')^{-\frac{n+1}{2}} .
\end{array} $$

If we formally put $\epsilon = 0$ we obtain:

$$\begin{array}{l} e^{i t A} =
 C_n i  \sin t (
\cos t -  \cos r(\omega, \omega'))^{-\frac{n+1}{2}}. \end{array}
$$

This expression is  singular when $\cos t = \cos r$. We note that
$r \in [0, \pi]$ and that it is singular on the cut locus $r =
\pi$. Also,  $\cos: [0, \pi] \to [-1, 1]$ is decreasing, so the
wave kernel  is singular when $t = \pm r$ if $t \in [- \pi, \pi]$.

When $n$ is even, the expression appears to be pure imaginary but
that is because we need to regularize it on the set $t = \pm r$.
When $n$ is odd, the square root is real if $\cos t \geq \cos r$
and pure imaginary if $\cos t < \cos r. $

We see that the kernels of $\cos t A, \frac{\sin t A}{A}$ are
supported inside the light cone $|r| \leq |t|.$ On the other hand,
$e^{i t A}$ has no such support property (it has infinite
propagation speed). On odd dimensional spheres, the kernels are
supported on the distance sphere (sharp Huyghens phenomenon).

The Poisson kernel of the unit sphere is then

$$\begin{array}{l} e^{- t A} =
 C_n   \sinh t (
\cosh t -  \cos r(\omega, \omega'))^{-\frac{n+1}{2}}. \end{array}
$$

It is singular on the complex characteristic conoid  when  $\cosh
t - \cos r(\zeta, \bar{\zeta}') = 0$.

\subsection{Analytic continuation into the complex}

If we write out the eigenfunction expansions of $\cos t
\sqrt{\Delta}(x, y)$ and $\frac{\sin t
\sqrt{\Delta}}{\sqrt{\Delta}}(x,y)$ for $t = i \tau$, we would not
expect convergence since the eigenvalues are now exponentially
growing. Yet the majorants argument seems to indicate that these
wave kernels admit an analytic continuation into a complex
neighborhood of the complex characteristic conoid. Define the
characteristic conoid in $\R \times M \times M$ by $r(x,y)^2 - t^2
= 0$. For simplicity of visualization, assume $x$ is fixed. Then
analytically the conoid to $\C \times M_{\C} \times M_{\C}$. By
definition $(\zeta, \bar{\zeta}, 2 \tau$ lies on the complexified
conoid. That is, the series also converge after analytic
continuation, again if $r^2 - t^2$ is small. If $t = i \tau$ then
we need $r(\zeta, y)^2 + \tau^2$ to be small, which either forces
$r(\zeta, y)^2$ to be negative and close to $\tau$ or else forces
both $\tau$ and $r(\zeta, y)$ to be small.

If we wish to use orthogonality relations on $M$ to sift out
complexifications of eigenfunctions, then we need $U(i \tau,
\zeta, y)$ to be holomorphic in $\zeta$ no matter how far it is
from $y$!. So far, we do not have a proof that $U(i \tau, \zeta,
y)$ is globally holomorphic for $\zeta \in M_{\tau}$ for every
$y$.

Regimes of analytic continuation. Let $E(t, x, y)$ be any of the
above kernels. Then analytically continue to $E(i \tau, \zeta,
\bar{\zeta'})$ where $r(\zeta, \zeta')^2 + \tau^2$ is small. For
instance if $\zeta' = \zeta$ and $\sqrt{\rho}(\zeta) =
\frac{\tau}{2}$, then $r(\zeta, \zeta')^2 + \tau^2 = 0$.

Is there a neighborhood of the characteristic conoid into which
the analytic continuation is possible? We need to have
$$r^2(\zeta, \zeta') \pm \tau^2 << \epsilon. $$
If we analytically continue in $\zeta$ and anti-analytically
continue in $\zeta'$, we seem to get a neighborhood of the conoid.

We would like to analytically continue the Hadamard parametrix to
a small neighborhood of the characteristic conoid. It is singular
on the conoid.

\section{\label{LAGAPP} Appendix:  Lagrangian distributions,  quasi-modes  and Fourier integral operators}

In this section, we go  over the definitions of Lagrangian distributions, both semi-classical and
homogeneous. A very detailed treatment of the homogeneous Lagrangian distributions (and Fourier
integral operators) can be found in \cite{HoIV}. The semi-classical case is almost the same and
the detailed treatments  can be found in \cite{D,DSj,GSj, CV2,Zw}. We also continue the discussion
in \S \ref{QMintro} of quasi-modes. 

\subsection{Semi-classical Lagrangian distributions and Fourier integral operators}

Semi-classical
Fourier integral operators with large parameter $\lambda = \frac{1}{\hbar}$  are operators whose Schwartz kernels are defined by semi-classical
Lagrangian distributions,
$$I_{\lambda}(x, y) = \int_{\R^N} e^{i \lambda \phi (x, y, \theta) }
a(\lambda, x, y, \theta) d \theta. $$
More generally, semi-classical Lagrangian distributions are defined by oscillatory integrals (see \cite{D}),
\begin{equation} \label{uk} u(x, \hbar) = \hbar^{- N/2} \int_{\R^N} e^{\frac{i}{\hbar} \phi(x,
\theta)} a(x, \theta, \hbar) d \theta. \end{equation} We assume that $a(x,
\theta, \hbar)$ is a semi-classical symbol,
$$a(x, \theta, \hbar) \sim \sum_{k = 0}^{\infty} \hbar^{\mu + k}
a_k(x, \theta). $$

The critical set of the phase is given by
$$C_{\phi} = \{(x, \theta): d_{\theta} \phi = 0\}. $$
 The phase is called non-degenerate if
 $$d (\frac{\partial \phi}{\partial \theta_1}), \dots, d (\frac{\partial \phi}{\partial
 \theta_N})$$
 are independent on $C_{\phi}$. Thus, the map
 $$\phi'_{\theta} : = \left( (\frac{\partial \phi}{\partial \theta_1}), \dots,  (\frac{\partial \phi}{\partial
 \theta_N}) \right): X \times \R^N \to \R^N $$
 is locally a submersion near $0$ and $(\phi'_{\theta})^{-1}(0)$ is
 a manifold of codimension $N$ whose tangent space is $\ker D
 \phi'_{\theta}$.
Then $$T_{(x_0, \theta_0)} C_{\phi} = \ker d_{x, \theta} d_{\theta}
\phi. $$

We write a tangent vector to $M \times \R^N$ as $(\delta_x,
\delta_{\theta})$.
  The kernel of
$$  D \phi'_{\theta} = \begin{pmatrix}
\phi''_{\theta x} & \phi''_{\theta \theta} \end{pmatrix}$$ is
$T_{(x, \theta)} C_{\phi}$. I.e. $(\delta_x, \delta_{\theta}) \in
TC_{\phi}$ if and only if $\phi''_{\theta x} \delta_x +
\phi''_{\theta \theta} \delta_{\theta} = 0$. Indeed,
$\phi'_{\theta} $ is the defining function of $C_{\phi}$ and $d
\phi_{\theta}$ is the defining function of $T C_{\phi}$. From
\cite{HoIV} Definition 21.2.5: The number of linearly
independent differentials $d \frac{\partial \phi}{\partial
\theta}$ at a point of $C_{\phi}$  is $N - e$ where $e$ is the
excess. Then $C \to \Lambda$ is locally a fibration with fibers of
dimension $e$. So to find the excess we need to compute  the rank
of $\begin{pmatrix} \phi''_{x \theta} & \phi''_{\theta \theta}
\end{pmatrix}$ on  $T_{x,\theta} (\R^N \times M)$.

Non-degeneracy is thus the condition that
$$\begin{pmatrix}
\phi''_{\theta x} & \phi''_{\theta \theta} \end{pmatrix} \;\;
\mbox{is surjective on } \;\; C_{\phi} \iff
\begin{pmatrix}
\phi''_{\theta x} \\ \\  \phi''_{\theta \theta} \end{pmatrix} \;\;
\mbox{is injective  on } \;\; C_{\phi}. $$
 If $\phi$ is non-degenerate, then
$\iota_{\phi}(x, \theta) = (x, \phi'_{x}(x, \theta))$ is an
immersion from $C_{\phi} \to T^* X$. Note that  $$d
\iota_{\phi}(\delta_x, \delta_{\theta}) = (\delta_x, \phi''_{xx}
\delta_x + \phi''_{x \theta} \delta_{\theta}). $$ 
So if
$\begin{pmatrix} \phi''_{\theta x} \\ \\  \phi''_{\theta \theta}
\end{pmatrix}$ is injective, then $\delta_{\theta} = 0$.

 If $(\lambda_1, \dots,
\lambda_n)$ are any local coordinates on $C_{\phi}$, extended as
smooth functions in  neighborhood, 
the delta-function on $C_{\phi}$ is defined by
$$ d_{C_{\phi}}: = \frac{|d \lambda|}{|D(\lambda,
\phi_{\theta}')/D(x, \theta)|} = \frac{dvol_{T_{x_0} M} \otimes dvol_{\R^N}}{d
\frac{\partial \phi}{\partial \theta_1} \wedge \cdots \wedge d
\frac{\partial \phi}{\partial \theta_N} } $$ where the denominator
can be regarded as the pullback of $dVol_{\R^N}$ under the map
$$d_{\theta, x} d_{\theta} \phi (x_0, \theta_0). $$


The symbol $\sigma(\nu)$ of a Lagrangian (Fourier integral)
distributions is a section of the bundle $\Omega_{\half} \otimes
\mcal_{\half}$ of the bundle of half-densities (tensor the Maslov
line bundle). In terms of a Fourier integral representation it is
the square root $\sqrt{d_{C_{\phi}}}$ of the delta-function on
$C_{\phi}$ defined by $\delta(d_{\theta} \phi)$, transported to
its image in $T^* M$ under $\iota_{\phi}$

\begin{defin} The principal symbol $\sigma_u(x_0, \xi_0)$ is
$$\sigma_u(x_0, \xi_0) = a_0(x_0, \xi_0) \sqrt{d_{C_{\phi}}}.  $$
It is a $\half$ density on $T_{(x_0, \xi_0)} \Lambda_{\phi}$ which
depends on the choice of a density
on $T_{x_0} M$). \end{defin}

\subsection{Homogeneous Fourier integral operators}
A homogeneous  Fourier integral operator $A: C^{\infty}(X) \to
C^{\infty}(Y)$ is an operator whose Schwartz kernel may be
represented by an oscillatory integral
$$K_A(x,y) = \int_{\R^N} e^{i \phi(x, y, \theta)} a(x, y, \theta) d\theta$$
where the phase $\phi$ is homogeneous of degree one in $\theta$.
We 
assume $a(x, y, \theta) $ is a zeroth order
classical polyhomogeneous symbol   with $a \sim
\sum_{j=0}^{\infty} a_j, \, a_j$ homogeneous of degree $-j$. 
We refer to \cite{DSj,GSt1} and especially to 
\cite{HoIV}  for background on Fourier integral operators. We
use the notation $I^m(X \times Y, C)$ for the class of Fourier
integral operators of order $m$ with wave front set along the
canonical relation $C$, and $WF'(F)$ to denote the canonical
relation of a Fourier integral operator $F$.

When
$$\iota_{\phi} : C_{\phi} \to \Lambda_{\phi} \subset T^*(X, Y), \;\;\; \iota_{\phi}(x, y,
\theta) = (x, d_x \phi, y, - d_y \phi) $$ is an embedding, or at
least an immersion, the phase is called
non-degenerate.  Less restrictive, although still an ideal
situation, is where the phase is clean. This means that the map
$\iota_{\phi} : C_{\phi} \to \Lambda_{\phi} $, where
$\Lambda_{\phi} $ is the image of $\iota_{\phi}$,  is locally a
fibration with fibers of dimension $e$.  From \cite{HoIV}
Definition 21.2.5, the number of linearly independent
differentials $d \frac{\partial \phi}{\partial \theta}$ at a point
of $C_{\phi}$ is $N - e$ where $e$ is the excess.

We a recall that the order of $F: L^2(X) \to L^2(Y)$ in the
non-degenerate case  is given in terms of a local oscillatory
integral formula by $m +
 \frac{N}{2} - \frac{n}{4},$, where $n = \dim X + \dim Y, $ where
$m$ is the order of the amplitude, and $N$ is the number of phase
variables in the local Fourier integral representation (see
\cite{HoIV}, Proposition 25.1.5); in the general clean case with
excess $e$, the order goes up by $\frac{e}{2}$ (\cite{HoIV},
Proposition 25.1.5'). Further, under clean composition of
operators of orders $m_1, m_2$, the order of the composition is
$m_1 + m_2 - \frac{e}{2}$ where $e$ is the so-called excess (the
fiber dimension of the composition); see \cite{HoIV}, Theorem
25.2.2.


The definition of the principal symbol is essentially the same as in the semi-classical case. 
 As discussed in \cite{SV},
(see  (2.1.2) and ((2.2.5) and Definition 2.7.1)), if an
oscillatory integral is represented as
$$\ical_{\phi, a}(t, x, y) = \int e^{i \phi(t, x,y, \eta)} a(t,x ,
y, \eta) \zeta(t, x, y, \eta) d_{\phi}(t, x, y, \eta) d \eta, $$
where
$$d_{\phi}(t, x, y, \eta) = |\det  \phi_{x, \eta}|^{\half} $$ and
where the  number of phase variables equals the number of $x$
variables, then the  phase is  non-degenerate if and only if
$(\phi_{x \eta}(t, x, y, \eta)$ is non-singular. Then $ a_0 |\det
\phi_{x, \eta}|^{- \half}$ is the symbol.

  The behavior of symbols under pushforwards and pullbacks of Lagrangian
submanifolds are described in \cite{GSt1}, Chapter IV. 5 (page 345).
The main statement (Theorem 5.1, loc. cit.) states that the symbol
map $\sigma: I^m(X, \Lambda) \to S^m(\Lambda)$ has the following
pullback-pushforward  properties under maps $f: X \to Y$
satisfying  appropriate transversality conditions,
\begin{equation} \label{PFPB} \left\{\begin{array}{l} \sigma (f^* \nu) =
f^* \sigma(\nu), \\ \\
\sigma(f_* \mu) = f_* \sigma(\mu), \end{array} \right.
\end{equation}  Here, $f_* \sigma(\mu)$ is integration over the
fibers of $f$ when $f$ is a submersion. In order to define a
pushforward, $f$ must be a ``morphism'' in the language of
\cite{GSt2}, i.e. must be accompanied by a map $r(x):
|\bigwedge|^{\half} TY_{f(x)} \to |\bigwedge|^{\half} TX_{x}$, or
equivalently a half-density on $N^*(\mbox{graph}(f))$, the
co-normal bundle to the graph of $f$ which is constant long the
fibers of $N^*(\mbox{graph}(f)) \to \mbox{graph}(f))$.

\subsection{\label{QM} Quasi-modes}

We consider here Lagrangian quasi-modes of order zero, i.e. semi-classical oscillatory integrals \eqref{uk}
which solves $\Delta u_k = O(1)$. If we pass $\Delta$ under the integral sign we obtain a leading
order term whose amplitude contains the factor $k^2 |\nabla_x \phi(x, \theta)|^2. $ The integral
of this term must vanish and hence must vanish on the stationary phase set $\nabla_{\theta} \phi(x, \theta) = 0$.
Hence $\phi$ must generate a Lagrangian submanifold $\Lambda_{\phi} \subset S^* M$. Such a Lagrangian
submanifold must be invariant under the geodesic flow $G^t: S^* M \to S^*M$ because the Hamilton vector
field $\Xi_H$ of a function $H$ which is constant on  a Lagrangian submanifold is tangent to $\Lambda$.
The Lagrangian submanifold may only be locally defined and its global extension might be dense in $S^*M$. 
The amplitude $a$ must then be assumed to be compactly supported and the distribution is $O(k^{- M})$
for all $M$ outside of its support. 

The $O(k)$ term has an invariant interpretation as $\lcal_{\Xi_H} a_0 \sqrt{d_{C_{\phi}}}$, the Lie derivative
of the principal symbol, which is a half-density on $\Lambda_{\phi}$. To define a quasi-mode of order zero,
the principal symbol must define a global half-density invariant under the geodesic flow. Here we suppress
the role of the Maslov bundle and refer to \cite{D,DSj,Zw} for its definition. 

To obtain a quasi-mode of higher order, one must solve recursively a sequence of transport equations,
which are homogeneous equations of the form  $\lcal_{\Xi_H}  a_{j + 1} = \dcal_j a_j$ for various
operators $\dcal_k$. In general there are obstructions to solving the inhomogeneous equations. The integral
of the left side over $\Lambda$ with respect to the invariant density equals zero, and therefore so must
the right side. Here we express all half-densities in terms of the invariant one. 

In sum, a zeroth order quasi-mode is an oscillatory integral \eqref{uk} quantizing a pair $(\Lambda, \sigma)$
where $\Lambda \subset S^* M$ is a closed invariant Lagrangian submanifold and $\sigma$ is a $G^t$-invariant
half-density along it. 

There are additional quasi-modes associated to isotropic submanifolds of $S^*M$, i.e. manifolds on which
the symplectic form vanishes but which have dimension $< \dim M$. Gaussian beams are of this kind. 
We refer to \cite{BB,R1,R2} for their definition. They are constructed along stable elliptic  closed geodesics,
and may be regarded as Lagrangian distributions with complex phase. 

In the case of toric completely integrable systems, the joint eigenfunctions are automatically semi-classical Lagrangian
distributions. Indeed, they may be expressed as Fourier coefficients of the unitary Fourier integral operator
quantizing the torus action. It is also true that joint eigenfunctions in the general case of quantum integrable
systems are Lagrangian. We refer to \cite{TZ3} for background and references.

We should emphasize that it is very rare that eigenfunctions of the Laplacian are quasi-modes or that
quasi-modes are genuine eigenfunctions of the Laplacian.  In \S \ref{SHAPPQM} we will see that
the standard eigenfunctions on $S^2$ are Lagrangian quasi-modes.

\section{\label{SHAPP}Appendix on Spherical Harmonics}

Spherical harmonics furnish the extremals for $L^p$ norms of eigenfunctions $\phi_{\lambda}$ as $(M,g)$
ranges over Riemannian manifolds and $\phi_{\lambda}$ ranges over its eigenfunctions.  They are not unique in this respect: surfaces of revolution
and their higher dimensional analogues also give examples where extremal eigenfunction 
bounds are achieved. In this appendix
we review the definition and properties of spherical harmonics.

Eigenfunctions of the Laplacian  $\Delta_{S^n}$ on the standard
sphere  $S^n$  are restrictions of harmonic homogeneous
polynomials on $\R^{n+1}$.

Let $\Delta_{\R^{n+1}} = - ( \frac{\partial^2}{\partial x_1^2} +
\cdots + \frac{\partial^2}{\partial x_{n+1}^2})$ denote the
Euclidean Laplacian. In  polar coordinates $(r, \omega)$ on
$\R^{n+1}$, we have  $\Delta_{\R^{n+1}} = - \left( \frac{\partial^2}{\partial r^2} +
\frac{n}{r} \frac{\partial}{\partial r}\right) + \frac{1}{r^2}
\Delta_{S^{n}}. $ A polynomial $P(x) = P(x_1, \dots, x_{n+1})$ on
$\R^{n+1}$  is called:

\begin{itemize}

\item  homogeneous of degree $k$ if $P(r x ) = r^k P(x).$ We
denote the space of such polynomials by $\pcal_k$. A basis is
given by the monomials $$x^{\alpha} = x_1^{\alpha_1} \cdots
x_{n+1}^{\alpha_{n+1}}, \;\;\; |\alpha| = \alpha_1 + \cdots +
\alpha_{n+1} = k.$$

\item Harmonic if $\Delta_{\R^{n+1}} P(x) = 0.$ We denote the
space of harmonic homogeneous polynomials of degree $k$ by
$\hcal_k$.

\end{itemize}

Suppose that $P(x)$ is a homogeneous harmonic polynomial of degree
$k$ on $\R^{n+1}$. Then,
$$\begin{array}{l} 0 = \Delta_{\R^{n+1}} P  = - \{\frac{\partial^2}{\partial r^2} + \frac{n}{r}
\frac{\partial}{\partial r} \}  r^k P(\omega)   + \frac{1}{r^2}
\Delta_{S^{n}} P (\omega)  \\ \\
\implies \Delta_{S^{n}} P (\omega) = (k (k - 1) + n k) P(\omega).
\end{array}
$$
Thus, if we restrict $P(x)$ to the unit sphere $S^n$ we obtain an
eigenfunction of eigenvalue $k (n + k - 1). $
 Let $\hcal_k \subset L^2(S^n)$ denote the space of spherical harmonics
of degree $k$.  Then:
\begin{itemize}

\item $L^2(S^n) = \bigoplus_{k = 0}^{\infty} \hcal_k$. The sum is
orthogonal.

\item $Sp(\Delta_{S^n}) = \{ \lambda_k^2 = k ( n + k - 1) \}. $

\item $\dim \hcal_k$ is given by $$d_k = { n + k - 1  \choose k}
  - { n + k - 3 \choose k -
2}$$

\end{itemize}

The Laplacian $\Delta_{S^n}$ is quantum integrable. For
simplicity, we restrict to $S^2$. Then the  group $SO(2) \subset
SO(3)$ of rotations around the $x_3$-axis commutes with the
Laplacian. We denote its infinitesimal generator by $L_3 =
\frac{\partial }{i
\partial \theta}$. The standard basis of spherical harmonics
is given by the joint eigenfunctions $(|m| \leq k$)
$$\left\{ \begin{array}{l} \Delta_{S^2} Y^k_m = k ( k + 1) Y^k_m; \\ \\
 \frac{\partial }{i \partial \theta} Y^k_m  = m Y^k_m. \end{array} \right.$$

Two basic spherical harmonics are:

\begin{itemize}

\item The highest weight spherical harmonic $Y^k_k$. As a
homogeneous polynomial it is given up to a normalizing constant by
$ (x_1 + i x_2)^k$ in $\R^3$ with coordinates $(x_1, x_2, x_3)$.
It is a  `Gaussian beam' along the equator $\{x_3 = 0\}$, and is
also a quasi-mode associated to this stable elliptic orbit. 
\medskip

\item The zonal spherical harmonic $Y^k_0$. It may be expressed in
terms of the orthogonal projection $\Pi_k: L^2(S^2) \to \hcal_k. $

\end{itemize}

\begin{center}
\includegraphics[scale=0.9]{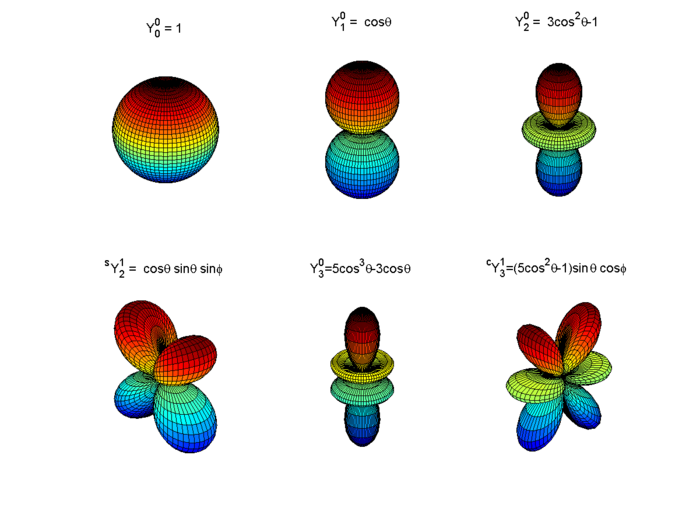}
\end{center}

\begin{center}
\includegraphics[scale=0.9]{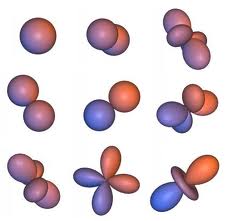}
\end{center}

We now explain the last statement:  For any $n$, the  kernel  $\Pi_k(x, y)$ of
$\Pi_k$ is defined by
$$\Pi_k f(x) = \int_{S^n} \Pi_k(x, y) f(y) dS(y), $$
where $dS$ is the standard surface measure.  If $\{Y^k_m\}$ is an
orthonormal basis of $\hcal_k$ then
$$\Pi_k(x, y) = \sum_{m = 1}^{d_k} Y^k_m(x) Y^k_m(y). $$
Thus for each $y$, $\Pi_k(x, y) \in \hcal_k.$ We can $L^2$
normalize this function by dividing by the square root of
$$||\Pi_k(\cdot, y)||_{L^2}^2 = \int_{S^n} \Pi_k(x, y) \Pi_k(y, x)
dS(x) = \Pi_k(y, y). $$

We note that $\Pi_k(y, y) = C_k$ since it is rotationally
invariant and $O(n + 1)$ acts transitively on $S^n$. Its integral
is $\dim \hcal_k$, hence, $\Pi_k(y, y) = \frac{1}{Vol(S^n)} \dim
\hcal_k.$ Hence the normalized projection kernel with `peak' at
$y_0$ is
$$Y^k_0(x) = \frac{\Pi_k(x, y_0) \sqrt{Vol(S^n)}}{\sqrt{\dim
\hcal_k}}.$$ Here, we put $y_0$ equal to the north pole $(0, 0
\cdots, 1).$  The resulting function is called a zonal spherical
harmonic since it is invariant under the group $O(n + 1)$ of
rotations fixing $y_0$.

One can rotate $Y^k_0(x) $ to $Y^k_0(g \cdot x)$ with $g \in
O(n+1)$ to place the `pole' or `peak point' at any point in $S^2$.

\subsection{Highest weight spherical harmonics}

As mentioned above, the highest weight spherical harmonic is an extremal for
the $L^6$ norm and for all $L^p$ norms with $2 < p < 6$ it is the unique extremal. Let us
verify that it achieves the maximum.

We claim that $||(x + i y)^k||_{L^2(S^2)} \sim k^{-1/4}. $ Indeed
we compute it using Gaussian integrals:
$$\begin{array}{l} \int_{\R^3} (x^2 + y^2)^{k} e^{-(x^2 + y^2 +
z^2)}  = ||(x + i y)^k||_{L^2(S^2)}^2  \int_0^{\infty} r^{2k}
e^{-
r^2} r^2 dr, \\ \\
 \int_{\R^3} (x^2 + y^2)^{k} e^{-(x^2 + y^2 +
z^2)} = \int_{\R^2} (x^2 + y^2)^{k} e^{-(x^2 + y^2)} \\ \\
 =   \int_0^{\infty} r^{2k} e^{-
r^2} r dr,
\implies  ||(x + i y)^k||_{L^2(S^2)}^2 = \frac{\Gamma(k +
1)}{\Gamma(k + \frac{3}{2})} \sim k^{-1/2}. \end{array}$$

Thus,  $k^{1/4} (x + i y)^k$ is the $L^2$-normalized   highest weight
spherical harmonic. It achieves its $L^{\infty}$ norm at $(1,0, 0)$ where it has
size $k^{1/4}. $

To see that it is an extremal for $L^p$ for $2 \leq p
\leq 6$, we use Gaussian integrals:  $$\begin{array}{l} \int_{\R^3} (x^2 +
y^2)^{3k} e^{-(x^2 + y^2 + z^2)}  = ||(x + i
y)^k||_{L^6(S^2)}^6 \int_0^{\infty} r^{6k} e^{-
r^2} r^2 dr, \\ \\
 \int_{\R^3} (x^2 + y^2)^{3 k} e^{-(x^2 + y^2 +
z^2)}  = \int_{\R^2} (x^2 + y^2)^{3 k} e^{-(x^2 + y^2)} 
 =   \int_0^{\infty} r^{6k} e^{-
r^2} r dr, \\ \\
\implies  ||(x + i y)^k||_{L^6(S^2)}^6 = \frac{\Gamma(6k +
1)}{\Gamma(6k + \frac{3}{2})} \sim k^{-1/2}. \end{array}$$
Hence, the $L^6$ norm of $k^{1/4} (x + i y)^k$ equals
$$k^{1/4} k^{-1/12} = k^{1/6}. $$
Since $\lambda_k \sim k$ and $\delta(6) = \frac{1}{6}$ in
dimension $2$, we see that it is an extremal.
\bigskip

\noindent{\bf Problem:}  If a surface $(M^2, g)$ has maximal $L^p$ growth for $2 < p < 6$, must
it have a Gaussian beam? Here we may insist that the Gaussian beam be a sequence of eigenfunctions
or we may relax the definition and allow it to be a quasi-mode. We might also ask, if there exists 
a quasi-mode Gaussian beam, does $(M^2, g)$ has maximal $L^p$ growth for $2 < p < 6$?

\subsection{\label{SHAPPQM} Spherical harmonics as quasi-modes}

The normalized joint eigenfunctions on the standard sphere are given by
\begin{equation} \label{YMN} Y^N_m(\theta, \phi) = \sqrt{(2 N + 1) \frac{(N - m)!}{(N + m)!}} P^N_m(\cos \phi) e^{ im \theta}, 
\end{equation} where
$$P_m^{N} (\cos \phi) = \frac{1}{2 \pi} \int_0^{ 2 \pi} (i \sin \phi \cos \theta + \cos \phi)^N
e^{- i m \theta} d \theta $$
are the Legendre polynomials. \bigskip

To obtain a Lagrangian distribution, we consider a sequence of $Y_m^N$ with $\frac{m}{N} \to C$ for
some $C$. I.e. we consider pairs $k (m_0, N_0)$ lying on a ray in the lattice in $\Z^2$ of $(m, N)$
with $|m| \leq N$.

The Lagrangian submanifold of $T^* S^2$ associated to $Y^N_m$ with $m/N \to c$ is
the torus in $S^*S^2$ defined by 
$$p_{\theta}(x, \xi): = \langle \xi, \frac{\partial}{\partial \theta}) = c. $$
This is a level set of the Clairaut integral
$$p_{\theta} : S^*S^2 \to [-1,1].$$
\bigskip

Examples: 

\begin{itemize}

\item (i) The Lagrangian submanifold associated to the zonal spherical harmonic is the
``meridian torus'' consisting of geodesics from the north to south poles. 

\item (ii) The Gaussian beam is associated to the unit vectors along the equator-- a degenerate Lagrangian torus
of dimension 1.

\end{itemize}

Note that we also express the standard spherical harmonics as
$$Y_m^N (x) = \int_{0}^{2 \pi} \Pi_N(x, r_{\theta} y) e^{- i m \theta} d \theta, $$
where $r_{\theta}$ are rotations around the third axis.  Since $\Pi_N(x, y)$ is a semi-classical 
Lagrangian quasi-mode, this also exhibits $Y^N_m$ as one.

\end{document}